\documentclass[11pt]{article}
\usepackage[margin=1in]{geometry}
 
\usepackage[]{amsmath,amssymb,epsfig}
\usepackage{amsthm}
\usepackage{amsmath}
\usepackage{amssymb}
\usepackage{graphicx}
\usepackage{epstopdf}
\usepackage{comment}
\usepackage{array}
\usepackage{algorithm}
\usepackage{url}
\usepackage{bm}

\usepackage{lscape}
\usepackage{algpseudocode}
\usepackage{setspace}
\usepackage{multicol}
\usepackage{multirow}
\usepackage{color}
\usepackage{colortbl}
\usepackage{xcolor}
\usepackage{mathtools}

\usepackage{placeins}

\usepackage[%
    font={small,sf},
    labelfont=bf,
    format=hang,    
    format=plain,
    margin=0pt,
    width=0.8\textwidth,
]{caption}
\usepackage[list=true]{subcaption}

\usepackage{enumerate}
\usepackage{tikz}  
\usepackage{tikz-3dplot} 
\usepackage{amssymb}
\usepackage{xifthen}
\usepackage{graphicx}% http://ctan.org/pkg/graphicx
\usepackage{caption,subcaption}

\usepackage{pgfplots}
\pgfplotsset{compat=newest}
\usetikzlibrary{plotmarks}
\usetikzlibrary{arrows.meta}
\usepgfplotslibrary{patchplots}
\usepackage{grffile}
\usepackage{amsmath}

\usepackage{epstopdf}
\DeclareGraphicsExtensions{.png,.pdf}

\newtheorem{theorem}{Theorem}

\newtheorem{fact}{Fact}

\newtheorem{corollary}{Corollary}

\newtheorem{lemma}{Lemma}

\newtheorem{proposition}{Proposition}
\newtheorem{remark}{Remark}

\newtheorem{innercustomthm}{Fact}
\newenvironment{customthm}[1]
  {\renewcommand\theinnercustomthm{#1}\innercustomthm}
  {\endinnercustomthm}

\numberwithin{equation}{section}

\newcommand{\rev}[1]{\textcolor{black}{#1}}

\DeclareMathOperator{\proj}{proj}
\DeclareMathOperator{\grad}{grad}
\DeclareMathOperator{\Hess}{Hess}
\DeclareMathOperator{\Tr}{Tr}

\DeclareMathOperator{\rank}{rank}

\DeclareMathOperator{\diag}{diag}

\newcommand{\calH}{\ensuremath{\mathcal{H}}}

\newcommand{\calN}{\ensuremath{\mathcal{N}}}

\newcommand{\calX}{\ensuremath{\mathcal{X}}}

\newcommand{\calQ}{\ensuremath{\mathcal{Q}}}

\newcommand{\bbT}{\ensuremath{\mathbb{T}}}

\newcommand{\norm}[1]{\left\|{#1}\right\|}
\newcommand{\abs}[1]{\left|{#1}\right|}
\newcommand{\ceil}[1]{\lceil{#1}\rceil}

\newcommand{\set}[1]{\left\{{#1}\right\}}
\newcommand{\dotprod}[2]{\langle#1,#2\rangle}
\newcommand{\est}[1]{\widehat{#1}}

\newcommand{\matR}{\ensuremath{\mathbb{R}}}

\newcommand{\rmi}{\iota}
\newcommand{\gdot}{\dot{g}}
\newcommand{\ghat}{\est{g}}

\newcommand{\rmd}{{\rm d}}
\newcommand{\Lip}{M}

\newcommand{\real}{\text{Re}}
\newcommand{\imag}{\text{Im}}
\newcommand{\gest}{\ensuremath{\est{g}}}

\newcommand{\gtil}{\ensuremath{\widetilde{g}}}
\newcommand{\ftil}{\ensuremath{\widetilde{f}}}

\newcommand{\xtil}{\ensuremath{\widetilde{x}}}

\newcommand{\util}{\ensuremath{\widetilde{u}}}
\newcommand{\gesttil}{\ensuremath{\widetilde{g}}}
\newcommand{\smooth}{\ensuremath{B}}

\newcommand{\fhat}{\ensuremath{\est{f}}}
\newcommand{\qhat}{\ensuremath{\est{q}}}  
\newcommand{\itup}{\ensuremath{\mathbf{i}}}  
\newcommand{\jtup}{\ensuremath{\mathbf{j}}}  
\newcommand{\diff}{\ensuremath{D}}

\title{Denoising modulo samples: $k$-NN regression and tightness of SDP relaxation\footnotetext{Authors are listed in alphabetical order} }

\author{Micha{\"e}l Fanuel \thanks{KU Leuven, Department of Electrical Engineering (ESAT),
STADIUS Center for Dynamical Systems, Signal Processing and Data Analytics,
Kasteelpark Arenberg 10, B-3001 Leuven, Belgium.}\\ \texttt{michael.fanuel@kuleuven.be}
\and Hemant Tyagi\thanks{Inria, Univ. Lille, CNRS, UMR 8524 - Laboratoire Paul Painlev\'{e}, F-59000 } \\ \texttt{hemant.tyagi@inria.fr}}
 
\begin{document}
\maketitle

\begin{abstract}
Many modern applications involve the acquisition of noisy modulo samples of a function $f$, with the goal being to recover estimates of the original samples of $f$. For a Lipschitz function $f:[0,1]^d \to \matR$, suppose we are given the samples $y_i = (f(x_i) + \eta_i)\bmod 1; \quad i=1,\dots,n$ where $\eta_i$ denotes noise. Assuming $\eta_i$ are zero-mean i.i.d Gaussian's, and $x_i$'s form a uniform grid, we derive a two-stage algorithm that recovers estimates of the samples $f(x_i)$ with a uniform error rate $O((\frac{\log n}{n})^{\frac{1}{d+2}})$ holding with high probability. The first stage involves embedding the points on the unit complex circle, and obtaining denoised estimates of $f(x_i)\bmod 1$ via a $k$NN (nearest neighbor) estimator. The second stage involves a sequential unwrapping procedure which unwraps the denoised mod $1$ estimates from the first stage. \rev{The  estimates of the samples $f(x_i)$ can be subsequently utilized to construct an estimate of the function $f$, with the aforementioned uniform error rate.}

Recently, Cucuringu and Tyagi \cite{CMT18_long} proposed an alternative way of denoising modulo $1$ data which works with their representation on the unit complex circle. They formulated a smoothness regularized least squares problem on the product manifold of unit circles, where the smoothness is measured with respect to the Laplacian of a proximity graph $G$ involving the $x_i$'s. This is a nonconvex quadratically constrained quadratic program (QCQP) hence they proposed solving its semidefinite program (SDP) based relaxation. We derive sufficient conditions under which the SDP is a tight relaxation of the QCQP. Hence under these conditions, the global solution of QCQP can be obtained in polynomial time.
\end{abstract}

% Introduction
%------------------
% Introduction
%
\section{Introduction} \label{sec:intro}
In many real-life applications, we are often given access to noisy modulo samples of an underlying signal $f: \matR^d \to \matR$, i.e., 
\begin{equation} \label{eq:noisy_mod}
y_i = (f(x_i) + \eta_i)\bmod \zeta; \quad i=1,\dots,n 
\end{equation}
for some $\zeta \in \matR^+$, where $\eta_i$ denotes noise. Here, $a \bmod \zeta \in [0,\zeta)$ is the remainder term so that $a = q \zeta + (a\bmod \zeta)$ for an integer $q$. For example, self-reset analog to digital converters (ADCs) are a new generation of ADCs which handle voltage surges by resetting its value via a modulo operation. In other words, if the voltage signal lies outside the range $[0,\zeta]$, then its value is simply reset by taking its modulo $\zeta$ value \cite{Kester,RHEE03,yamaguchi16}. \rev{In this paper, we analyse a situation where only modulo samples are available and not the reset counts (i.e., the integer $q$).}
Another important application is phase unwrapping where the general idea is to infer the structure of an object by transmitting waveforms, and capturing the phase coherence (measured modulo $2\pi$ radians) between the transmitted and scattered waveforms. This arises for instance in InSAR (synthetic radar aperture interferometry) for estimating the depth map of a terrain (e.g., \cite{graham_insar,zebker86}); MRI, for estimating the position of veins in tissues (e.g., \cite{hedley92,laut73}), and non destructive testing of components (e.g., \cite{paoletti94,hung96}), to name a few applications.
\rev{A main difference between phase unwrapping and moludo samples coming from a self-reset ADC is that folding is deliberately injected in the latter case, while in phase unwrapping problems the data is assumed to be available in the form of modulo $2\pi$ information. }

Given the measurement model in \eqref{eq:noisy_mod}, where we assume from now that $\zeta = 1$, we are interested in \emph{unwrapping} $(y_i)_i$ to recover the original samples $f(x_i)$ for $i=1,\dots,n$. This is clearly only possible up to a global integer shift. Moreover, it is not difficult to see that one needs to make additional structural assumptions on $f$, such as of smoothness (Lipschitz, continuous differentiability etc.). 

Let us first discuss a natural sequential procedure for unwrapping $y_i$ with $f$ assumed to be Lipschitz smooth; the reader is referred to Section~\ref{subsec:unwrap_mod} for details. To begin with, if there is no noise, one can exactly recover the original samples $f(x_i)$ provided $n$ is large enough. To see this, let us consider first the univariate setting with $f:[0,1] \rightarrow \mathbb{R}$, and suppose for simplicity that the $x_i$'s form a uniform grid. If $n$ is large enough w.r.t. the Lipschitz constant of $f$, then the following identity is easy to verify (see Lemma~\ref{lem:itoh_type_cond}) and is reminiscent of the classical Itoh's condition \cite{Itoh_82} from the phase unwrapping literature,
\begin{equation*}  
  f(x_i) - f(x_{i-1}) =  \left\{
\begin{array}{rl}
y_i - y_{i-1} \ ; & \text{ if } \abs{y_i - y_{i-1}} < 1/2, \\
1 + y_i - y_{i-1} \ ; & \text{ if } y_i - y_{i-1} < -1/2, \\
-1 + y_i - y_{i-1} \ ; & \text{ if } y_i - y_{i-1} > 1/2.
\end{array} \right.
\end{equation*}
This directly suggests a simple sequential procedure for recovering the original samples $f(x_i)$ using $y_i$. The above argument can be extended to the noisy setting -- one can show that if $\eta_i \bmod 1 \leq \delta$ for all $i$, then provided $\delta \lesssim 1$ and $n$ is large enough, one can apply the same sequential procedure to obtain estimates $\ftil(x_i)$ such that for some integer $q^{\star}$,
\begin{equation} \label{eq:err_bd_seq_unwrap_intro}
    \abs{\ftil(x_i) + q^{\star} - f(x_i)} \leq \delta, \quad  i=1,\dots,n;
\end{equation}
see Lemma~\ref{lem:err_seq_unwrap}. In fact, perhaps surprisingly, one can generalize the above discussion to the general multivariate setting as well. \rev{In that case, we assume that $n=m^d$ where $d\geq 1$ is the dimension of the problem.} In particular, we show that if $\delta \lesssim 1$ and $n$ is large enough, then there exists a sequential unwrapping procedure (see Algorithm~\ref{algo:seq_unwrap_mult}) that generates estimates \rev{$\ftil(x_\itup)$ satisfying a generalization to $d\geq 1$ of~\eqref{eq:err_bd_seq_unwrap_intro} %, given in ~\eqref{eq:err_bd_seq_unwrap_mult},
 for all $\itup\in [m]^d$}; see Lemma\rev{s}~\ref{lem:itoh_type_cond_mult},~\ref{lem:err_seq_unwrap_mult}. 

An important takeaway from the above discussion is that one could consider a two stage approach for unwrapping -- first denoise the modulo samples $(y_i)_i$, and then apply the aforementioned unwrapping procedure. Indeed, the hope is that the denoising procedure will lead to estimates of $f(x_i) \bmod 1$ with a uniform error bound much smaller than $\delta$, which in turn will improve the final error estimate for the unwrapped samples on account of \eqref{eq:err_bd_seq_unwrap_intro}. This is also the basis of our first main result for this problem which we now outline. \rev{ Before concluding, we remark that while the focus of the above discussion was on recovering the samples $f(x_i)$, we will show in Section \ref{subsec:estim_f} that one can subsequently construct an estimate of the function $f$ (using the recovered samples) via classical tools from approximation theory.}
%
%-------------------------------------------------------
% Error rates for unwrapping noisy modulo $1$ data
%-------------------------------------------------------
\subsection{Error rates for unwrapping noisy modulo $1$ data} \label{subsec:err_rate_unwrap}
We derive an algorithm (namely, Algorithm~\ref{algo:Main}) for the problem of unwrapping \rev{ of } the noisy modulo $1$ samples $(y_i)_{i=1}^n$ generated as in \eqref{eq:noisy_mod}. The algorithm \rev{ consists of two stages}. In the first stage, we obtain denoised modulo samples by performing a $k$NN ($k$ nearest neighbor) regression procedure. Specifically, this is done by embedding the mod $1$ data \rev{onto} the unit circle as $z_i = \exp(\iota 2\pi y_i)$ for each $i$, and by then performing a $k$NN estimation in this space (see Algorithm~\ref{algo:kNNregression}). Assuming $\eta_i \sim \calN(0,\sigma^2)$ to be i.i.d. Gaussian, and the $x_i$'s forming a uniform grid in $[0,1]^d$, this results in denoised estimates of $f(x_i) \bmod 1$ satisfying, with high probability, a uniform error bound (w.r.t the wrap around metric) of $O((\frac{\log n}{n})^{\frac{1}{d+2}})$. Then\rev{, in the second stage,} feeding the denoised mod 1 samples to the multivariate unwrapping procedure in Algorithm~\ref{algo:seq_unwrap_mult}, the same error rate carries over for the unwrapped estimates, i.e., $\delta = O((\frac{\log n}{n})^{\frac{1}{d+2}})$ in \eqref{eq:err_bd_seq_unwrap_intro}. We outline this below in the form of the following informal Theorem; the full result is in Theorem~\ref{thm:main_thm_unwrap_mod1} and is completely non-asymptotic.
\begin{theorem}
\rev{Let $n = m^d$ and assume the $x_\itup$'s form a uniform grid in $[0,1]^d$ for $\itup \in [m]^d$.}
In the model \eqref{eq:noisy_mod} with $\zeta = 1$, suppose $\eta_\itup \sim \calN(0,\sigma^2)$ i.i.d and $f:[0,1]^d \rightarrow \mathbb{R}$ is Lipschitz continuous w.r.t the $\ell_{\infty}$ norm. Let $\sigma \leq \frac{1}{2\pi}$. Then, if $n$ is large enough, the estimates \rev{ $\ftil(x_\itup)$} obtained from Algorithm~\ref{algo:Main} satisfy, with high probability, the uniform error bound
\begin{equation} \label{eq:disc_linfty_err_intro}
    \abs{\ftil(\rev{x_\itup}) + q^{\star} - f(\rev{x_\itup})} = O\left(\left(\frac{\log n}{n}\right)^{\frac{1}{d+2}} \right), \quad  \rev{\itup\in [m]^d};
\end{equation}
for some $q^{\star} \in \mathbb{Z}.$
\end{theorem}
%
%\MF{What about the numbering of Theorems? Do we write both an informal and formal version of the same theorem?}
The rate $\left(\frac{\log n}{n}\right)^{\frac{1}{d+2}}$ is the well known minimax optimal rate for estimating a Lipschitz function in the $L_{\infty}$ norm over the cube $[0,1]^d$ (using a uniform grid), see for e.g., \cite[Theorem 1.3.1]{nemirovski2000topics}. While we defer a detailed discussion with existing work to the end of the paper, we remark that such a result has so far been elusive in the literature for the modulo measurement model in \eqref{eq:noisy_mod}.  
\rev{\paragraph{Estimating $f$.} Once we have the estimates $\ftil(x_\itup)$ on a uniform grid in $[0,1]^d$, it is straightforward to construct an estimate $\est{f}$ of the function $f$, using the recovered samples $\ftil(x_\itup)$. Indeed, we show in Section \ref{subsec:estim_f} that this can be accomplished via quasi-interpolant operators which are classical tools from approximation theory. In particular, we show in Theorem \ref{thm:main_unwrap_err_f} therein that the uniform error bound in \eqref{eq:disc_linfty_err_intro} directly implies the same $L_{\infty}$ error rate between $\est{f}$ and $f$, i.e., 
\begin{equation*}
    \norm{\est{f} + q^* - f}_{\infty} = O\left(\left(\frac{\log n}{n}\right)^{\frac{1}{d+2}} \right).
\end{equation*}
Hence we can estimate Lipschitz continuous functions from their noisy modulo samples (on a uniform grid in $[0,1]^d$) at the optimal $L_{\infty}$ rate.
}

%------------------------------------------------------------
% Tightness of SDP formulation for denoising modulo $1$ data
%------------------------------------------------------------
\subsection{Tightness of SDP formulation for denoising modulo $1$ data} \label{subsec:tightness_SDP_intro}
In the previous section, observe that the denoising of the modulo $1$ samples was performed by representing the samples $y_i$ on the unit complex circle (denoted $\mathbb{T}_1$) as $z_i = \exp(\iota 2 \pi y_i)$, with $z = (z_1,\dots,z_n) \in \mathbb{T}_n$. Here, $\mathbb{T}_n$ is the product manifold of $n$ unit complex circles. This representation idea is motivated from a recent paper of Cucuringu and Tyagi \cite{CMT18_long} where they proposed a different scheme for denoising mod 1 samples, which we now describe. 

For a smooth function $f$, we know that $\exp(\iota 2\pi f(x_i)) \approx \exp(\iota 2\pi f(x_j))$ provided $x_i \approx x_j$. Hence,  \cite{CMT18_long} proposed constructing a proximity graph $G$ on the sampling points -- with an edge between $i$ and $j$ provided $x_i$ is close enough to $x_j$ -- and solving the following optimization problem
\begin{equation} 
\min_{g \in \mathbb{T}_n} \norm{g - z}_2^2 + \lambda g^* L g \iff \min_{g \in \bbT_n} \lambda g^* L g - 2\real(g^* z).  \label{prog:qcqp} \tag{$\text{QCQP}$}  
\end{equation}
Here, $L$ is the Laplacian of $G$, and $\lambda \geq 0$ is a regularization parameter that promotes smoothness with respect to $G$. This is a non-convex problem -- albeit with a convex objective -- and it is unclear whether one can efficiently (i.e., in polynomial time) find a global minimizer. Therefore, they considered solving the semidefinite progamming relaxation of \eqref{prog:qcqp}, i.e., 
\begin{align} \label{prog:sdp_relax}
\min_{W \in \mathbb{C}^{(n+1) \times (n+1)}} \Tr(T W) \quad \text{ s.t } \quad W \succeq 0, \ W_{ii} = 1   \tag{$\text{SDP}$}
\end{align}
which is solvable in polynomial time via interior point methods (see for e.g. \cite{Vanden96}). The matrices $T, W $ are defined in \eqref{eq:Tdef} and the steps leading to the formulation \eqref{prog:sdp_relax} are outlined in Section~\ref{subsec:prob_sdp}. As discussed therein, if the solution $X$ of \eqref{prog:sdp_relax} is rank $1$, then it has the form
\begin{eqnarray} \label{eq:X_sol_rank1_intro}
X = \begin{pmatrix}
  \gest \gest^* \quad & \gest \\ 
  \gest^* \quad & 1  
\end{pmatrix} =
\begin{pmatrix} \gest \\ 1 \end{pmatrix}
\begin{pmatrix} \gest^* & 1 \end{pmatrix}
\end{eqnarray}
where $\gest \in \mathbb{T}_n$ is a global solution of \eqref{prog:qcqp}. 

\paragraph{Main result.} Hence an important question is to identify conditions under which \eqref{prog:sdp_relax} is a tight relaxation of \eqref{prog:qcqp}, i.e., its solution is rank $1$, since under those conditions the solution of \eqref{prog:qcqp} would have been obtained in polynomial time. Such an analysis was missing in \cite{CMT18_long} although experimentally, \eqref{prog:sdp_relax} was shown to perform quite well. This brings us to the second main result of this paper where we take a step towards answering this question. It is outlined in the theorem below; for the complete statement, see Theorem~\ref{thm:main_sdp_relax_tight}.
\begin{theorem} \label{thm:main_sdp_relax_intro}
Let $z \in \bbT_n$ be a noisy observation of the ground truth signal $h \in \mathbb{T}_n$ satisfying $\norm{z-h}_{\infty} \leq \delta$. Denote $\Delta$ to be the maximum degree of the graph $G$. Then there exist constants $0 < c_1, c_2 < 1$ such that if $\delta \leq c_1$  and $\lambda \Delta \leq c_2$,
then \eqref{prog:sdp_relax} has a unique solution $X \in \mathbb{C}^{(n+1) \times (n+1)}$ of the form \eqref{eq:X_sol_rank1_intro}.  Consequently, $\gest$ is the unique solution of \eqref{prog:qcqp}.
\end{theorem}
%\MF{Same question as above: do we write both an informal and formal version of the same theorem?}
%
%
The above result is for any graph $G$, and does not make any assumptions on the noise, other th\rev{a}n being uniformly bounded. While in the setup of \cite{CMT18_long}, we have $h_i = \exp(\iota 2\pi f(x_i))$, the framework in which we study the problem is more abstract since it applies to any graph $G$ and does not necessarily assume the model in \eqref{eq:noisy_mod}. Since $\norm{z-h}_{\infty} \leq 2$ is always true, the requirement $\delta \lesssim 1$ is not stringent. On the other hand, one might perhaps intuitively expect that the smoothness of $h$ w.r.t. $G$ should also play an important role as part of the conditions ensuring tightness. While Theorem~\ref{thm:main_sdp_relax_tight} does have a  smoothness parameter $0 \leq \smooth_n \leq 2$ (see \eqref{eq:hsmooth} for definition) appearing in the conditions, the effect is admittedly mild. This is likely due to an artefact of the analysis, and is discussed in detail in Section~\ref{sec:discussion}.
Denoising mod 1 samples thanks to (SDP) is empirically very successful, however this insight is not yet fully reflected by our theoretical understanding.
Although we believe that our result about the SDP tightness is meaningful, we expect that these guarantees can be improved, in particular, under a random noise assumption.  Those prospects are discussed in Section~\ref{subsec:fut_work}.
%\HT{State that we dmit the result can be improved and point to later remark and Section 5.2.}

\paragraph{\rev{Practical considerations}.} %\label{subsec:QCQP_vs_kNN}
\rev{The \eqref{prog:qcqp} problem is particularly natural. Its SDP relaxation discussed here can be solved thanks to the Burer-Monteiro approach (see \cite{CMT18_long}) which typically scales better with the data set size compared with interior point methods. Compared with the $k$NN approach, we expect the SDP problem to be more robust to large noise values as it is already discussed in the context of angular synchronization~\cite{SyncRank}. Clearly, the $k$NN approach is a simpler and faster strategy which showed good performance in our numerical experiments. }
%--------------------------------
% Notation and outline of paper
%--------------------------------
\subsection{Notation and outline of paper}
We now discuss the notation used throughout, followed by the outline of the rest of the paper.
\paragraph{Notation.}
We will denote $[n] = \{1,\dots , n\}$, and $\rmi = \sqrt{-1}$ to be the imaginary unit. The symbol $$\mathbb{T}_n := \set{u \in \mathbb{C}^n: \abs{u_i} = 1; \ i=1,\dots,n}$$
is the product manifold of unit radius circles, 
i.e., $\mathbb{T}_n = \mathbb{T}_1 \times \cdots \times \mathbb{T}_1$. For $u \in \mathbb{C}$, we define a projection on $\mathbb{T}_n$ as 
\begin{equation*}
    \left(\frac{u}{|u|}\right)_i = \begin{cases}
    \frac{u_i}{|u_i|} \text{ if } u_i\neq 0,  \\
    1 \text{ otherwise, }
    \end{cases}
\end{equation*}
for all  $i\in[n]$, and also define the angle $\arg(u)\in [0,2\pi)$ such that $u= |u| \exp(\iota \arg(u))$. Denote $d_w:[0,1) \rightarrow [0,1/2]$ to be the usual wrap around metric defined as
\begin{equation*}
 d_w(t,t') := \max \set{\abs{t-t'}, 1-\abs{t-t'}}.
\end{equation*}
For a vector $x \in \mathbb{C}^n$ and any $1 \leq p \leq \infty$, $\norm{x}_p$ denotes the usual $\ell_p$ norm of $x$. 
We say that a function $f: [0,1]^d \to \mathbb{C}$ is $\Lip$-Lipschitz if there exists a constant $M > 0$ such that $|f(x)-f(y)|\leq \Lip\|x-y\|_\infty$ for all $x,y \in [0,1]^d$. Moreover, the $L_{\infty}$ norm of $f$ is defined as $\|f\|_\infty := \sup_{x\in [0,1]^d}|f(x)|$.
For non-negative numbers $a,b$, we write $a \lesssim b$ if there exist a constant $C > 0$ such that $a \leq C b$. Furthermore, we write $a\asymp b$ if $a \lesssim b$ and $b \lesssim a$. Finally, we also denote by $\circ$ the usual Hadamard product.
\paragraph{Outline of paper.}  Section~\ref{sec:knn_analysis} contains the analysis for the unwrapping problem, culminating with \rev{Theorems \ref{thm:main_thm_unwrap_mod1} and \ref{thm:main_unwrap_err_f} which are our main results for this problem.} Section~\ref{sec:sdp_analysis} derives sufficient conditions under which  \eqref{prog:sdp_relax} is a tight relaxation of \eqref{prog:qcqp}, with Theorem~\ref{thm:main_sdp_relax_tight} being our main result for this problem. Section~\ref{sec:sims} contains some numerical simulations, and we conclude with a discussion with related work along with directions for future work in Section~\ref{sec:discussion}.

% Regression
%-------------------------------------------------
% Denoising and unwrapping via $k$NN regression
%
\section{Denoising and unwrapping via $k$NN regression} \label{sec:knn_analysis}
In this section, we introduce and analyze an algorithm for robustly unwrapping noisy mod $1$ samples of a Lipschitz function. We begin by formally outlining the problem setup. 
\subsection{Problem setup \label{sec:prob_setup}}
Let $f:[0,1]^d\to \mathbb{R}$ be an unknown $\Lip$-Lipschitz function. Let the circle-valued function $h:[0,1]^d\to \bbT_1$ be given as $h(x) = \exp(\rmi 2\pi f(x))$. 
\begin{fact}\label{Fact:Lipschiz}
The function  $h:[0,1]^d \to \bbT_1$ be given as $h(x) = \exp(\rmi 2\pi f(x))$ is $2\pi \Lip$-Lipschitz.
\end{fact}
\begin{proof}
We have 
 $|h(x)-h(y)| = 2|\sin[\pi (f(x)-f(y))]|\leq 2\pi|f(x)-f(y)|\leq 2\pi \Lip \|x-y\|_\infty,$
 for all $x,y\in[0,1].$
\end{proof}
We consider $n$ datapoints  on a uniform grid $\calX$ of points  $x_{\itup} = (x_{i_1},\dots,x_{i_d}) \in [0,1]^d$  indexed by the $d$-tuple $\itup = (i_1,\dots,i_d) \in [m]^d$, where $x_{i_j} = \frac{i_j -1}{m-1}$.
We assume that we have noisy versions of $f(x_\itup)$ modulo $1$, that is, $y_\itup = \left( f(x_\itup) + \eta_\itup \right) \mod 1$ where $\eta_\itup \sim \calN(0,\sigma^2)$ i.i.d. These noisy modulo $1$ samples are mapped to the complex circle as 
\begin{equation}
z_\itup = \exp(\rmi 2\pi y_{\itup}) =  h_\itup \exp(\rmi 2\pi \eta_\itup), 
%= \exp\left(\rmi 2\pi (f(x_\itup)+\eta_\itup)\right)  ,
\label{eq:z}
\end{equation}
where, for simplicity, we write $h_\itup = h(x_\itup)$.
The following simple fact will be used extensively in our analysis.
\begin{fact}\label{lemma:ExpZi}
We have $\mathbb{E} [z_\itup] = e^{-2\pi^2\sigma^2} h_\itup$ for all $\itup\in [m]^d$.
\end{fact}
\begin{proof}
This follows from the moment generating function of a normal \rev{distribution}.
\end{proof}
Our goal is to obtain estimates of the samples $f(x_{\itup})$, namely $\ftil(x_{\itup})$, such that for some integer $q^{\star} \in \mathbb{Z}$, $\abs{\ftil(x_{\itup}) + q^{\star} - f(x_{\itup})}$ is ``small'' for all $\itup \in [m]^d$. To this end, we propose a two-stage strategy outlined formally as Algorithm~\ref{algo:Main}. 
\begin{enumerate}
    \item In the first stage, we consider a $k$ nearest neighbors ($k$NN) regression scheme applied to the noisy samples $z_{\itup}$. The purpose of this stage is to produce denoised mod 1 estimates $\est{g}(x_{\itup}) \in [0,1)$ for all $\itup \in [m]^d$. This is outlined as Algorithm~\ref{algo:kNNregression} and analyzed in Section~\ref{sec:DenoiseViakNN}.
    
    \item The second stage involves a sequential unwrapping procedure which takes the denoised estimates $\est{g}(x_{\itup})$ as input, and outputs the final unwrapped estimates $\ftil(x_{\itup})$ for all $\itup \in [m]^d$. This is outlined as Algorithm~\ref{algo:seq_unwrap_mult} and analyzed in Section~\ref{subsec:unwrap_mod}.
\end{enumerate}
\rev{In Section \ref{subsec:main_res_unwrap_samps}, we put together our results from the preceding sections to derive bounds on $\abs{\ftil(x_{\itup}) + q^{\star} - f(x_{\itup})}$, holding uniformly for each $x_\itup \in \calX$ (see Theorem \ref{thm:main_thm_unwrap_mod1}). In Section \ref{subsec:estim_f}, we will describe how the recovered estimates $\ftil(x_\itup)$ can be used to obtain an estimate $\est{f}$ of the function $f$ via quasi-interpolant operators. In particular, the error rate (on the grid) in Theorem \ref{thm:main_thm_unwrap_mod1} is shown to carry forward for $\norm{\est{f} + q^* - f}_{\infty}$ as well (see Theorem \ref{thm:main_unwrap_err_f}).}

%---------------------------------------------
% Denoising mod 1 samples via kNN regression
%---------------------------------------------
\subsection{Denoising mod 1 samples via $k$NN regression \label{sec:DenoiseViakNN}}
Our $k$NN scheme for denoising the modulo samples is outlined in Algorithm~\ref{algo:kNNregression}. Before proceeding with its analysis, it will be useful to introduce some preliminaries for the $k$NN estimator. 
%
%
%-------------------------------------------------------------
% Algorithm: Denoising modulo samples with $k$NN regression
% 
\begin{algorithm*}[!ht]
\caption{Denoising modulo samples with $k$NN regression} \label{algo:kNNregression} 
\begin{algorithmic}[1] 
\State \textbf{Input:} integer $k>0$ and uniform grid $\calX \subset [0,1]^d$, $\abs{\calX} = n = m^d$; noisy modulo samples $y_\itup = (f(x_\itup)+\eta_\itup) \mod 1$ for $\itup\in [m]^d$.
\State \textbf{Output:} denoised modulo samples $\est{g}(x_\itup) \in [0,1)$  for all $\itup\in [m]^d$.
\State Compute $z_\jtup = \exp(\rmi 2\pi y_\jtup)$ for all $\jtup\in [m]^d$.
\For{$\jtup\in [m]^d$}
\State Compute $h_k(x_\jtup) = \frac{1}{k}\sum_{\itup: x_\itup \in \mathcal{N}_k(x_\jtup)} z_\itup$ and  normalize $\est{h}_k(x_\jtup) = \frac{h_k(x_\jtup)}{|h_k(x_\jtup)|}$.
\State $\est{g}(x_\jtup) = \frac{1}{2\pi}\arg\big(\est{h}_k(x_\jtup)\big)$.
\EndFor
\end{algorithmic}
\end{algorithm*}
%
%
%------------------------
% kNN estimator 
% 
\paragraph{$k$NN estimator.}
Let $x\in [0,1]^d$ and $ k$ be a strictly positive integer. Then, we define the $k$NN radius of $x$ as the smallest distance such that a $\ell_\infty$ ball centered at $x$ contains at least $k$ neighbours, namely  $r_k(x) = \inf\{ r: |B(x,r)\cap \calX|\geq k\}$ where $B(x,r) = \{x'\in[0,1]^d:\|x-x'\|_\infty \rev{\leq} r\}$. Hence, the $k$NN set of $x$ is simply $\mathcal{N}_k(x)= B(x,r_k(x))\cap \calX$.
We are ready to introduce the $k$NN estimator 
$$
\est{h}_k(x) = \frac{h_k(x)}{|h_k(x)|} \text{ with } h_k(x) = \frac{1}{|\mathcal{N}_k(x)|}\sum_{\itup: x_\itup \in \mathcal{N}_k(x)} z_\itup.
$$
Notice that $\est{h}_k(x)$ does not depend on the normalization of $h_k(x)$.
Hence, we introduce $\widetilde{h}_k(x) = e^{2\pi^2\sigma^2}h_k(x)$ where
by construction  $\mathbb{E}[\widetilde{h}_k(x)] = \frac{1}{|\mathcal{N}_k(x)|}\sum_{\itup: x_\itup \in \mathcal{N}_k(x)} h_\itup$ holds thanks to Fact~\ref{lemma:ExpZi}. 

%-------------------------
% Statistical guarantees
%
\paragraph{Statistical guarantees.} In order to obtain statistical guarantees for the $k$NN regressor, we first derive an expression for the $k$NN radius on a grid which essentially allows for bounding the bias of our estimator. 
\begin{lemma}\label{lemma:kNN}
With the notations defined above, it holds that
$
\sup_{x\in [0,1]^d}r_k(x)= \frac{\lceil k^{1/d}\rceil -1}{m-1}.
$
\end{lemma}
\begin{proof}
%Let $r_k^\star = \sup_{x\in [0,1]^d}r_k(x)$. 
The supremum can be attained at several $x\in [0,1]^d$, in particular, it is attained at a corner of the hyper-cube $[0,1]^d$, say $x = 0$ to fix the ideas. Let a sub-cube be positioned at $0$ so that each of its edges contains
$\ceil{k^{1/d}}$ grid points. This means that this hypercube contains at least $k$ grid points and that the length of its edge is $c = \frac{\ceil{k^{1/d}} -1}{m-1}$. This cube is included in a $\ell_{\infty}$ ball centered at $0$ and of radius $c$ which completes the proof.
\end{proof}
An upper bound on the pointwise expected risk follows readily from  Lemma~\ref{lemma:kNN}. This result is given in Proposition~\ref{Proposition:ExpectedRisk}, which displays a classical bias-variance trade-off in terms of the number of neighbours $k$.
\begin{proposition}[Pointwise expected risk] \label{Proposition:ExpectedRisk}
Let $x\in [0,1]^d$. If $\sigma \leq \frac{1}{2\pi}$ and if $n\geq 2^d$, we have
 $\mathbb{E} \left|\est{h}_k(x)-h(x)\right|^2\leq   64\pi^2 \Lip^2 \left(\frac{k}{n}\right)^{2/d} + \frac{32\pi^2\sigma^2}{k}.$
\end{proposition}
\begin{proof}
%Firstly, we use the following fact
%\begin{customthm}{3}[\cite{LiuPhaseSynchronization}]\label{Prop:Projection}
%Let $q\geq 1$, $z\in \bbT_n$ and $w\in \mathbb{C}_n$. Then, it holds
%$\| \frac{w}{|w|} -z \|_q \leq 2 \| w - z\|_q.$
%\end{customthm}
Thanks to Fact~\ref{Prop:Projection}, we have the following inequality
$
|\est{h}_k(x)-h(x)|\leq 2|\widetilde{h}_k(x)-h(x)|,
$
so that we upper bound only $|\widetilde{h}_k(x)-h(x)|$.
Classically, we have the bias/variance splitting
$$
\mathbb{E} \left|\widetilde{h}_k(x)-h(x)\right|^2= \underbrace{\mathbb{E} \left|\widetilde{h}_k(x)-\mathbb{E} [\widetilde{h}_k(x)]\right|^2}_{\text{variance}}+\underbrace{\left|\mathbb{E} [\widetilde{h}_k(x)]-h(x)\right|^2}_{\text{bias}^2}.
$$
Then, by using Fact~\ref{Fact:Lipschiz}, we have $\left|h(x_\itup) -h(x)\right|\leq  2\pi \Lip \left\|x_\itup -x\right\|_\infty$, and therefore, $\text{bias}^2\leq 4\pi^2 \Lip^2 r_k(x)^2$. By using Lemma~\ref{lemma:kNN}, we find
\[
\sup_{x\in [0,1]^d}r_k(x)= \frac{\lceil k^{1/d}\rceil -1}{m-1}\leq 2 \frac{k^{1/d}}{m}
\]
for $m\geq 2$ where we used that $1/(m-1)\leq 2/m$.
The variance can be exactly computed as follows
\begin{align*}
\mathbb{E} \left|\widetilde{h}_k(x)-\mathbb{E} [\widetilde{h}_k(x)]\right|^2 = \frac{1}{|\mathcal{N}_k(x)|^2}\sum_{\itup: x_\itup \in \mathcal{N}_k(x)}\mathbb{E} \left| \frac{z_\itup}{e^{-2\pi^2\sigma^2}} -h_\itup\right|^2
 = \frac{e^{4\pi^2\sigma^2}-1}{|\mathcal{N}_k(x)|},
\end{align*}
where we used again the formula of moment generating function of a normal \rev{distribution}.
Next, we use the inequality $e^x \leq 1+2x$ if $x\leq 1$. This gives
$ e^{4\pi^2\sigma^2}-1\leq 8\pi^2\sigma^2$ if $4\pi^2\sigma^2\leq 1$. By combining the bounds on the bias and variance, we obtain the desired result.
\end{proof}
\begin{corollary}[Rate for pointwise expected risk]
By choosing the number of neighbours $k= \ceil{k^\star}$ with
$
k^\star =  \left(\frac{d\sigma^2}{4\Lip^2 }\right)^{\frac{d}{d+2}} n^{\frac{2}{d+2}},
$
we obtain the following bound on the expected risk $$\mathbb{E} \left|\est{h}_k(x)-h(x)\right|^2 \leq 640 \pi^2   \Lip^{\frac{2d}{d+2}} \sigma^{\frac{4}{d+2}} n^{-\frac{2}{d+2}}.%256 \pi^2   \Lip^{\frac{2d}{d+2}} \sigma^{\frac{4}{d+2}} n^{-\frac{2}{d+2}}. 
$$
The rate $n^{-\frac{2}{d+2}}$ matches the pointwise rate for estimation of a Lipschitz function on the $[0,1]^d$ cube from a uniform grid, see page 24 of~\cite{nemirovski2000topics}.
\end{corollary}
\begin{proof}
The proof goes as follows. The value of $k^\star$ is obtained by minimizing the RHS of the bound in Proposition~\ref{Proposition:ExpectedRisk}, which is considered as a function over the reals of the form
$
R(k) = \alpha(n) k^{2/d} + \beta/k.
$
It is minimized at $k^\star = \left(\frac{d\beta}{2\alpha(n)} \right)^{\frac{d}{d+2}}$. Hence, by using $k^\star\leq \ceil{k^\star}\leq 2k^\star$, we find 
$$
R(k^\star) \leq R(k) = \alpha(n) k^{2/d} + \frac{\beta}{k}\leq \alpha(n) (2 k^{\star })^{2/d}+ \frac{\beta}{k^\star} = \widetilde{R}(k^\star),
$$
with $k=\ceil{k^\star}$.
The latter upper bound is 
$$
\widetilde{R}(k^\star)= \alpha(n)^{\frac{d}{d+2}}\beta ^{\frac{2}{d+2}}\left( 2^{\frac{2}{d}}\left(\frac{d}{2}\right)^{\frac{2}{d+2}} + \left(\frac{2}{d}\right)^{\frac{d}{d+2}} \right)\leq 10 \alpha(n)^{\frac{d}{d+2}}\beta ^{\frac{2}{d+2}}
$$
where we used the simplifications $(\frac{d}{2})^{\frac{2}{d+2}} \leq (\frac{d+2}{2})^{\frac{2}{d+2}} \leq 2$  and $(\frac{2}{d}) ^{\frac{d}{d+2}} \leq 2$, as well as $2^{\frac{2}{d}}\leq 4$ for $d\geq 1$. Then, the final expression is obtained by using the inequality $2^{\frac{d}{d+2}}<2$ for $d\geq 1$.
%The statistical rate is obtained by substituting $k^\star$ %hereabove, and simplifying the ensuing expression by noting that 
%
%\begin{equation*}
    %\left(\frac{d}{4}\right)^{\frac{2}{d+2}} + %\left(\frac{4}{d}\right)^{\frac{d}{d+2}} \leq 4.
%\end{equation*}
%
\end{proof}
Now, we provide \rev{a high probability error bound} for the estimator and the ground-truth mod $1$ function, with respect to the sup-norm. 
\begin{theorem}\label{Thm:whp_l_infty}
\rev{If $\sigma\leq \frac{1}{2\pi}$ and $n\geq 2^d$, the following in-sample $\ell_\infty$ bound holds.  With probability at least $1-1/n$, we have
\begin{equation}
\left|\est{h}_k(x_\itup)-h(x_\itup)\right| \leq 8\pi \Lip \left(\frac{k}{n}\right)^{1/d} + \frac{64}{3}(2\pi^2\sigma^2+1)\frac{\log n}{k}+32\pi\sigma\sqrt{\frac{\log n}{k}}, \ \forall \itup\in[m]^d.  \label{eq:in_sample_sup_norm_whp}
\end{equation}}
\end{theorem}
In order to have a non-empty bound, since $|\est{h}_k(x)-h(x)|\leq 2 $, we need to make sure that the RHS in Theorem~\ref{Thm:whp_l_infty} is smaller than $2$. Notice that, compared with $k$NN regression in the absence of the mod $1$ indeterminacy \cite{Jiang19}, the variance term in the bound does not vanish if $\sigma = 0$. This is due to the Bernstein inequality used in the proof.
\begin{proof}
Let $x\in [0,1]$. 
Firstly, we have the inequality
$
|\est{h}_k(x)-h(x)|\leq 2|\widetilde{h}_k(x)-h(x)|,
$
thanks to Fact~\ref{Prop:Projection}.
Then, we have
\[
|\widetilde{h}_k(x)-h(x)|\leq \underbrace{|\widetilde{h}_k(x)-\mathbb{E}_\eta[\widetilde{h}_k(x)]|}_{:=V_x }+\underbrace{|\mathbb{E}_\eta[\widetilde{h}_k(x)]-h(x)|}_{:=b_x}.
\]
Consider firstly the second term, interpreted as a bias and can be upper bounded with probability $1$. It holds that
\[
b_x =\left|\mathbb{E}_\eta[\widetilde{h}_k(x)]-h(x)\right|\leq \frac{1}{|\mathcal{N}_k(x)|}\sum_{\itup: x_\itup \in \mathcal{N}_k(x)}\left|h(x_\itup) -h(x)\right|\leq 2\pi \Lip r_k(x),
\]
where we used Fact~\ref{Fact:Lipschiz} that $\left|h(x_\itup) -h(x)\right|\leq 2\pi \Lip \left\|x_\itup -x\right\|_\infty$.
On the other hand, the first term can be interpreted as a variance term. Let 
\[
V_x = \left| \frac{1}{k}\sum_{\itup\in \mathcal{N}_k(x)} Z_\itup\right|, \text{ with } Z_\itup = \frac{z_\itup}{e^{-2\pi^2\sigma^2}}-h_\itup
\]
 a set of zero-mean independent random variables.
Notice that $|Z_\itup|\leq 1+e^{2\pi^2\sigma^2}:=K$ almost surely and $\text{Var}(Z_\itup) = e^{4\pi^2\sigma^2}-1:= S^2$.
In order to be able to use a concentration result, we split the sum above into real and imaginary parts as follows
\begin{equation}
V_x \leq  \left| \frac{1}{k}\sum_{\itup\in \mathcal{N}_k(x)} \real Z_\itup\right| + \left| \frac{1}{k}\sum_{\itup\in \mathcal{N}_k(x)} \imag Z_\itup\right| = u_x + w_x.\label{eq:Variance_Proof}  
\end{equation}
Each of the two terms above involves a sum of zero-mean independent variables which can be bounded via Bernstein inequality for bounded random variables (cfr. Theorem~\ref{thm:Bernstein} with the change of variables $t\mapsto k t$ in Appendix). Firstly, we notice that $|\real Z_\itup|\leq K $ and $\text{Var}(\real Z_\itup)\leq S^2$.
Then, Bernstein inequality gives
\begin{equation}
\text{Pr}\left(u_x >t  \right)\leq 2 \exp\left(\frac{-kt^2/2}{S^2 + Kt/3}\right),\label{eq:Bernstein}  
\end{equation}
while the same bound holds for $\text{Pr}\left(v_x >t  \right)$. Using the union bound, we then find that 
$$
\text{Pr}\left(u_x >t/2  \text{ or } v_x >t/2\right)\leq 4 \exp\left(\frac{-kt^2/8}{S^2 + Kt/6}\right),
$$
while by taking the complement of this event, we find, thanks to Morgan's law,
$$
\text{Pr}\left(u_x \leq t/2  \text{ and } v_x \leq t/2\right)\geq 1- 4 \exp\left(\frac{-kt^2/8}{S^2 + Kt/6}\right).
$$
Hence, in view of \eqref{eq:Variance_Proof}, the variance $V_x\leq t$ with a probability equal at least to $1- 4 \exp\left(\frac{-kt^2/8}{S^2 + Kt/6}\right)$.
%\\
%\noindent Let $\itup\in [m]^d$. 
The statement~\eqref{eq:in_sample_sup_norm_whp} follows by taking a fixed $x= x_\itup$ for some $\itup\in [m]^d$. We know thanks to Bernstein inequality \eqref{eq:Bernstein}  that $V_{x_\itup}\leq t$ with a probability larger than
$
1-4 \exp\left(\frac{-kt^2/8}{S^2 + Kt/6}\right).
$
This yields a condition on the minimal value for $t>0$ so that the failure probability is smaller than $\delta$. Indeed, thanks to a union bound, we find
\[
\text{Pr}\left(\bigcup_{x_\itup: \itup\in [m]^d }(V_{x_\itup} >t)  \right)\leq 4 n \exp\left(\frac{-kt^2/8}{S^2 + Kt/6}\right).
\]
Now, we redefine the failure probability $\delta$ such that $4 n \exp\left(\frac{-kt^2/8}{S^2 + Kt/6}\right)\leq \delta$, which yields the equivalent condition
\[
\frac{k}{8}t^2-\frac{K}{6}\log\left( \frac{4n}{\delta}\right) t -S^2 \log\left( \frac{4n}{\delta}\right) \geq 0.
\]
This gives, with a probability larger than $1-\delta$,
\[
 V_{x_\itup}\leq \frac{2K}{3k} \log\left( \frac{4n}{\delta}\right) + \sqrt{\left(\frac{2K}{3k} \log\left( \frac{4n}{\delta}\right)\right)^2+\frac{8S^2}{k}\log\left( \frac{4n}{\delta}\right)},\quad \forall \itup\in [m]^d.
\]
Now, we recall that $\left|\widetilde{h}_k(x_\itup)-h(x_\itup)\right| \leq b_{x_\itup} + V_{x_\itup}$, where we can upper bound $b_{x_\itup}\leq 2\pi \Lip \frac{\lceil k^{1/d}\rceil -1}{m-1}$ by using Lemma~\ref{lemma:kNN}.
Then, if $n^{1/d}=m\geq 2$, one obtains
\begin{align*}
\left|\est{h}_k(x_\itup)-h(x_\itup)\right| \leq& 4\pi \Lip \frac{\lceil k^{1/d}\rceil -1}{m-1} +\frac{8K}{3k} \log\left( \frac{4n}{\delta}\right) + 4\sqrt{\frac{2S^2}{k}\log\left( \frac{4n}{\delta}\right)}\\
&\leq 8\pi \Lip \left(\frac{k}{n}\right)^{1/d} + \left(\frac{8K}{3}\frac{\log\left( \frac{4n}{\delta}\right)}{k}+4\sqrt{2S^2}\sqrt{\frac{\log\left( \frac{4n}{\delta}\right)}{k}} \right)
\end{align*}
by using $\sqrt{a+b}\leq\sqrt{a}+\sqrt{b}$ in the first inequality, and then, $m-1\geq m/2 $ and $\lceil k^{1/d}\rceil -1\leq k^{1/d}$. Next, in order to simplify the expressions of $K$ and $S^2$, we use the inequality $\exp (x)\leq 1+2x$ if $x\leq 1$.
Namely, if $2\pi^2\sigma^2\leq 1$, we can upper bound
 $K = 1+e^{2\pi^2\sigma^2}\leq 2+4\pi^2\sigma^2$. Similarly,  $S^2 = e^{4\pi^2\sigma^2}-1\leq 8\pi^2\sigma^2$ if $4\pi^2\sigma^2\leq 1$. Next, we choose $\delta = 1/n$ so that we find
 \begin{align*}
\left|\est{h}_k(x_\itup)-h(x_\itup)\right| \leq 8\pi \Lip \left(\frac{k}{n}\right)^{1/d} + 16 \left(\frac{1}{3}(2\pi^2\sigma^2+1)\frac{\log( 4n^2)}{k}+\pi\sigma\sqrt{\frac{\log( 4n^2)}{k}}\right).
\end{align*}
In order to simplify the bound above,
%we firstly bound $\sqrt{\log(2n^2)}\leq \log(2n^2)$, which is true if $n\geq 2$. Next,
we use the inequality $\log (2n)\leq 2\log n$ for $n\geq 2$.
This finally yields
 \begin{align*}
\left|\est{h}_k(x_\itup)-h(x_\itup)\right| \leq 8\pi \Lip \left(\frac{k}{n}\right)^{1/d} + \frac{64}{3}(2\pi^2\sigma^2+1)\frac{\log n}{k}+32\pi\sigma\sqrt{\frac{\log n}{k}}.
\end{align*}
\end{proof}
The statistical rate of the in-sample $\ell_\infty$ bound is a direct consequence of Theorem~\ref{Thm:whp_l_infty}. 
\begin{corollary}[Statistical rates for $\ell_\infty$ risk -- in sample] \label{cor:knn_insample_rates}
Let $\sigma\leq \frac{1}{2\pi}$ and $n\geq 2^d$.
Let the number of neighbours be $k= \ceil{k^\star}$ with 
$$
k^\star = n^{\frac{2}{d+2}}\left( \log n\right)^{\frac{d}{d+2}}\left(\frac{d(\frac{4\pi^2\sigma^2+2}{3}+\pi\sigma)}{\pi \Lip} \right)^{\frac{2d}{d+2}}.
$$
Provided $\frac{n}{\log n} \geq \left(\frac{\pi \Lip}{2d(\frac{4\pi^2\sigma^2+2}{3}+\pi\sigma)} \right)^{d}$, the following upper bound on the $\ell_{\infty}$ risk in \eqref{eq:in_sample_sup_norm_whp} holds with probability at least $1-1/n$,
\begin{equation}
 \left|\est{h}_k(x_\itup)-h(x_\itup)\right| 
\leq \gamma \left(\frac{\log n}{n}\right)^{\frac{1}{d+2}}, \quad \forall \itup\in [m]^d \label{eq:InSampleRate}
\end{equation}
with $\gamma = 6 (8\pi \Lip)^{\frac{d}{d+2}}\left(32(\frac{4\pi^2\sigma^2+2}{3}+\pi\sigma)\right)^{\frac{2}{d+2}}$. Furthermore, if \eqref{eq:InSampleRate} holds and the RHS of \eqref{eq:InSampleRate} is less than or equal to $2$, then this implies 
\[
d_w\left(\est{g}(x_\itup),g(x_\itup)\right) \leq \frac{\gamma}{4} \left(\frac{\log n}{n}\right)^{\frac{1}{d+2}}, \quad \forall \itup\in [m]^d
\]
where $\est{g}(x) = \frac{1}{2\pi}\arg\big(\est{h}_k(x)\big)$ and $g(x) = \frac{1}{2\pi}\arg\big(h(x)\big)$.
\end{corollary}
\begin{proof}
Assuming $k \geq \log n$, the bound in \eqref{eq:in_sample_sup_norm_whp} simplifies to
\begin{equation*}
 \left|\est{h}_k(x_\itup)-h(x_\itup)\right| \leq 8\pi \Lip \left(\frac{k}{n}\right)^{1/d} + 32\left(\frac{1}{3}(4\pi^2\sigma^2+2)+\pi\sigma\right)\sqrt{\frac{\log n}{k}} = R(k)
\end{equation*}
with $R(k) =  \alpha(n)  k^{1/d} + \beta(n)k^{-1/2}$
where $\alpha(n)= \frac{8\pi \Lip}{n^{1/d}}$ and $\beta(n) =32\left(\frac{4\pi^2\sigma^2+2}{3}+\pi\sigma\right)\sqrt{\log n}$.
Then, the minimization of $R(k)$ with respect to $k$ gives $k^\star = \left(\frac{d\beta(n)}{2  \alpha(n)}\right)^{\frac{2d}{d+2}}$. Then, the lower bound on $\frac{n}{\log n}$ in the statement follows from the requirement that $\left(\frac{d\beta(n)}{2  \alpha(n)}\right)^{\frac{2d}{d+2}} \geq \log n$. Hence, by using $k^\star\leq \ceil{k^\star}\leq 2 k^\star$ where we take $k=\ceil{k^\star}$, we obtain the upper bound
$$
R(k^\star)\leq R(k)\leq \alpha(n) (2k^\star)^{1/d} + \beta(n) (k^{\star})^{-1/2} = \widetilde{R}(k^\star).
$$
By substituting back the expression of $k^\star$ in $\widetilde{R}(k^\star)$, we find
\[
\widetilde{R}(k^\star) = \alpha(n)^{\frac{d}{d+2}}\beta(n)^{\frac{2}{d+2}}\left( 2^{1/d}\left( \frac{d}{2} \right)^{\frac{2}{d+2}} + \left(\frac{2}{d} \right)^{\frac{d}{d+2}}\right) \leq 6 \alpha(n)^{\frac{d}{d+2}}\beta(n)^{\frac{2}{d+2}} = \gamma \left(\frac{\log n}{n}\right)^{\frac{1}{d+2}}
\]
where we used the simplifications $(\frac{d}{2})^{\frac{2}{d+2}} \leq (\frac{d+2}{2})^{\frac{2}{d+2}} \leq 2$  and $(\frac{2}{d}) ^{\frac{d}{d+2}} \leq 2$, as well as $2^{1/d}\leq 2$. 
Finally, note that if $\gamma \left(\frac{\log n}{n}\right)^{\frac{1}{d+2}} \leq 2$, then the stated bound on the wrap-around distance is obtained readily using Fact~\ref{Fact:Wrap-around} in the Appendix.
\end{proof}
\begin{remark}
The restriction $\sigma\leq \frac{1}{2\pi}$ on the noise in the statements of this section is assumed in order to avoid cumbersome expressions. Therefore, similar guarantees can be obtained by relaxing this condition.
\end{remark}
\subsection{Robustly unwrapping the modulo samples} \label{subsec:unwrap_mod}
In this subsection, we discuss and analyze a procedure for unwrapping the denoised mod 1 samples obtained from Algorithm~\ref{algo:kNNregression}. 

As a warm up, let us first look at the univariate case where $f:[0,1] \rightarrow \mathbb{R}$. Consider the unknown ground truth function $g(x) = f(x) \mod 1$ and let $\est{g}: [0,1] \rightarrow [0,1)$ be an estimate of $g$. 
%Recall $d_w:[0,1) \rightarrow [0,1/2]$ to be the usual wrap around metric defined as
%
%\begin{equation*}
% d_w(u,v) := \max \set{\abs{u-v}, 1-\abs{u-v}}.
%\end{equation*}
%
Say $\est{g}(x)$ is close to $g(x)$ for all $x$ on the grid $\calX = \set{x_1,\dots,x_n}$ where $x_i = \frac{i-1}{n-1}$; $i=1,\dots,n$. Formally, for some $\delta \in [0,1/2]$, we assume that $d_w(\est{g}(x_i), g(x_i)) \leq \delta$ holds for all $i$. Given the perturbed estimates $\est{g}(x_i)$, we will now show a \emph{stable} recovery procedure that produces estimates $\ftil(x_i)$ of $f(x_i)$, which satisfy (up to an integer shift) the bound  
\begin{equation*}
\abs{f(x_i) - \ftil(x_i)} \leq \delta, \quad \forall \ i=1,\dots,n.    
\end{equation*}
To begin with, observe that for each $i$, there exists $\eta_i \in [-\delta,\delta]$ such that $\est{g}(x_i) = (f(x_i) + \eta_i) \bmod 1$. Denoting $\fhat(x_i) := f(x_i) + \eta_i$, we will now recover each $\fhat(x_i)$ (up to an integer shift) sequentially by a simple procedure that relies on the following lemma. It can be viewed as an adaptation of the classical Itoh's condition \cite{Itoh_82} from the phase unwrapping literature, to our setup. 
\begin{lemma} \label{lem:itoh_type_cond}
If $2\delta + \frac{\Lip}{n-1} < \frac{1}{2}$, then the following holds true for each $i = 2,\dots,n$.
\begin{equation} \label{eq:fin_diff_rel}
  \fhat(x_i) - \fhat(x_{i-1}) =  \left\{
\begin{array}{rl}
\ghat(x_i) - \ghat(x_{i-1}) \ ; & \text{ if } \abs{\ghat(x_i) - \ghat(x_{i-1})} < 1/2, \\
1 + \ghat(x_i) - \ghat(x_{i-1}) \ ; & \text{ if } \ghat(x_i) - \ghat(x_{i-1}) < -1/2, \\
-1 + \ghat(x_i) - \ghat(x_{i-1}) \ ; & \text{ if } \ghat(x_i) - \ghat(x_{i-1}) > 1/2.
\end{array} \right.
\end{equation}
\end{lemma}
\begin{proof}
Using the Lipschitz continuity of $f$ and triangle inequality, we readily obtain the bound
\begin{equation*}
    \abs{ \fhat(x_i) - \fhat(x_{i-1})} \leq 2\delta + \frac{\Lip}{n-1} < \frac{1}{2}, \quad \forall \ i=1,\dots,n,
\end{equation*}
due to our assumption on $\delta,n$. Now denoting $\est{q}(x_i) \in \mathbb{Z}$ to be the quotient term associated with $\fhat(x_i)$, we arrive at the identity 
\begin{equation} \label{eq:fhat_ghat_rel}
     \fhat(x_i) - \fhat(x_{i-1}) = \qhat(x_i) - \qhat(x_{i-1}) + \ghat(x_i) - \ghat(x_{i-1}).
\end{equation}
Clearly, the bound $\abs{ \fhat(x_i) - \fhat(x_{i-1}) } < 1/2$ implies that $\qhat(x_i) - \qhat(x_{i-1}) \in \set{0,1,-1}$.
\begin{enumerate}
    \item If $\qhat(x_i) = \qhat(x_{i-1})$, then \eqref{eq:fhat_ghat_rel} readily implies $\abs{\ghat(x_i) - \ghat(x_{i-1})} < 1/2$.
    
    \item If $\qhat(x_i) = \qhat(x_{i-1}) + 1$, then $\fhat(x_i) - \fhat(x_{i-1}) \in (0,1/2)$, and so \eqref{eq:fhat_ghat_rel} leads to the bound $\ghat(x_i) - \ghat(x_{i-1}) < -1/2$.
    
    \item If $\qhat(x_i) = \qhat(x_{i-1}) - 1$, then $\fhat(x_i) - \fhat(x_{i-1}) \in (-1/2,0)$, and so \eqref{eq:fhat_ghat_rel} leads to the bound $\ghat(x_i) - \ghat(x_{i-1}) > 1/2$.
\end{enumerate}
Since the above conditions stated on $\ghat(x_i) - \ghat(x_{i-1})$ are all disjoint, the identity in \eqref{eq:fin_diff_rel} follows.
\end{proof}
The above lemma tells us that if $\delta \lesssim 1$ and $n \gtrsim M$, then the finite difference $\fhat(x_i) - \fhat(x_{i-1})$ is determined completely by $\ghat(x_i) - \ghat(x_{i-1})$. Therefore we can recover the estimates $\ftil(x_i)$ as 
\begin{equation} \label{eq:seq_unwrap_proc}
    \ftil(x_1) = \ghat(x_1), \quad \ftil(x_i) = \ftil(x_{i-1}) + 
    \left\{
\begin{array}{rl}
\ghat(x_i) - \ghat(x_{i-1}) \ ; & \text{ if } \abs{\ghat(x_i) - \ghat(x_{i-1})} < 1/2, \\
1 + \ghat(x_i) - \ghat(x_{i-1}) \ ; & \text{ if } \ghat(x_i) - \ghat(x_{i-1}) < -1/2, \\
-1 + \ghat(x_i) - \ghat(x_{i-1}) \ ; & \text{ if } \ghat(x_i) - \ghat(x_{i-1}) > 1/2.
\end{array} \right.
\end{equation}
\begin{remark}
The sequential procedure in \eqref{eq:seq_unwrap_proc} was also considered in \cite{CMT18_long}, however, without any formal analysis. 
\end{remark}
Using Lemma~\ref{lem:itoh_type_cond}, it is easy to derive the following uniform error bound for the estimates $\ftil(x_i)$.
\begin{lemma} \label{lem:err_seq_unwrap}
If $2\delta + \frac{\Lip}{n-1} < \frac{1}{2}$ then there exists $q^{\star} \in \mathbb{Z}$ such that
\begin{equation} \label{eq:err_bd_seq_unwrap}
    \abs{\ftil(x_i) + q^{\star} - f(x_i)} \leq \delta, \quad \forall i=1,\dots,n.
\end{equation}
\end{lemma}
\begin{proof}
Denote $q^{\star} = \qhat(x_1) \in \mathbb{Z}$ to be the quotient term of $\fhat(x_1)$. We will show by induction that $\ftil(x_i) + q^{\star} = \fhat(x_i)$ holds for each $i$. The bound in \eqref{eq:err_bd_seq_unwrap} then follows since $\abs{\fhat(x_i) - f(x_i)} \leq \delta$. 

To show the induction argument based step, note that $\ftil(x_1) + q^{\star} = \fhat(x_1)$ is trivially true. For convenience, denote the term within braces in \eqref{eq:seq_unwrap_proc} by $a_{i,i-1}$. Then for any $i > 1$, we have that 
\begin{align*}
    \ftil(x_i) + q^{\star} 
    &= \ftil(x_{i-1}) + q^{\star} + a_{i,i-1} \qquad \text{ (using \eqref{eq:seq_unwrap_proc}) } \\
    &= \fhat(x_{i-1}) + a_{i,i-1} \qquad \text{ (using the induction hypothesis on $i$) } \\
    &= \fhat(x_i) \qquad \text{ (using Lemma~\ref{lem:itoh_type_cond}) }
\end{align*}
which completes the proof.
\end{proof}
%
%
%----------------------
% The general case
%
\paragraph{The general $d \geq 1$ setting.} We now show that the above discussion generalizes to the multivariate setting where $f: [0,1]^d \rightarrow \mathbb{R}$. For an integer $m > 1$, denote $\calX = \set{x_1,\dots,x_m}^d$ where $x_i = \frac{i-1}{m-1}$ to be the uniform grid, with $\abs{\calX} = n = m^d$. Also denote $\itup = (i_1,\dots,i_d) \in [m]^d$ to be a $d$-tuple, and $x_{\itup} = (x_{i_1},\dots,x_{i_d}) \in [0,1]^d$ where $x_{i_j} = \frac{i_j -1}{m-1}$. Then with $g:[0,1]^d \rightarrow [0,1)$ defined as $g(x) = f(x) \bmod 1$, denote $\ghat:[0,1]^d \rightarrow [0,1)$ to be an estimate of $g$ in the sense that for some $\delta \in [0, 1/2]$,
\begin{equation*}
    d_w(\ghat(x_{\itup}), g(x_{\itup})) \leq \delta \quad \forall \ x_{\itup} \in \calX.
\end{equation*}
This means that for each $\itup \in [m]^d$, there exists $\eta_{\itup} \in [-\delta,\delta]$ such that $\est{g}(x_{\itup}) = (f(x_{\itup}) + \eta_{\itup}) \bmod 1$. Denoting $\est{f}(x_{\itup}) := f(x_{\itup}) + \eta_{\itup}$, we will recover each $\fhat(x_{\itup})$ (up to an integer shift) by a generalization of the procedure in \eqref{eq:seq_unwrap_proc}. For clarity of exposition, let us define the finite difference operator
\begin{equation*}
    \diff_j \fhat(x_{\itup}) := \fhat(x_{i_1},\dots,x_{i_j}, \dots, x_{i_d}) - \fhat(x_{i_1},\dots,x_{i_j-1}, \dots, x_{i_d}); \quad \forall \ j \in [d], \text{ and } \itup \in [m]^d \text{ with } i_j > 1. 
\end{equation*}
We now present the following generalization of Lemma~\ref{lem:itoh_type_cond} to the multivariate setting. 
\begin{lemma} \label{lem:itoh_type_cond_mult}
If $2\delta + \frac{\Lip}{m-1} < \frac{1}{2}$, then the following holds true for each $j \in [d]$, and $\itup \in [m]^d$ with $i_j > 1$.
\begin{equation} \label{eq:fin_diff_rel_mult}
  \diff_j \fhat(x_{\itup}) =  \left\{
\begin{array}{rl}
\diff_j \ghat(x_{\itup}) \ ; & \text{ if } \abs{\diff_j \ghat(x_{\itup})} < 1/2, \\
1 + \diff_j \ghat(x_{\itup}) \ ; & \text{ if } \diff_j \ghat(x_{\itup}) < -1/2, \\
-1 + \diff_j \ghat(x_{\itup}) \ ; & \text{ if } \diff_j \ghat(x_{\itup}) > 1/2.
\end{array} \right.
\end{equation}
\end{lemma}
\begin{proof}
The proof is similar to that of Lemma~\ref{lem:itoh_type_cond}. Indeed, using the Lipschitz continuity of $f$ and triangle inequality, we first have that
\begin{equation*}
    \abs{  \diff_j \fhat(x_{\itup}) } \leq 2\delta + \frac{\Lip}{m-1} < \frac{1}{2}, \quad \text{ for each } \ 1 \leq j \leq d, \text{ and } \itup \in [m]^d \text{ with } i_j > 1.
\end{equation*}
 Now denoting $\est{q}(x_{\itup}) \in \mathbb{Z}$ to be the quotient term associated with $\fhat(x_{\itup})$, note that  
\begin{equation} \label{eq:fhat_ghat_rel_mult}
      \diff_j \fhat(x_{\itup})  =  \diff_j \qhat(x_{\itup})  +  \diff_j \ghat(x_{\itup}) .
\end{equation}
The key observation is that \eqref{eq:fhat_ghat_rel_mult} involves taking a finite difference only along the coordinate $j$, hence the same reasoning as for the univariate setting applies, and it is clear that $\abs{ \diff_j \fhat(x_{\itup}) } < 1/2$ implies $ \diff_j \qhat(x_{\itup}) \in \set{\pm 1,0}$. From \rev{here on}, the rest of the argument is the same as for Lemma~\ref{lem:itoh_type_cond} and hence omitted.
\end{proof}
\begin{remark} \label{rem:itoh_multiv}
When $d = 2$, Lemma ~\ref{lem:itoh_type_cond_mult} is similar in spirit to the generalization of Itoh's condition to the bivariate case, in the phase unwrapping literature (see \cite[Lemma 1.2]{Ying06}). To the best of our knowledge, a general multivariate unwrapping procedure does not exist in the literature.  
\end{remark}
Equipped with the above lemma, we arrive at the procedure in Algorithm~\ref{algo:seq_unwrap_mult} which is a generalization of the procedure in \eqref{eq:seq_unwrap_proc} for recovering $\fhat(x_{\itup})$ at each $x_{\itup}$. 
%
%
%---------------------------------------------------
% Algorithm outline for the Multivariate SPAM case
%---------------------------------------------------
\begin{algorithm*}[!ht]
\caption{Sequentially unwrapping the modulo samples} \label{algo:seq_unwrap_mult} 
\begin{algorithmic}[1] 
\State \textbf{Input:} uniform grid $\calX \subset [0,1]^d$, $\abs{\calX} = n = m^d$; modulo samples $\ghat(x_{\itup}) \in [0,1)$, $\forall x_{\itup} \in \calX$

\State \textbf{Initialization:} $\ftil(0,\dots,0) = \ghat(0,\dots,0)$.

\State \textbf{Output:} $\ftil(x_{\itup})$ $\forall x_{\itup} \in \calX$

\For{$j=1,\dots,d$} %\qquad \textsc{// Estimation of } $\totsupp_i$ 

\State Fix $\itup_{j+1:d} := (i_{j+1}, \dots, i_d) = (1,\dots,1)$.

\For{each $\itup_{1:j-1} := (i_1,\dots,i_{j-1}) \in [m]^{j-1}$}

\For{$i_j = 2,\dots,m$}
\begin{align}
\ftil(x_{\itup_{1:j-1}}, x_{i_j}, x_{\itup_{j+1:d}}) &=  \ftil(x_{\itup_{1:j-1}}, x_{i_j -1}, x_{\itup_{j+1:d}}) \nonumber \\
&+ 
\left\{
\begin{array}{rl}
\diff_j \ghat(x_{\itup_{1:j-1}}, x_{i_j}, x_{\itup_{j+1:d}}) \ ; & \text{ if } \abs{\diff_j\ghat(x_{\itup_{1:j-1}}, x_{i_j}, x_{\itup_{j+1:d}})} < 1/2, \\
1 + \diff_j\ghat(x_{\itup_{1:j-1}}, x_{i_j}, x_{\itup_{j+1:d}}) \ ; & \text{ if } \diff_j \ghat(x_{\itup_{1:j-1}}, x_{i_j}, x_{\itup_{j+1:d}}) < -1/2, \\
-1 + \diff_j\ghat(x_{\itup_{1:j-1}}, x_{i_j}, x_{\itup_{j+1:d}}) \ ; & \text{ if } \diff_j \ghat(x_{\itup_{1:j-1}}, x_{i_j}, x_{\itup_{j+1:d}}) > 1/2.
\end{array} \right. \label{eq:mult_seq_ftil_exp}
\end{align}
\EndFor
\EndFor
\EndFor
\end{algorithmic}
\end{algorithm*}
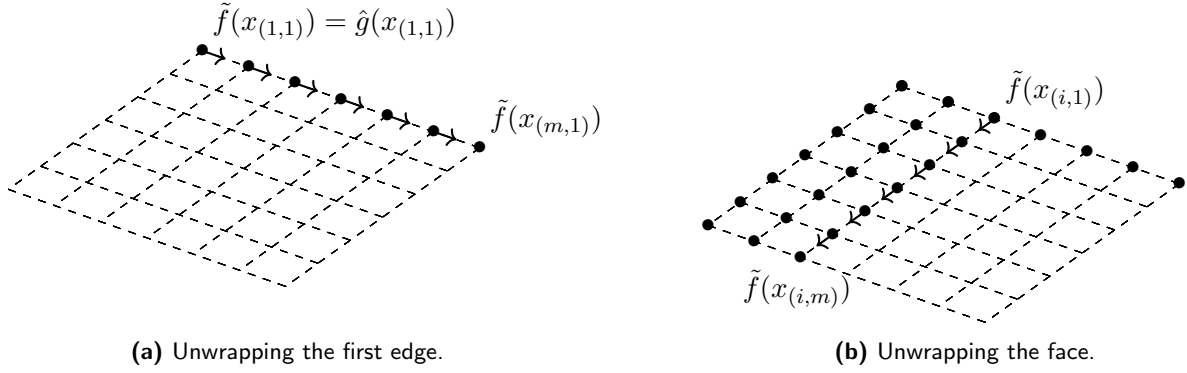
\begin{figure}[ht]
\centering
\begin{subfigure}[b]{0.45\textwidth}
% \begin{adjustbox}{width=\textwidth} % rescale box
\tdplotsetmaincoords{60}{125}
\tdplotsetrotatedcoords{0}{0}{0} %<- rotate around (z,y,z)
\colorlet{lightblue}{blue!30}
\begin{tikzpicture}[scale=1.5,tdplot_rotated_coords,
                    cube/.style={very thick,black},
                    grid/.style={very thin,black},
                    axis/.style={->,blue,ultra thick},
                    rotated axis/.style={->,purple,ultra thick}]

    %draw a grid in the x-y plane
    \foreach \x in {-0.5,0,...,2.5}
    {
        \foreach \y in {-0.5,0,...,2.5}
        {
            \draw[grid,dashed] (\x,-0.5) -- (\x,2.5);
            \draw[grid,dashed] (-0.5,\y) -- (2.5,\y);
        }
    }
    \foreach \y in {-0.5,0,...,2.}
    {
            \draw[->,thick] (-0.5,\y) -- (-0.5,\y+0.25);
    }
    \foreach \y in {-0.5,0,...,2.5}
    {
        \draw (-0.5,\y) node {\textbullet};
    }
    \draw (-0.5,-0.5) node[above right] %{$\tilde{f}(x_{(1,1)})$};
{$\tilde{f}(x_{(1,1)}) = \hat{g}(x_{(1,1)})$};
    \draw (-0.5,2.5) node[above right]{$\tilde{f}(x_{(m,1)})$};

\end{tikzpicture}
%\end{adjustbox}
\caption{Unwrapping the first edge.\label{fig:edge}}
\end{subfigure}
\hfill
\begin{subfigure}[b]{0.45\textwidth}
 %\begin{adjustbox}{width=\textwidth} % rescale box
\tdplotsetmaincoords{60}{125}
\tdplotsetrotatedcoords{0}{0}{0} %<- rotate around (z,y,z)
\colorlet{lightblue}{blue!30}

\begin{tikzpicture}[scale=1.5,tdplot_rotated_coords,
                    cube/.style={very thick,black},
                    grid/.style={very thin,black},
                    axis/.style={->,blue,ultra thick},
                    rotated axis/.style={->,purple,ultra thick}]
                    
    %draw a grid in the x-y plane
    \foreach \x in {-0.5,0,...,2.5}
        \foreach \y in {-0.5,0,...,2.5}
        {
            \draw[grid,dashed] (\x,-0.5) -- (\x,2.5);
            \draw[grid, dashed] (-0.5,\y) -- (2.5,\y);
        }

        \foreach \x in {-0.5,0,...,2.}
    {
            \draw[->,thick] (\x,0.5) -- (\x+0.25,0.5);
    }
    \foreach \x in {-0.5,0,...,2.5}
    {
        \draw (\x,-0.5) node {\textbullet};
        \draw (\x,0.) node {\textbullet};
        \draw (\x,0.5) node {\textbullet};

    }
    \foreach \y in {-0.5,0,...,2.5}
    {
        \draw (-0.5,\y) node {\textbullet};
    }
    \draw (-0.5,0.5) node[above right] {$\tilde{f}(x_{(i,1)})$};
    %\draw (-0.5,2.5) node[above right] {$\tilde{f}(x_{(m,1)})$};
    \draw (2.55,0.5) node[below] {$\tilde{f}(x_{(i,m)})$};
    %\draw (-0.5,-0.5) node[above right] {$ $};
\end{tikzpicture}
%\end{adjustbox}
\caption{Unwrapping the face. \label{fig:face}}
\end{subfigure}
\caption{Illustration of the unwrapping algorithm in two dimensions. The unwrapping starts at the top corner in Figure~\ref{fig:edge}. Each arrow indicates an operation as given in \eqref{eq:mult_seq_ftil_exp} in Algorithm~\ref{algo:seq_unwrap_mult}, while dark dots are unwrapped samples. Once the edge is unwrapped, each row 'rooted' at this edge is unwrapped, as in Figure~\ref{fig:face}. In higher dimensions, the same procedure continues, by starting from the last unwrapped face.}
\end{figure}
Finally, we arrive at the following lemma which provides uniform error bounds for the estimates $\ftil(x_{\itup})$, $\forall x_{\itup} \in \calX$.
\begin{lemma} \label{lem:err_seq_unwrap_mult}
If $2\delta + \frac{\Lip}{m-1} < \frac{1}{2}$ then there exists $q^{\star} \in \mathbb{Z}$ such that
\begin{equation} \label{eq:err_bd_seq_unwrap_mult}
    \abs{\ftil(x_{\itup}) + q^{\star} - f(x_{\itup})} \leq \delta, \quad \forall x_{\itup} \in \calX.
\end{equation}
\end{lemma}
\begin{proof}
The proof is along the same lines as for Lemma~\ref{lem:err_seq_unwrap}. Denote $q^{\star} = \qhat(0,\dots,0) \in \mathbb{Z}$ to be the quotient term of $\fhat(0,\dots,0)$. We will show by induction that for every $1 \leq j \leq d$, 
\begin{equation} \label{eq:mult_ind_st}
    \ftil(x_{\itup_{1:j-1}}, x_{i_j}, x_{\itup_{j+1:d}}) + q^{\star} = \fhat(x_{\itup_{1:j-1}}, x_{i_j}, x_{\itup_{j+1:d}}), \quad \forall \itup \in [m]^{j-1} \times [m] \times \set{1}^{d-j},
\end{equation}
as the bound in \eqref{eq:err_bd_seq_unwrap_mult} then follows readily. Let us denote the term in braces in \eqref{eq:mult_seq_ftil_exp} by $a_{(i_j, i_j-1)}$.
\begin{enumerate}
    \item Consider $j = 1$. We will show by induction that \eqref{eq:mult_ind_st} is true for each $i_j \in [m]$. When $i_1 = 1$, then \eqref{eq:mult_ind_st} is trivially true by construction. When $i_1 > 1$, we have
    \begin{align*}
    \ftil(x_{i_1}, 0,\dots,0) + q^{\star} 
    &= \ftil(x_{i_1 - 1}, 0,\dots,0) + q^{\star} + a_{(i_1, i_1-1)} \qquad \text{ (using \eqref{eq:mult_seq_ftil_exp}) } \\
    &= \fhat(x_{i_1 - 1}, 0,\dots,0) + a_{(i_1, i_1-1)} \ \text{ (using the induction hypothesis on $i_1$) } \\
    &= \fhat(x_{i_1}, 0,\dots,0) \qquad \text{ (using Lemma~\ref{lem:itoh_type_cond_mult}) }.
\end{align*}

\item Now consider any $j > 1$ and assume that \eqref{eq:mult_ind_st} is true ``up to $j-1$''. Then for $i_j = 1$, this implies  
\begin{equation*}
 \ftil(x_{\itup_{1:j-1}}, \underbrace{x_{i_j}}_{= 0}, 0,\dots,0) + q^{\star} = \fhat(x_{\itup_{1:j-1}}, 0, 0,\dots,0)   
\end{equation*}
which verifies \eqref{eq:mult_ind_st}. Now we apply induction on $i_j$. If $i_j > 1$, then
\begin{align*}
    \ftil(x_{\itup_{1:j-1}}, x_{i_j}, 0,\dots,0) + q^{\star} 
    &= \ftil(x_{\itup_{1:j-1}}, x_{i_j-1}, 0,\dots,0) + q^{\star} + a_{(i_j, i_j-1)} \qquad \text{ (using \eqref{eq:mult_seq_ftil_exp}) } \\
    &= \fhat(x_{\itup_{1:j-1}}, x_{i_j-1}, 0,\dots,0) + a_{(i_j, i_j-1)}\\%\text{ (using induction hypothesis on $i_j$) } \\
    &= \fhat(x_{\itup_{1:j-1}}, x_{i_j}, 0,\dots,0) \qquad \text{ (using Lemma~\ref{lem:itoh_type_cond_mult}) },
\end{align*}
where we use for the second equality above the induction hypothesis on $i_j$.
This completes the proof.
\end{enumerate}
\end{proof}
The unwrapping procedure is summarized in Algorithm~\ref{algo:seq_unwrap_mult}. For notational convenience, the lines $5$ and $6$ have to be skipped whenever  $\itup$ is indexed by an empty set.

\rev{ 
\begin{remark}
In the recent literature, other sequential unwrapping algorithms have been proposed in~\cite{bhandari17} for univariate bandlimited functions, in~\cite{HDR_QN} for univariate functions generated by B-splines and in~\cite{ModuloRadon} for bivariate functions. These papers consider taking higher order finite differences of modulo samples and rely heavily on the function $f$ being sufficiently smooth. A main difference with the unwrapping procedure given above is that our approach only relies on first order differences of samples coming from a multivariate Lipschitz function while the aforementioned papers only consider at most bivariate functions with typically more stringent smoothness assumptions.
\end{remark}
}
%

%-----------------------------------
% Main result: Putting it together
%
\subsection{Main result: Putting it together} \label{subsec:main_res_unwrap_samps}
We can combine the results of Corollary~\ref{cor:knn_insample_rates} and Lemma~\ref{lem:err_seq_unwrap_mult} to provide a \rev{statistical} guarantee on the recovering of samples of $f$ given noisy mod $1$ samples by following the procedure described in Algorithm~\ref{algo:Main}.
\begin{theorem}[Main result] \label{thm:main_thm_unwrap_mod1}
Let $\sigma\leq \frac{1}{2\pi}$ and $m^d = n\geq 2^d$ and let \[\delta(n) = 6 (8\pi \Lip)^{\frac{d}{d+2}}\left(32\left(\frac{4\pi^2\sigma^2+2}{3}+\pi\sigma \right)\right)^{\frac{2}{d+2}} \left(\frac{\log n}{n}\right)^{\frac{1}{d+2}}.\]
If $\delta(n) \leq 2$, then with probability at least $1-1/n$, Algorithm~\ref{algo:kNNregression} yields denoised mod $1$ estimates $\est{g}(x_\itup)$ such that
\begin{equation} \label{eq:mod_den_fin_bd}
 d_w\Big(\est{g}(x_\itup) , g(x_\itup)\Big) \leq \frac{1}{4}\delta(n), \quad \itup\in [m]^d.
\end{equation}
Furthermore, if $\delta(n) + \frac{2\Lip}{m-1} < 1$ then there exists $q^{\star} \in \mathbb{Z}$ and $\tilde{f}(x_\itup)$ given by Algorithm~\ref{algo:seq_unwrap_mult}  such that
\begin{equation} \label{eq:err_bd_seq_unwrap_mult_discrete}
    \abs{\ftil(x_{\itup}) + q^{\star} - f(x_{\itup})} \leq \frac{1}{4}\delta(n), \quad \forall \itup \in [m]^d.
\end{equation}
\end{theorem}
\begin{proof}
The result is obtained by using Corollary~\ref{cor:knn_insample_rates} and subsequently, by choosing $\delta$ in Lemma~\ref{lem:err_seq_unwrap_mult} as $\frac{1}{4}\delta(n)$. Note that the condition $\delta(n) \leq 2$ is equivalent to 
\begin{equation*}
    \frac{n}{\log n} \geq 3^{d+2} (8\pi \Lip)^d \left(32(\frac{4\pi^2\sigma^2+2}{3}+\pi\sigma)\right)^2
\end{equation*}
which is stricter than the condition $\frac{n}{\log n} \geq \left(\frac{\pi \Lip}{2d(\frac{4\pi^2\sigma^2+2}{3}+\pi\sigma)} \right)^{d}$ stated in Corollary~\ref{cor:knn_insample_rates}.
\end{proof}
This result indicates indeed that if the number of samples $n$ is large enough, the denoising process yields a sufficiently good estimate of the noiseless mod 1 signal so that the unwrapping procedure achieves a good estimate of the ground truth signal.

\begin{algorithm*}[!ht]
\caption{Denoising and unwrapping modulo samples with kNN regression} \label{algo:Main} 
\begin{algorithmic}[1] 
\State \textbf{Input:} integer $k>0$ and uniform grid $\calX \subset [0,1]^d$, $\abs{\calX} = n = m^d$; noisy modulo samples $y_\itup = (f(x_\itup)+\eta_\itup) \mod 1$ for all $\itup\in [m]^d$.
\State \textbf{Output:} unwrapped denoised modulo samples $\tilde{f}(x_\itup)$ for all $\itup\in[m]^d$.
\State Denoising step: Input $(y_\itup)_{\itup\in [m]^d}$ in Algorithm~\ref{algo:kNNregression} to get denoised mod $1$ samples $(\est{g}(x_\itup))_{ x_{\itup} \in \calX}$.
\State Unwrapping step: Input $(\est{g}(x_\itup))_{ x_\itup \in \calX}$ in Algorithm~\ref{algo:seq_unwrap_mult}  to yield $(\ftil(x_{\itup}))_{ x_{\itup} \in \calX}$.%
\end{algorithmic}
\end{algorithm*}

%-----------------------------------------
% Remark on when denoising ``kicks in''
%-----------------------------------------
\rev{
\begin{remark} 
Denoting $y_{\itup} = (f(x_\itup) + \eta_i) \bmod 1$ to be the noisy input modulo samples where $\eta_{\itup} \sim \mathcal{N}(0,\sigma^2)$ i.i.d for $\itup \in [m]^d$, it is not difficult to show that if $\frac{\log n}{ \sqrt{n}} \lesssim \sigma \lesssim 1$, then w.h.p,
\begin{equation} \label{eq:input_mod1_lw_bd}
    \max_{\itup \in [m]^d} d_w(y_\itup, f(x_\itup) \bmod 1) \geq 
    \frac{\sigma}{2}.
\end{equation}
Comparing \eqref{eq:input_mod1_lw_bd} with \eqref{eq:mod_den_fin_bd}, we see that for a constant noise level $\sigma$ (with $\sigma \lesssim 1$), the bound in \eqref{eq:mod_den_fin_bd} becomes smaller than in \eqref{eq:input_mod1_lw_bd} when $n$ becomes sufficiently large. The same conclusion holds if $\sigma \rightarrow 0$ at a ``sufficiently slow'' rate as $n \rightarrow \infty$. In order to establish \eqref{eq:input_mod1_lw_bd}, we note that $\norm{z - h}_{\infty}^2 \geq \frac{1}{n}\norm{z-h}_2^2$ along with the fact \cite[Lemma 3(iii)]{tyagi20} that $\norm{z-h}_2^2 \geq \pi^2 \sigma^2 n$ holds w.h.p if $\frac{72 \log n}{\pi \sqrt{n}} \leq \sigma \leq \frac{1}{2\sqrt{2} \pi}$. This yields the bound $\norm{z-h}_{\infty} \geq \pi \sigma$. Finally, as seen in the proof of Fact \ref{Fact:Wrap-around}, 
\begin{equation*}
    \frac{1}{2} \abs{z_\itup - h_{\itup}} = \sin (\pi d_w(y_\itup, f(x_\itup) \bmod 1)) \leq \pi d_w(y_\itup, f(x_\itup) \bmod 1), \quad \forall \itup \in [m]^d,
\end{equation*}
which readily yields \eqref{eq:input_mod1_lw_bd}.
\end{remark}
}

%----------------------------
% Estimating the function f
%
\subsection{Estimating the function $f$} \label{subsec:estim_f}
\rev{Given the estimates $\ftil(x_{\itup})$ of $f(x_\itup)$ satisfying the error bound in \eqref{eq:err_bd_seq_unwrap_mult_discrete} uniformly over each $x_{\itup} \in \calX$, we now show that this readily leads to an estimate $\est{f}$ of $f$ with a bound on the error $\norm{\est{f} + q^* - f}_{\infty}$. Here, $\norm{g}_{\infty}$ denotes the $L_{\infty}$ norm for $g \in C([0,1]^d)$, where $C([0,1]^d)$ corresponds to the space of continuous functions with the domain $[0,1]^d$.}
\rev{\paragraph{Quasi-interpolants.} The construction of the estimate $\est{f}$ can be accomplished using classical tools from approximation theory, namely spline based tensor product quasi-interpolant operators \cite{deBoor1990}. These are linear operators $\calQ_m: C([0,1]^d) \rightarrow C([0,1]^d)$ such that $\calQ_m (g)$ depends only on the values of $g$ on the grid $\calX$. While a full discussion on the construction of such operators is outside the scope of the paper, there are many texts providing a detailed overview in this regard (e.g., \cite{devore_book, deBoor1990}). For the present discussion, we highlight certain important properties possessed by $\calQ_m$ which are relevant to our setup. These are adapted from \cite[Section 2]{Devore2011} which also provides a more general discussion on these operators. }

\begin{enumerate}
    \item \rev{There exists an absolute constant $C > 0$ such that 
    \begin{equation} \label{eq:prop1_quasi}
        \norm{\calQ_m(g)}_{\infty} \leq C \max_{x \in \calX} \abs{g(x)}
    \end{equation}
    for each $m=1,2,\dots,$}
    
    \item \rev{If $g$ is a constant, then $\calQ_m(g) \equiv g$.}
    
    \item \rev{Denote $B^s_{p,q}([0,1]^d)$ to be the Besov space of functions $g:[0,1]^d \rightarrow \matR$ with smoothness $s>0$ and $1 \leq p,q \leq \infty$ (see for e.g., \cite{DevorePopov}). There exists $C_{d,s} > 0$ (depending only on $d,s$) such that if $g \in B^s_{\infty,\infty}([0,1]^d)$, then 
        \begin{equation} \label{eq:prop3_quasi_besov}
      \norm{\calQ_m(g) - g}_{\infty} \leq C_{d,s} m^{-s} = C_{d,s} n^{-s/d}.
    \end{equation}
   }
\end{enumerate}

\begin{remark}
\rev{For $\kappa \in \mathbb{N}_0$ and $r \in (0,1]$, the  H\"older space $C^{\kappa, r}([0,1]^d)$ is the collection of functions $g:[0,1]^d \to \mathbb{R}$ for which the partial derivatives $D^\beta g$ are H\"older continuous (for all $|\beta| = \kappa$) with coefficient $0< r \leq 1$, i.e.,
\begin{equation*}
    \abs{D^\beta g (x) - D^\beta g (x')} \leq M\norm{x - x'}^r_{\infty}, \quad \forall x,x' \in [0,1]^d 
\end{equation*}
for some $M > 0$. It is well known, and not difficult to verify, that if $s \geq 1$ is an integer, then $C^{s-1, 1}([0,1]^d) \subseteq B^s_{\infty,\infty}([0,1]^d)$. Moreover, if $s=\kappa + r$ (with $\kappa \in \mathbb{N}_0$, $0 < r < 1$) is not an integer, then $C^{\kappa,r}([0,1]^d) \subseteq B^s_{\infty,\infty}([0,1]^d)$. This is shown in \cite[Chapter 2]{devore_book} for $d=1$ but extends readily to $d \geq 1$ as well. This means that,  for functions $g \in C^{\kappa, r}([0,1]^d)$, \eqref{eq:prop3_quasi_besov} implies that $\norm{\calQ_m(g) - g}_{\infty} \lesssim n^{-\frac{\kappa+r}{d}}$, and in fact, this rate is also optimal (see for e.g., \cite[Section 1.3.9]{Novak1988}).} 

\rev{For the Lipschitz class $C^{0,1}([0,1]^d)$, the above discussion implies that there exists $C_d > 0$ (depending only on $d$) such that if $g \in C^{0, 1}([0,1]^d)$, then
    \begin{equation} \label{eq:prop3_quasi_lip}
      \norm{\calQ_m(g) - g}_{\infty} \leq C_{d} m^{-1} = C_{d} n^{-1/d}.
    \end{equation}
    }
\end{remark}

\rev{\paragraph{Estimating $f$.} Equipped with the above discussion, we can form the estimate $\est{f} = \calQ_m(\ftil)$, where $\ftil \in C([0,1]^d)$ takes the value $\ftil(x_\itup)$ for each $x_\itup \in \calX.$ This leads to the following theorem which states that up to a global shift, $\est{f}$ is uniformly close to $f$.}

\rev{\begin{theorem} \label{thm:main_unwrap_err_f}
Under the notations and assumptions of Theorem \ref{thm:main_thm_unwrap_mod1}, suppose that the bound in \eqref{eq:err_bd_seq_unwrap_mult_discrete} holds. For a quasi-interpolant operator $\calQ_m$, define $\est{f} := \calQ_m(\ftil)$ depending only on the values $\ftil(x_{\itup})$ for $x_\itup \in \calX$. We then have 
\begin{equation*}
    \norm{\est{f} + q^* - f}_{\infty} \leq C_d n^{-1/d} + C\frac{\delta(n)}{4}
\end{equation*}
where $C, C_d$ are as in \eqref{eq:prop1_quasi},  \eqref{eq:prop3_quasi_lip}. 
\end{theorem}}
\begin{proof}
\rev{Let $\ftil, \triangle \in C([0,1]^d)$ be such that $\ftil(x) + q^*= f(x) + \triangle(x)$ for all $x \in [0,1]^d$, with $\abs{\triangle(x)} \leq \frac{\delta(n)}{4}$ for each $x \in \calX$. Then by linearity of $\calQ_m$, and since $\calQ_m(q^*) = q^*$, we have
\begin{equation*}
    \calQ_m(\ftil) + q^* = \calQ_m(f) + \calQ_m(\triangle).
\end{equation*}
Subtracting $f$ from both sides, and applying triangle inequality, the statement follows easily using \eqref{eq:prop1_quasi},  \eqref{eq:prop3_quasi_lip}.
}
\end{proof}
\rev{\begin{remark}
Since $\delta(n)$ dominates the $n^{-1/d}$ term, the error bound in Theorem \ref{thm:main_unwrap_err_f} can be written in the simplified manner
\begin{equation*}
 \norm{\est{f} + q^* - f}_{\infty} \leq C(d,\sigma,M) \left(\frac{\log n}{n}\right)^{\frac{1}{d+2}}   
\end{equation*}
Hence the function $f$ is estimated at the minimax optimal rate for Lipschitz functions in the $L_{\infty}$ norm over the cube $[0,1]^d$ (using a uniform grid), see for e.g., \cite[Theorem 1.3.1]{nemirovski2000topics}.
\end{remark}
}

\section{Denoising mod $1$ samples on a graph with an SDP relaxation} \label{sec:sdp_analysis}
This section formally analyzes a SDP approach for denoising modulo $1$ samples which was recently proposed by Cucuringu and Tyagi \cite{CMT18_long}. We begin by formally outlining the problem setup.
\subsection{Problem setup} \label{subsec:prob_sdp}
Let $G = ([n], E)$ be a connected, undirected graph where $E \subseteq \set{\set{i,j}: i \neq j \in [n]}$ denotes its set of edges. Denote the degree of vertex $p$ by $d_p$, the maximum degree of $G$ by $\triangle := \max_{p} d_p$, and the (combinatorial) Laplacian matrix associated with $G$ by $L \in \matR^{n \times n}$. Let $h \in \bbT_n$ be an unknown ground truth signal which is smooth w.r.t $G$ in the sense that 
\begin{align} 
\smooth_n := \max_{\set{i,j} \in E} \abs{h_i - h_j} \label{eq:hsmooth}  
\end{align}
is ``small'', where $\smooth_n > 0$ depends on $n$. Ideally, we will be interested in the setting where $\smooth_n \rightarrow 0$ as $n \rightarrow \infty$. For example, in the setting of Section \ref{sec:knn_analysis} with $d = 1$, we may choose $G$ to be the path graph where $E = \set{\set{i,i+1}: i=1,\dots,n-1}$. Then $\triangle = 2$ and $\smooth_n = 2\pi L / n$, the latter obtained from Fact \ref{Fact:Lipschiz}.

Given information about $h$ in the form of noisy $z \in \bbT_n$ (cfr. \eqref{eq:z}), our goal is to identify conditions under which \eqref{prog:sdp_relax} is a \emph{tight} relaxation of \eqref{prog:qcqp}, i.e., the solution to \eqref{prog:sdp_relax} is of rank $1$. As we will see below, this will lead to the global solution $\gest$ of \eqref{prog:qcqp}, and in particular, we would have obtained $\gest$ in polynomial time. 
\begin{remark}
 In the notation of Section \ref{sec:knn_analysis}, Cucuringu and Tyagi \cite{CMT18_long} considered a specific class of graphs $G = ([m]^d, E)$ where
$E = \set{\set{\itup,\jtup}: \itup,\jtup \in [m]^d, \quad \itup \neq \jtup, \quad \norm{\itup - \jtup}_{\infty} \leq k}$, for some integer $k > 0$. 
For such graphs, we have $\triangle = (2k+1)^d - 1$. Moreover, no analysis was provided concerning the tightness of \eqref{prog:sdp_relax}.
\end{remark}
% 
%
%\item The next problem is a trust region subproblem (TRS) which relaxes the constraint set of \eqref{prog:qcqp} to a sphere.
%
%\begin{align} 
%\min_{g \in \mathbb{C}^n: \norm{g}_2^2 = n} \lambda g^* L g - 2\real(g^* z). \label{prog:trs} \tag{$\text{TRS}$}
%\end{align}
%
%By separating the real and imaginary parts of the complex variables involved, one can rewrite $\eqref{prog:trs}$ as a trust region subproblem in the domain $\matR^{2n}$. On account of its special structure, the trust region subproblem can be solved in polynomial time with a rich body of work (see for e.g. \cite{Sorensen82, MoreSoren83, Hager01, Naka17}).

The SDP relaxation of \eqref{prog:qcqp} can be derived\footnote{The steps leading to the relaxation are explained more clearly here as opposed to \cite{CMT18_long}.} as follows. Denoting
\begin{eqnarray} \label{eq:Tdef}
T = \begin{pmatrix}
  \lambda L \quad & -z \\ 
  -z^* \quad & 0  
 \end{pmatrix} , \quad
W = \begin{pmatrix}
  gg^* \quad & g \\ 
  g^* \quad & 1  
\end{pmatrix}; \ g \in \mathbb{T}_n,
\end{eqnarray}
the objective of (QCQP) is simply $\Tr(TW)$. Note that $W$ is a rank-$1$ positive semi-definite (p.s.d) matrix with $W_{ii} = 1$ for each $i$. In fact, it is easy to see that any rank-$1$, p.s.d matrix $W' \in \mathbb{C}^{(n+1) \times (n+1)}$ with $W'_{ii} = 1$ will be of the form
\begin{eqnarray} \label{eq:rank_1_psd_form}
W' = \begin{pmatrix}
  g' g'^{*} \quad & g' \\ 
  g'^{*} \quad & 1  
\end{pmatrix}; \quad g' \in \mathbb{T}_n.
\end{eqnarray}
Hence \eqref{prog:qcqp} is equivalent to 
\begin{align} 
\min_{W \in \mathbb{C}^{(n+1) \times (n+1)}} \Tr(T W) \quad \text{ s.t } \quad W \succeq 0, \ \text{rank}(W) = 1, \ W_{ii} = 1 \label{prog:sdp_hard} 
\end{align}
and we obtain a solution $g$ of \eqref{prog:qcqp} as the first $n$ entries of the last column of a solution $W$ of \eqref{prog:sdp_hard}. 
Problem \eqref{prog:sdp_hard} is non-convex (due to the rank constraint) and is in general NP-hard if $T$ is an arbitrary\footnote{In our case, $T$ is a specific matrix involving the Laplacian $L$ so it is not clear if it is NP hard.} Hermitian matrix \cite[Proposition 3.3]{zhang06}. By dropping the rank constraint, we finally arrive at \eqref{prog:sdp_relax}. 

Due to the equivalence of \eqref{prog:sdp_hard} and \eqref{prog:qcqp}, we can see that if $X$ is a solution of \eqref{prog:sdp_hard} and has rank $1$, then it will be of the form in \eqref{eq:rank_1_psd_form} where $g'$ is a solution of \eqref{prog:qcqp}.

% l_inf bound for QCQP
%------------------------------
% l_inf bound for (QCQP)
%------------------------------
\subsection{$\ell_{\infty}$ error bound for \eqref{prog:qcqp}}
To begin with, we will prove the following $\ell_{\infty}$ stability bound for any solution $\gest \in \bbT_n$ of \eqref{prog:qcqp}. This stability result will then be employed later in Section~\ref{sec:sdp_tightness} for proving the main result, namely Theorem~\ref{thm:main_sdp_relax_tight}.
\begin{theorem} \label{thm:main_linf_qcqp}
Assume that the observation $z \in \bbT_n$ satisfies $\norm{z - h}_{\infty} \leq \delta$. If $\lambda \triangle < \sqrt{2}$ holds, then any solution $\gest$ of \eqref{prog:qcqp} satisfies the bound
\begin{align*}
 \norm{\gest - h}_{\infty}^2 \leq \frac{2\delta + \delta^2 + \lambda\triangle(\smooth_n^2 + \sqrt{2})}{1-\frac{\lambda\triangle}{\sqrt{2}}}.    
\end{align*}
\end{theorem}
\begin{proof}[Proof of Theorem~\ref{thm:main_linf_qcqp}]
For any given $p \in [n]$, consider $\gtil \in \bbT_n$ of the form\footnote{This idea of constructing $\gtil$ is taken from the proof of Lemma $4.2$ in \cite{Bandeira2017}.} $\gtil = \gest + (h_p - \gest_p) e_p$, i.e., 
\begin{equation*}
\gtil_q = \left\{
\begin{array}{rl}
h_q \quad & \text{if } q=p, \\
\gest_q \quad & \text{otherwise,}
\end{array} \right. \quad \text{for} \quad q=1,\dots,n.
\end{equation*}
Clearly $\gtil$ is feasible for \eqref{prog:qcqp}. Since $\gest$ is optimal, we obtain the inequality
\begin{align} \label{eq:linf_qcqp_1}
 \lambda \gest^* L \gest - 2\real(\gest^* z) &\leq \lambda \gtil^* L \gtil - 2\real(\gtil^* z)  \nonumber \\
 \Leftrightarrow \quad \lambda(\gest^* L \gest  - \gtil^* L \gtil) &\leq 2 \real((\gest - \gtil)^* z).
\end{align}
The LHS of \eqref{eq:linf_qcqp_1} can be simplified as 
\begin{align*}
  &\lambda \sum_{\set{i,j} \in E} (\abs{\gest_i - \gest_j}^2 - \abs{\gtil_i - \gtil_j}^2)  \\
  &=   \lambda \sum_{\set{i,j} \in E: p \not\in \set{i,j}} (\abs{\gest_i - \gest_j}^2 - \abs{\gtil_i - \gtil_j}^2)  
  + \lambda \sum_{j: \set{p,j} \in E} (\abs{\gest_p - \gest_j}^2 - \abs{\gtil_p - \gtil_j}^2) \\
  &= \lambda \sum_{\set{i,j} \in E: p \not\in \set{i,j}} (\abs{\gest_i - \gest_j}^2 - \abs{\gest_i - \gest_j}^2)  
  + \lambda \sum_{j: \set{p,j} \in E} (\abs{\gest_p - \gest_j}^2 - \abs{\gtil_p - \gtil_j}^2) \\
  &= 2 \lambda \sum_{j: \set{p,j} \in E} \real((h_p - \gest_p)^* \gest_j)
\end{align*}
where the second equality uses the definition of $\gtil$, and the final equality follows from simple algebra. Since the RHS of \eqref{eq:linf_qcqp_1} equals $2\real((\gest_p - h_p)^* z_p)$, hence \eqref{eq:linf_qcqp_1} is equivalent to
\begin{equation} \label{eq:linf_qcqp_2}
 \lambda \sum_{j: \set{p,j} \in E} \real((h_p - \gest_p)^* \gest_j) \leq  \real((\gest_p - h_p)^* z_p).
\end{equation}
Not let us denote $\varepsilon \in [-1,1]$ to be the largest number such that $\real(\gest_i^{*} h_i) \geq \varepsilon$ holds for all $i=1,\dots,n$. Since $\real(\gest_i^{*} h_i) = 1 - \frac{\abs{\gest_i - h_i}^2}{2}$, hence $\abs{\gest_i - h_i}^2 \leq 2(1-\varepsilon)$ holds for each $i$. Our goal is to now provide simplified lower and upper bounds on the LHS and RHS of \eqref{eq:linf_qcqp_2} respectively. In order to obtain the lower bound, we note that 
\begin{align}
 & \ \lambda \sum_{j: \set{p,j} \in E} \real(h_p^* \gest_j) - \lambda \sum_{j: \set{p,j} \in E} \real(\gest_p^* \gest_j) \nonumber \\
 &=  \lambda \sum_{j: \set{p,j} \in E} \real(h_p^* h_j) +  \lambda \sum_{j: \set{p,j} \in E} \real(h_p^* (\gest_j - h_j)) -  \lambda \sum_{j: \set{p,j} \in E} \real(\gest_p^* \gest_j) \nonumber \\
 &\geq  \lambda \sum_{j: \set{p,j} \in E} \real(h_p^* h_j) - \lambda d_p - \lambda d_p \sqrt{2(1-\varepsilon)} \label{eq:linf_qcqp_3}
\end{align}
where the last inequality follows from $\real(\gest_p^* \gest_j) \leq 1$ and $\real(h_p^* (\gest_j - h_j)) \geq - \abs{\gest_j - h_j} \geq -\sqrt{2(1-\varepsilon)}$. Plugging the bound $\real(h_p^* h_j) = 1 - \frac{\abs{h_p - h_j}^2}{2} \geq 1 - \frac{\smooth_n^2}{2}$ in \eqref{eq:linf_qcqp_3}, we obtain the lower bound 
\begin{align}
\lambda d_p \left(1 - \frac{\smooth_n^2}{2} \right) - \lambda d_p - \lambda d_p \sqrt{2(1-\varepsilon)} 
= - \lambda d_p \left(\frac{\smooth_n^2}{2} + \sqrt{2(1-\varepsilon)}\right).  \label{eq:linf_qcqp_4}
\end{align}
The upper bound on the RHS of \eqref{eq:linf_qcqp_2} follows readily as shown below. 
\begin{align}
\real((\gest_p - h_p)^* z_p) 
&= \real(\gest_p^* (z_p - h_p)) + \real(\gest_p^* h_p) - \real(h_p^* z_p) \nonumber \\
&= \underbrace{\real(\gest_p^* (z_p - h_p))}_{\leq \abs{z_p - h_p} \leq \delta} + \real(\gest_p^* h_p) - \left( 1 - \underbrace{\frac{\abs{h_p - z_p}^2}{2}}_{\leq \delta^2/2} \right) \nonumber \\
&\leq \real(\gest_p^* h_p) - 1  + \delta + \frac{\delta^2}{2}. \label{eq:linf_qcqp_5}
\end{align}
Applying \eqref{eq:linf_qcqp_4}, \eqref{eq:linf_qcqp_5} in \eqref{eq:linf_qcqp_2}, we obtain
\begin{align}
  \real(\gest_p^* h_p) 
  &\geq 1  - \delta - \frac{\delta^2}{2} - \lambda d_p \left(\frac{\smooth_n^2}{2} + \sqrt{2(1-\varepsilon)} \right)  \nonumber \\
  &\geq 1  - \delta - \frac{\delta^2}{2} - \lambda \triangle \left(\frac{\smooth_n^2}{2} + \sqrt{2}(1-\frac{\varepsilon}{2}) \right) \quad \text{(Since $d_p \leq \triangle$ and $\sqrt{1-\varepsilon} \leq 1-\frac{\varepsilon}{2}$)}. \label{eq:linf_qcqp_6}
\end{align}
Since $\varepsilon$ is the largest lower bound holding uniformly for $\real(\gest_p^* h_p)$ for each $p=1,\dots,n$, hence $\varepsilon$ must be larger than the RHS of \eqref{eq:linf_qcqp_6}. If furthermore $\lambda \triangle < \sqrt{2}$, we obtain
\begin{align*}
  \varepsilon
  &\geq  1  - \delta - \frac{\delta^2}{2} - \lambda \triangle \left(\frac{\smooth_n^2}{2} + \sqrt{2}(1-\frac{\varepsilon}{2}) \right) \\
  \Leftrightarrow  \varepsilon  &\geq \frac{1  - \delta - \frac{\delta^2}{2} - \lambda \triangle \left(\frac{\smooth_n^2}{2} + \sqrt{2}\right)}{1 - \frac{\lambda \triangle}{\sqrt{2}}}
\end{align*}
and the stated bound on $\norm{\gest - h}_{\infty}^2$ follows since $\abs{\gest_i - h_i}^2 \leq 2(1-\varepsilon)$ holds for each $i=1,\dots,n$.
\end{proof}
\begin{remark} \label{rem:qcqp_linf_bd}
The above result is a stability result for the  solution $\gest$ of \eqref{prog:qcqp} and states that if $\lambda \lesssim 1/\triangle$ then $\norm{\gest-h}_{\infty} \lesssim \sqrt{\delta} + \sqrt{\lambda \triangle}$. The bound is admittedly not satisfactory, since, as seen from Theorem~\ref{thm:main_linf_qcqp}, it does not really satisfy the ``denoising property'' $\norm{\gest - h}_{\infty} < \norm{z-h}_{\infty}$ which is what one would ideally like to prove (since \eqref{prog:qcqp} is denoising $z$). The problem lies in the proof technique where we do not really make use of any particular property of $\gest$, but rather, use a feasibility argument with a suitably constructed feasible point $\gtil$. In the next section, we will see certain optimality conditions necessarily satisfied by $\gest$, however it is unclear how they can be employed to yield better bounds. 
\end{remark}

% SDP tightness proof
%------------------------------
% SDP tightness proof
%------------------------------
\subsection{Tightness of the SDP relaxation of \eqref{prog:qcqp}} \label{sec:sdp_tightness}
We will now derive conditions under which \eqref{prog:sdp_relax} is a tight relaxation of \eqref{prog:qcqp}, i.e., the solution of \eqref{prog:sdp_relax} is a rank-$1$ matrix. The main result of this section is stated as the following Theorem.
\begin{theorem} \label{thm:main_sdp_relax_tight}
Assume that the observation $z \in \bbT_n$ satisfies $\norm{z-h}_{\infty} \leq \delta$. If $\lambda,\triangle,\delta$ satisfy 
\begin{enumerate}
    \item $\delta + \sqrt{\frac{8}{7}(3\delta +  \lambda\triangle (\smooth_n^2 + \sqrt{2}))} \leq \frac{\sqrt{2}}{3}$, and 
    
    \item $\lambda \triangle \leq \frac{1}{8}$,
\end{enumerate}
then \eqref{prog:sdp_relax} has a unique solution $X \in \mathbb{C}^{(n+1) \times (n+1)}$ where 
\begin{align*}
X = \begin{pmatrix}
  \gest \gest^{*} \quad & \gest \\ 
  \gest^{*} \quad & 1  
\end{pmatrix} 
=
\begin{pmatrix} \gest \\ 1 \end{pmatrix}
\begin{pmatrix} \gest^* & 1 \end{pmatrix}; \quad \gest \in \bbT_n.
\end{align*}
Consequently, $\gest$ is the unique solution of \eqref{prog:qcqp}.
\end{theorem}
Our technique will essentially follow the idea proposed by Bandeira \textit{et al.} \cite{Bandeira2017} in the context of the tightness of the SDP relaxation of the MLE for the phase synchronization problem, and is detailed in the ensuing sections. The first step of the proof is to identify the KKT conditions which are satisfied by any global minimizer $X$ of \eqref{prog:sdp_relax} and its dual (see Lemma \ref{lemma:kkt_sdp}). This necessitates constructing a dual feasible matrix $\est{S} \in \mathbb{C}^{(n+1) \times (n+1)}$ satisfying $\est{S} X = 0$. The second step involves identifying the first order optimality conditions of \eqref{prog:qcqp} which are necessarily satisfied by any (local) minimizer of \eqref{prog:qcqp} (and hence by a global minimizer $\gest$ as well). This enables us to guess the form of $\est{S}$, which itself depends on $\gest$, and satisfies all except one KKT condition, namely positive semi-definiteness. In the third step, we show in Lemma \ref{lem:dual_cert_rank_n_psd} that if $\delta$ (noise level) and $\triangle \lambda$ are respectively sufficiently small, then $\rank(\est{S}) = n$ and $\est{S} \succeq 0$. This implies (from Lemma \ref{lemma:kkt_sdp}) that the solution $X$ of \eqref{prog:sdp_relax} is unique and has rank $1$.     
%
%------------------------------------------------------------
% KKT conditions for \eqref{prog:sdp_relax} and its dual
%------------------------------------------------------------
\subsubsection{KKT conditions for \eqref{prog:sdp_relax} and its dual} \label{subsec:kkt_conditions}
To begin with, we will need the following Lemma from \cite{Bandeira2017} (adapted to our setup) which states the KKT conditions for \eqref{prog:sdp_relax} and its dual.
\begin{lemma}[{\cite[Lemma 4.3]{Bandeira2017}}] \label{lemma:kkt_sdp}
A Hermitian matrix $X \in \mathbb{C}^{(n+1) \times (n+1)}$ is a global minimizer of \eqref{prog:sdp_relax} iff there exists a Hermitian matrix $\est{S} \in \mathbb{C}^{(n+1) \times (n+1)}$ such that 
\begin{enumerate}
    \item $X_{ii} = 1$ for all $i=1,\dots,n+1$; \label{kkt_cond_prim_a}
    \item $X \succeq 0$; \label{kkt_cond_prim_b}
    \item $\est{S} X = 0$; \label{kkt_cond_com_slack}
    \item $\est{S} - T$ is (real) diagonal; \label{kkt_cond_dual_feas_a}
    \item $\est{S} \succeq 0$. \label{kkt_cond_dual_feas_b}
\end{enumerate}
If furthermore $\rank(\est{S}) = n$, then $X$ has rank one and is the unique global minimizer of \eqref{prog:sdp_relax}.
\end{lemma}
Conditions \eqref{kkt_cond_prim_a}, \eqref{kkt_cond_prim_b} (resp. \eqref{kkt_cond_dual_feas_a}, \eqref{kkt_cond_dual_feas_b}) are primal (resp. dual) feasibility conditions, while condition \eqref{kkt_cond_com_slack} is the complementary slackness condition.
\subsubsection{Constructing the dual certificate $\est{S}$}
We will now derive the first order optimality conditions that are necessarily satisfied by any (local) minimizer of \eqref{prog:qcqp}. Since every $u = u_1 + \iota u_2 \in \mathbb{C}$ is uniquely identified by $(u_1, u_2) \in \mathbb{R}^2$ we can endow $\mathbb{C}$ with the Euclidean metric $\dotprod{u}{v} = \real(u^* v)$. Moreover, $\bbT_1$ is a submanifold of $\mathbb{C}$ with the tangent space at each $u \in \bbT_1$ given by
\begin{align*}
    T_u \bbT_1 = \set{\util \in \mathbb{C}: \dotprod{\util}{u} = 0}.
\end{align*}
Thus $\bbT_1$ (resp. $\bbT_n$) is a smooth Riemannian submanifold of $\mathbb{C}$ (resp. $\mathbb{C}^n$). For $x \in \bbT_n$, the tangent space of $\bbT_n$ at $x$ is given by 
\begin{align*}
    T_x \bbT_n = \set{\xtil \in \mathbb{C}^n: \real(\diag(\xtil x^*)) = 0}.
\end{align*}
Denoting $F(g) = \lambda g^* L g - 2\real(g^* z)$, we would like to minimize $F$ over $\bbT_n$. Following \cite{Bandeira2017}, let us define the orthogonal projection operator $\proj_x: \mathbb{C}^n \rightarrow T_x \bbT_n$ as 
\begin{equation}
    \proj_x(\xtil) = \xtil - \real(\diag(\xtil x^*)) x.\label{eq:proj}
\end{equation}
Denoting $\grad F(g) := \proj_g \nabla F(g)$ to be the Riemannian gradient of $F$ at $g \in \bbT_n$, where $\nabla F(g)$ is the usual Euclidean gradient of $F$ with respect to $(\real (g), \imag (g))^\top$. Then, for any local minimizer $g$ of \eqref{prog:qcqp}, it holds that 
\begin{align*}
    \grad F(g) = \proj_g \nabla F(g) = 0,
\end{align*}
which is a necessary first order optimality condition for \eqref{prog:qcqp}. One can easily verify that $\nabla F(g) = 2(\lambda L g - z)$, and hence the global minimizer $\gest$ of $\eqref{prog:qcqp}$ satisfies $\proj_{\gest} (\lambda L \gest - z) = 0$, or equivalently 
\begin{align} \label{eq:first_order_1}
 \left(\begin{pmatrix}
  \lambda L \quad & -z \\ 
  -z^* \quad & 0  
 \end{pmatrix}  
 - \real \left( \diag\left(
 \begin{pmatrix}
  \lambda L \quad & -z \\ 
  -z^* \quad & 0  
 \end{pmatrix}
 \begin{pmatrix}
 \gest \\  1  
 \end{pmatrix}
 \begin{pmatrix}
  \gest^* \quad & 1 
 \end{pmatrix}
 \right)
 \right)
 \right)
  \begin{pmatrix}
  \gest \\  1  
 \end{pmatrix} = 0.
\end{align}
Recall the definition of the matrix $T$ in \eqref{eq:Tdef}. Then denoting 
\begin{align} \label{eq:dual_cert_def}
\gesttil =  \begin{pmatrix} \gest \\  1 
\end{pmatrix} \text{ and } 
\quad \est{S} = T - \real(\diag(T \gesttil \gesttil^*)),
\end{align}
\eqref{eq:first_order_1} can be written as $\est{S} \gesttil = 0$. Now denoting $X = \gesttil \gesttil^* \in \mathbb{C}^{(n+1) \times (n+1)}$, clearly 
\begin{itemize}
    \item $X$ is primal feasible for \eqref{prog:sdp_relax} (conditions \eqref{kkt_cond_prim_a}, \eqref{kkt_cond_prim_b} of Lemma \ref{lemma:kkt_sdp});
    
    \item $\est{S} X = 0$  (condition \eqref{kkt_cond_com_slack} of Lemma \ref{lemma:kkt_sdp}), and 
    
    \item  $\est{S} - T$ is a real, diagonal matrix (condition \eqref{kkt_cond_dual_feas_a} of Lemma \ref{lemma:kkt_sdp}).
\end{itemize}
Hence setting $\est{S}$ as in \eqref{eq:dual_cert_def} to be our dual certificate candidate, we now only need to find conditions under which $\est{S} \succeq 0$ and is of rank $n$ since from Lemma \ref{lemma:kkt_sdp} this would imply $X$ is the unique solution of \eqref{prog:sdp_relax}. Consequently, $\gest$ will be the unique solution of $\eqref{prog:qcqp}$. 
%We remark that in contrast to \cite{Bandeira2017}, we do not require second-order optimality conditions for \eqref{prog:qcqp}. 
%
%
\begin{lemma}[Properties of critical points]\label{lemma:Opt}
Let $L = \diag(W \bm{1}) - W$ where $W$ is the adjacency matrix of the graph.
If $\gest$ is a first order critical point of (QCQP) then, the following statements hold.
\begin{enumerate}
    \item[(i)] $z^* \gest$ is real, and
    \item[(ii)] $\ghat^*_i  ( z +  \lambda W\ghat)_i$ is real, for all $1\leq i\leq n$.
\end{enumerate}
Furthermore, let the real symmetric  matrix $W_{\ghat} =  W \circ \real(\ghat\ghat^*)$. Then, if $\gest$ is a second order critical point of (QCQP), we have
\begin{enumerate}
    \item[(iii)] $u^\top \left\{\real\diag(z\ghat^*) + \lambda \left[\diag(W_{\ghat} \bm{1}) - W_{\ghat} \right]\right\} u \geq 0$, for all $u\in \mathbb{R}^n$.
\end{enumerate}
\end{lemma}
\begin{proof}
The first order condition \eqref{eq:first_order_1} can be written as follows
\[
\lambda L \gest - \real(\diag(\lambda L\gest \gest^*))\gest +\real(\diag(z\gest^*))\gest  = z.
\]
Then, $(i)$  and $(ii)$ are obtained, respectively, by multiplying the above expression by $\gest^*$ and $\gest^{*}_i e_i$. Next, we rely on the second order condition (see Proposition \ref{prop:SecondOrder} in appendix)
\[
\langle \gdot , \Hess F(\ghat) [\gdot]\rangle \geq 0 \text{ for all } \gdot\in T_{\ghat}\bbT_n,
\]
with
$
\Hess F(g)[\gdot ] = 2\left\{ \lambda L \gdot -\real[\diag\left((\lambda Lg-z)g^{*}\right)] \gdot \right\}.
$
Then, we parametrize $\gdot\in T_{\ghat}\bbT_n$ as $\gdot = \rmi \diag(u) \ghat$ where $u\in \mathbb{R}^n$. Consequently, since $\diag(u) \ghat = \diag(\ghat) u$ and thanks to (ii),  we have the equivalent expression
\begin{align*}
\langle \gdot , \Hess F(\ghat) [\gdot]\rangle &= 2\real \left \{ u^\top \diag(\ghat^{*}) \left[ \real\diag(z \ghat^{*} + \lambda W\ghat \ghat^{*}) -\lambda W  \right] \diag(\ghat) u\right\}\\
&= 2\real \left \{ u^\top  \left[ \real\diag(z \ghat^{*} + \lambda W\ghat \ghat^{*}) -\lambda \diag(\ghat^{*})W\diag(\ghat)  \right]  u\right\}.
\end{align*}
Then, the condition becomes
\begin{align*}
\real \left \{ u^\top  \left[ \real\diag(z \ghat^{*}) + \real\diag(\lambda\diag(\ghat^{*})W\diag(\ghat)\bm{1}) -\lambda \diag(\ghat^{*})W\diag(\ghat)  \right]  u\right\}\geq 0
\end{align*}
for all $u\in \mathbb{R}^n$ and (iii) follows.
\end{proof}
\begin{corollary}
If $\ghat$ is a second order critical point of (QCQP), then it holds that  $z^{*} \gest \geq 0$ and $\ghat^{*}_i  ( z + \lambda W\ghat)_i\geq 0,$
for all $1\leq i\leq n$.
\label{coro:second_order}
\end{corollary}
\begin{proof}
The inequality $z^{*} \gest \geq 0$ follows from (iii) of Lemma \ref{lemma:Opt}, by taking $u = \bm{1}$. While the second inequality follows by using the positivity of the diagonal in (iii) and  the statement (ii).
\end{proof}
\begin{remark}
Although Lemma \ref{lemma:Opt} states that $z^{*} \gest$ is real whenever $\gest$ is a first order critical point, the quantity  $z^{*}_i \gest_i$, with $1\leq i\leq n$,  is not necessarily real. Similarly, if $\gest$ is a second order critical point, we have, thanks to Corollary \ref{coro:second_order}, that $\ghat^{*}_i  z_i + \lambda\ghat^{*}_i( W\ghat)_i\geq 0,$
for all $1\leq i\leq n$. However, again, the first term in the latter inequality is not necessarily real.
\end{remark}
\subsubsection{Establishing conditions under which $\rank(\est{S}) = n$ and $\est{S} \succeq 0$}
We will now establish conditions under which the dual certificate candidate $\est{S}$ as in \eqref{eq:dual_cert_def} is positive semidefinite and of rank $n$. The following Lemma states that if the noise level $\delta$, and the term $\lambda \triangle$ are respectively small, then $\rank(\est{S}) = n$ and $\est{S} \succeq 0$ implying that the solution of \eqref{prog:sdp_relax} is a unique rank-$1$ matrix (and hence, \eqref{prog:sdp_relax} is a tight relaxation of \eqref{prog:qcqp}).

\begin{lemma}[Sufficient condition for tightness of SDP] \label{lem:dual_cert_rank_n_psd}
Let $\gest \in \bbT_n$ be a global minimizer of \eqref{prog:qcqp}
$\min_{g \in \bbT_n} \lambda g^* L g - 2\real(g^* z),
$ and denote $\gesttil =  \begin{pmatrix} \gest \\  1 \end{pmatrix} $. Under the notation defined earlier, if we have
\begin{align} 
  \lambda \triangle \left( B_n +  2\|\ghat -h\|_\infty \right)+\frac{3-(\delta + \norm{\gest - h}_{\infty})^2}{2 - (\delta + \norm{\gest - h}_{\infty})^2}\Big(\delta + \norm{\gest - h}_{\infty}\Big)^2 < 1, \label{eq:SufficientSDPTight}
\end{align} then $\est{S} = T - \real(\diag(T \gesttil \gesttil^*))$ satisfies $\est{S} \succeq 0$, and $\rank(S) = n$ while the unique solution to \eqref{prog:sdp_relax} reads
\begin{align*}
X = \gesttil \gesttil^* = \begin{pmatrix}
  \gest \gest^{*} \quad & \gest \\ 
  \gest^{*} \quad & 1  
\end{pmatrix}.
\end{align*}
In particular, the condition \eqref{eq:SufficientSDPTight} is satisfied if %
\begin{enumerate}
    \item $\delta + \sqrt{\frac{8}{7}(3\delta +  \lambda\triangle (\smooth_n^2 + \sqrt{2}))} \leq \frac{\sqrt{2}}{3}$, and 
    
    \item $\lambda \triangle \leq \frac{1}{8}$.
\end{enumerate}

\end{lemma}
\begin{proof}
To begin with, note that the first order optimality condition \eqref{eq:first_order_1} states that
\begin{equation*}
    \real( (T \gesttil)_i \gesttil_i^* ) \gesttil_i = (T \gesttil)_i; \ i=1,\dots,n+1,
\end{equation*}
and hence $(T \gesttil)_i \gesttil_i^*$ is real for each $i$. Consequently, we have that $z^* \gest$ and $(\lambda L \gest - z)_i \gest_i^*$ for $i=1,\dots,n$ are real. Thus $\est{S} = T - \diag(T \gesttil \gesttil^*)$. Denoting $D \in \matR^{n \times n}$ to be a real diagonal matrix with $D = \diag(\gest^*\circ (z-\lambda L \gest )) $, one can verify that 
\begin{equation*}
    \est{S} = 
     \begin{pmatrix}
  \lambda L \quad & -z \\ 
  -z^* \quad & 0  
 \end{pmatrix} + 
  \begin{pmatrix}
  D \quad & 0 \\ 
  0 \quad & z^* \gest  
 \end{pmatrix}
=  \begin{pmatrix}
  \lambda L + D \quad & -z \\ 
  -z^* \quad & z^* \gest  
 \end{pmatrix}.
 \end{equation*}
 If $z^* \gest \neq 0$, then from Sylvester's law of inertia (see {\cite[Theorem 4.5.8]{Horn2012}}) we know that $\est{S}$ is \mbox{$^*$-congruent} to the Hermitian block diagonal matrix
 \begin{equation*}
  \est{S}_D =  \begin{pmatrix}
  \lambda L + D - \frac{zz^*}{z^* \gest} \quad & 0 \\ 
  0 \quad & z^* \gest  
 \end{pmatrix}   
 \end{equation*}
which is equivalent to saying that $\est{S}, \est{S}_D$ have the same \emph{inertia}. Since $\gest$ is a second order critical point of (QCQP), we have $z^* \gest \geq 0$ thanks to Corollary \ref{coro:second_order}. Suppose that $z^* \gest > 0$, while we establish below necessary conditions so that this is satisfied. Thus, it follows that $\est{S}$ is rank-$n$ and p.s.d iff the matrix
\begin{equation*}
    M = \lambda L + D - \frac{zz^*}{z^* \gest} \in \mathbb{C}^{n \times n}
\end{equation*}
is p.s.d and has rank $(n-1)$. Hence we will now focus on establishing conditions under which $M \succeq 0$ and $\rank(M) = n-1$. 

To this end, note that since $ \gest_i^*(\lambda L \gest - z)_i$ is real for each $i$, hence
\begin{equation*}
     \gest_i^*(\lambda L \gest - z)_i  = \real( \gest_i^*(\lambda L \gest - z)_i ) = \lambda \real( \gest_i^*( L \gest)_i) - \real( \gest_i^* z_i).
\end{equation*}
Denote $D_1, D_2 \in \mathbb{R}^{n \times n}$ to be diagonal matrices with $(D_1)_{ii} = \real( \gest_i^* z_i)$, and $(D_2)_{ii} = \lambda \real( \gest_i^* ( L \gest)_i)$ for $i=1,\dots,n$. Then we can write $M$ as 
\begin{equation*}
    M = \lambda L + D_1 - D_2  -  \frac{zz^*}{z^* \gest},
\end{equation*}
where we observe that $\gest$ is an eigenvector of $M$ with eigenvalue $0$. Let $u \in \mathbb{C}^n$ be orthogonal to $\gest$, i.e., $u^* \gest = 0$ and $u \neq 0$. We will now establish conditions under which $u^* M u > 0$ which in turn will imply $\rank(M) = n-1$, and $M \succeq 0$. Since $u^* L u \geq 0$ and $u^* \gest = 0$, we arrive at the bound
\begin{align} 
u^* M u 
&\geq u^*D_1 u - u^* D_2 u - \frac{\abs{u^* z}^2}{z^* \gest} \nonumber \\     
&= \sum_{i=1}^n \abs{u_i}^2 \real(\gest_i^*z_i) - \lambda \sum_{i=1}^n \abs{u_i}^2 \real(\gest_i^* (L \gest)_i ) - \frac{\abs{u^* (z-\gest)}^2}{z^* \gest} \nonumber  \\ 
&\geq \sum_{i=1}^n \abs{u_i}^2 \real(\gest_i^* z_i) - \lambda \sum_{i=1}^n \abs{u_i}^2 \real(\gest_i^* (L \gest)_i ) - \frac{\norm{u}_2^2 \norm{z-\gest}_2^2}{z^* \gest}.
\label{eq:quad_bd_1}
\end{align}
Now note that 
\begin{align} \label{eq:zi_gi_bd}
 \real(\gest_i^* z_i) = 1 - \frac{\abs{z_i - \gest_i}^2}{2}  \geq 1 - \frac{(\delta + \norm{\gest - h}_{\infty})^2}{2}; \quad i=1,\dots,n,
\end{align}
where we used the triangle inequality $\abs{z_i - \gest_i} \leq \abs{z_i - h_i} + \abs{h_i - \gest_i}$. Hence $ \real(z_i \gest_i^*) > 0$ for each $i$ if $\frac{(\delta + \norm{\gest - h}_{\infty})^2}{2}  < 1$. Consequently, we have the bounds 
\begin{align}
    z^* \gest &= \real(z^* \gest) = \sum_{i=1}^n \real(z_i^* \gest_i) \geq n\left(1 - \frac{(\delta + \norm{\gest - h}_{\infty})^2}{2} \right), \label{eq:zstar_g_bd} \\
    \norm{z - \gest}_2^2  &= 2(n - \real(z^*\gest)) \leq n(\delta + \norm{\gest - h}_{\infty})^2. \label{eq:z_g_norm_bd}
\end{align}
Finally, for any $i=1,\dots,n$ we can bound the term $\real(\gest_i^*(L \gest)_i)$ in two ways as follows:
\begin{align*} %\label{eq:Lgest_bd}
    \real(\gest_i^*(L \gest)_i) =  \gest_i^* \sum_{j: \set{i,j} \in E} (\gest_i - \gest_j) \leq  \sum_{j: \set{i,j} \in E} \abs{\gest_i - \gest_j} \leq 2 \triangle. 
\end{align*}
and by using the triangle inequality
\begin{align} \label{eq:Lgest_bd}
    \real(\gest_i^*(L \gest)_i) \leq   \|L h \|_{\infty} + \|L(g-h)\|_{\infty} \leq \triangle \left( B_n + 2 \|\ghat -h\|_\infty\right).
\end{align}
Plugging \eqref{eq:zi_gi_bd}, \eqref{eq:zstar_g_bd}, \eqref{eq:z_g_norm_bd}, \eqref{eq:Lgest_bd} in \eqref{eq:quad_bd_1} we arrive at the bound 
\begin{align*} \label{eq:quad_bd_2}
   u^* M u \geq \norm{u}_2^2 \left( 1 - \frac{(\delta + \norm{\gest - h}_{\infty})^2}{2} - 2\lambda \triangle \min\left\{1, \frac{B_n}{2} +  \|\ghat -h\|_\infty \right\} - \frac{(\delta + \norm{\gest - h}_{\infty})^2}{1 - \frac{(\delta + \norm{\gest - h}_{\infty})^2}{2}} \right) 
\end{align*}
One can verify that $u^* M u \geq \norm{u}_2^2/4$ if $\lambda \triangle \leq 1/8$ and $\delta + \norm{\gest - h}_{\infty} \leq \sqrt{2}/3$. Finally, note that the condition $\lambda \triangle \leq 1/8$ satisfies the condition of Theorem \ref{thm:main_linf_qcqp}, and hence, using the bound on $\norm{\gest - h}_{\infty}$ therein, one can verify that 
\begin{align*}
\delta + \norm{\gest - h}_{\infty} \leq \delta + \sqrt{\frac{8}{7}(3\delta +  \lambda\triangle (\smooth_n^2 + \sqrt{2}))}. 
\end{align*}
Therefore $\delta + \norm{\gest - h}_{\infty} \leq \sqrt{2}/3$ holds provided $\delta + \sqrt{\frac{8}{7}(3\delta +  \lambda\triangle (\smooth_n^2 + \sqrt{2}))} \leq \sqrt{2}/3$. This completes the proof.
\end{proof}

\begin{comment}

\begin{remark}
Another approach for obtaining sufficient conditions in order to have a rank one solution could be the following. Let $U(g) = \diag(g)$.
We keep the idea of defining $M$ as 
\begin{align*}
M &= \real\diag(z\ghat^{*}) -\frac{z z^{*}}{z^{*} \ghat}+ \lambda \left\{ \real\diag[U(\ghat^{*}) W U(\ghat) \bm{1}] -W\right\}\\
&= \underbrace{\diag(z\ghat^{*}) -\frac{z z^{*}}{z^{*} \ghat}}_{M_1}+ \underbrace{\lambda \left\{ \diag[U(\ghat^{*}) W U(\ghat) \bm{1}] -W\right\}}_{M_2}
\end{align*}
where the equality is due to Lemma \ref{lemma:Opt} (ii). Notice that $M_1 \ghat =0$ and $M_2 \ghat = 0$. The idea would be to ask that $M_2$ is small.
\end{remark}

\end{comment}

% Simulations
%--------------------------
% Numerical simulations
%--------------------------
\FloatBarrier

\section{Numerical simulations} \label{sec:sims}
As a proof of concept, numerical illustrations of Algorithm~\ref{algo:Main} for denoising and unwrapping mod 1 samples are given firstly on artificial 1D examples, and then, on a 2D problem constructed with real data.
\subsection{1D example}
The output of Algorithm~\ref{algo:Main} is compared with two other methods which are described in more detail in Section~\ref{subsec:existing_work}, namely a trust region subproblem (TRS) and an unconstrained quadratic program (UCQP).
Two example functions are chosen to illustrate unwrapping and denoising on a uniform grid when $d=1$,
\begin{itemize}
    \item \textbf{Example 1}: $f: [0,1]\to \mathbb{R},\quad x\mapsto  \sin(4\pi x)$,
    \item \textbf{Example 2}: $f: [0,1]\to \mathbb{R},\quad x\mapsto 4x\cos(2\pi x)^2 - 2\sin(2\pi x)^2 + 4.7$.
\end{itemize}
%%%%%%%%%% Simple example: Comparison and Denoising Plot  %%%%%%%%%% 
\begin{figure}[h]
\centering
\includegraphics[scale=0.9,trim={1.8cm 1.2cm 1.8cm 0.8cm}, clip]{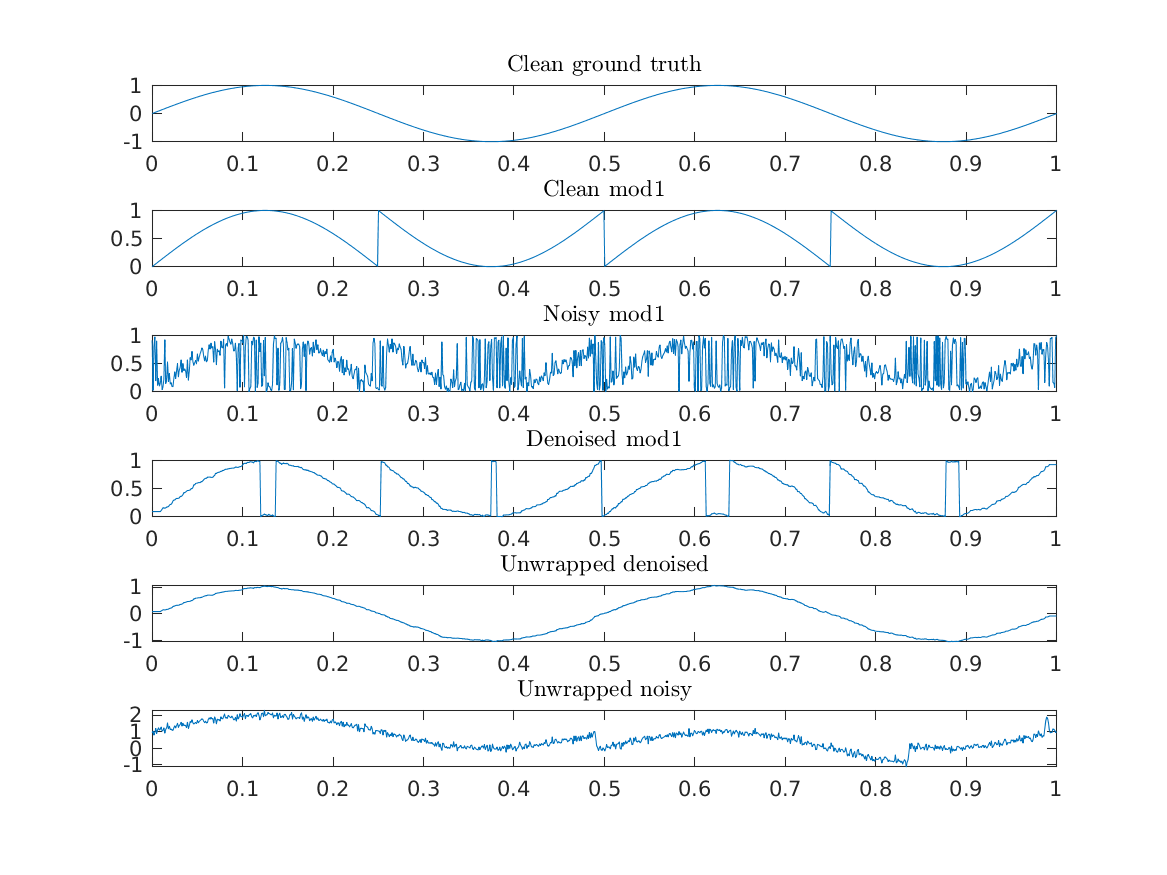}
\caption{kNN denoising and unwrapping for Example 1. Parameters: $n = 10^3$, $C = 0.09$. \label{fig:DenoisingExample_ex2}}
\end{figure}
%%%%%%%%%% Simple example: Error rates %%%%%%%%%% 
\begin{figure}[h!]
\centering
\begin{subfigure}[b]{0.49\textwidth}
% This file was created by matlab2tikz.
%
%The latest updates can be retrieved from
%  http://www.mathworks.com/matlabcentral/fileexchange/22022-matlab2tikz-matlab2tikz
%where you can also make suggestions and rate matlab2tikz.
%
\begin{tikzpicture}

\begin{axis}[%
width=0.7\textwidth,
height=0.7\textwidth,
at={(1.011in,0.669in)},
scale only axis,
xmin=95,
xmax=1005,
xlabel style={font=\color{white!15!black}},
xlabel={$n$},
ymin=0.02,
ymax=0.14,
ylabel style={font=\color{white!15!black}},
ylabel={Wrap around MSE},
axis background/.style={fill=white},
legend style={legend cell align=left, align=left, draw=white!15!black},
yticklabel style={
            /pgf/number format/fixed,
            /pgf/number format/precision=2,
            /pgf/number format/fixed zerofill
        },
]
\addplot [color=blue, mark size=1.0pt, mark=*, mark options={solid, fill=blue, blue}]
  table[row sep=crcr]{%
100	0.0925718421617097\\
150	0.0666903202784698\\
200	0.0614055477140928\\
250	0.052239736960158\\
300	0.0499161481069273\\
350	0.0449007821312493\\
400	0.0445481205327156\\
450	0.0421007914424146\\
500	0.0396485299262736\\
550	0.0391040599718312\\
600	0.0379785881014804\\
650	0.0362984401121319\\
700	0.0360224478321758\\
750	0.0359065512476312\\
800	0.0335156373312016\\
850	0.0330424380448347\\
900	0.0324612812097477\\
950	0.0315556718735505\\
1000	0.0308902558587181\\
};
\addlegendentry{kNN}

\addplot [color=red, mark size=1.0pt, mark=*, mark options={solid, fill=red, red}]
  table[row sep=crcr]{%
100	0.0956110081460097\\
150	0.0689081902871699\\
200	0.0566811730160186\\
250	0.0513980091508275\\
300	0.0484864819564342\\
350	0.0441898206805156\\
400	0.0422293064539145\\
450	0.0413384051065958\\
500	0.0391318052664424\\
550	0.0380620103263706\\
600	0.0373163879107826\\
650	0.0365407300658216\\
700	0.0348928581477444\\
750	0.0350296109242771\\
800	0.0340861636989492\\
850	0.0327360758295533\\
900	0.0329914511850095\\
950	0.0323415407611012\\
1000	0.0320756276768504\\
};
\addlegendentry{TRS}

\addplot [color=green, mark size=1.0pt, mark=*, mark options={solid, fill=green, green}]
  table[row sep=crcr]{%
100	0.084560498288168\\
150	0.0684869326679298\\
200	0.0603860460664958\\
250	0.0527011094723898\\
300	0.0496690488879835\\
350	0.046981235171547\\
400	0.0447003532911585\\
450	0.0421770951933794\\
500	0.0416649083408557\\
550	0.0400620571224048\\
600	0.0396589359808057\\
650	0.0382398232655041\\
700	0.0375294087948371\\
750	0.0370941466833506\\
800	0.0357909676527939\\
850	0.0359525349970529\\
900	0.0354905238589301\\
950	0.0347380640676589\\
1000	0.0344590044549803\\
};
\addlegendentry{UCQP}

\addplot [color=black, mark size=1.3pt, mark=triangle*, mark options={solid, rotate=90, fill=black, black}]
  table[row sep=crcr]{%
100	0.119015981416182\\
150	0.118275131189152\\
200	0.118046780960033\\
250	0.120005577905754\\
300	0.12139307852256\\
350	0.119807531790135\\
400	0.119524120247871\\
450	0.119830155205483\\
500	0.119643734019314\\
550	0.11981637970435\\
600	0.120656956792477\\
650	0.12042636986446\\
700	0.119914536432794\\
750	0.120802034930048\\
800	0.119894286633133\\
850	0.119079943606681\\
900	0.11978345508958\\
950	0.119262945768902\\
1000	0.120207938267943\\
};
\addlegendentry{Noisy}

%%%%%%%%%%% Error bars %%%%%%%%%%%%%%

%%% Error bars: kNN
\addplot [color=blue]
 plot [error bars/.cd, y dir = both, y explicit]
 table[row sep=crcr, y error plus index=2, y error minus index=3]{%
100	0.0925718421617097	0.0114393154785438	0.0114393154785438\\
150	0.0666903202784698	0.0100346056642021	0.0100346056642021\\
200	0.0614055477140928	0.00465619216291088	0.00465619216291088\\
250	0.052239736960158	0.00531234273432094	0.00531234273432094\\
300	0.0499161481069273	0.00456612916746215	0.00456612916746215\\
350	0.0449007821312493	0.00411862183554548	0.00411862183554548\\
400	0.0445481205327156	0.00435066777771866	0.00435066777771866\\
450	0.0421007914424146	0.00341733457171077	0.00341733457171077\\
500	0.0396485299262736	0.00357002136664103	0.00357002136664103\\
550	0.0391040599718312	0.00289385226732931	0.00289385226732931\\
600	0.0379785881014804	0.00279094846953559	0.00279094846953559\\
650	0.0362984401121319	0.00286729488017429	0.00286729488017429\\
700	0.0360224478321758	0.00285103244779339	0.00285103244779339\\
750	0.0359065512476312	0.00268397468242401	0.00268397468242401\\
800	0.0335156373312016	0.00241507362384202	0.00241507362384202\\
850	0.0330424380448347	0.00239667136683642	0.00239667136683642\\
900	0.0324612812097477	0.00191352813377951	0.00191352813377951\\
950	0.0315556718735505	0.00262344353258998	0.00262344353258998\\
1000	0.0308902558587181	0.00234068071788697	0.00234068071788697\\
};

%%% Error bars: TRS

\addplot [color=red]
 plot [error bars/.cd, y dir = both, y explicit]
 table[row sep=crcr, y error plus index=2, y error minus index=3]{%
100	0.0956110081460097	0.0144463881616375	0.0144463881616375\\
150	0.0689081902871699	0.00945406690666466	0.00945406690666466\\
200	0.0566811730160186	0.00570758680751506	0.00570758680751506\\
250	0.0513980091508275	0.00518579704257655	0.00518579704257655\\
300	0.0484864819564342	0.00448291967012536	0.00448291967012536\\
350	0.0441898206805156	0.00309629510206067	0.00309629510206067\\
400	0.0422293064539145	0.00349464693933987	0.00349464693933987\\
450	0.0413384051065958	0.00361606450343815	0.00361606450343815\\
500	0.0391318052664424	0.00296278047059419	0.00296278047059419\\
550	0.0380620103263706	0.00248130603012335	0.00248130603012335\\
600	0.0373163879107826	0.00262323715994008	0.00262323715994008\\
650	0.0365407300658216	0.00191844242839548	0.00191844242839548\\
700	0.0348928581477444	0.00268832691278409	0.00268832691278409\\
750	0.0350296109242771	0.00236035488041561	0.00236035488041561\\
800	0.0340861636989492	0.00247137310760394	0.00247137310760394\\
850	0.0327360758295533	0.00223892480575403	0.00223892480575403\\
900	0.0329914511850095	0.00252309498918289	0.00252309498918289\\
950	0.0323415407611012	0.00233344252775367	0.00233344252775367\\
1000	0.0320756276768504	0.00205727355910629	0.00205727355910629\\
};

%%% Error bars: UCQP

\addplot [color=green]
 plot [error bars/.cd, y dir = both, y explicit]
 table[row sep=crcr, y error plus index=2, y error minus index=3]{%
100	0.084560498288168	0.00824930626130678	0.00824930626130678\\
150	0.0684869326679298	0.00661636015251172	0.00661636015251172\\
200	0.0603860460664958	0.00539639794691362	0.00539639794691362\\
250	0.0527011094723898	0.00488606043067434	0.00488606043067434\\
300	0.0496690488879835	0.0047539880873565	0.0047539880873565\\
350	0.046981235171547	0.00385297101588654	0.00385297101588654\\
400	0.0447003532911585	0.0041877631741323	0.0041877631741323\\
450	0.0421770951933794	0.00316933826225099	0.00316933826225099\\
500	0.0416649083408557	0.00268632680293135	0.00268632680293135\\
550	0.0400620571224048	0.00313955778582778	0.00313955778582778\\
600	0.0396589359808057	0.00262890062623371	0.00262890062623371\\
650	0.0382398232655041	0.00256960045914788	0.00256960045914788\\
700	0.0375294087948371	0.00279390919433899	0.00279390919433899\\
750	0.0370941466833506	0.00304451172942413	0.00304451172942413\\
800	0.0357909676527939	0.0020132618802117	0.0020132618802117\\
850	0.0359525349970529	0.00211032252285168	0.00211032252285168\\
900	0.0354905238589301	0.00256667412049964	0.00256667412049964\\
950	0.0347380640676589	0.00195740184897064	0.00195740184897064\\
1000	0.0344590044549803	0.00179866988457158	0.00179866988457158\\
};
%%% Error bars: Noisy

\addplot [color=black]
 plot [error bars/.cd, y dir = both, y explicit]
 table[row sep=crcr, y error plus index=2, y error minus index=3]{%
100	0.119015981416182	0.00884887785159505	0.00884887785159505\\
150	0.118275131189152	0.00739893011682572	0.00739893011682572\\
200	0.118046780960033	0.00656910114137755	0.00656910114137755\\
250	0.120005577905754	0.0053189410207801	0.0053189410207801\\
300	0.12139307852256	0.00478776108334696	0.00478776108334696\\
350	0.119807531790135	0.00450876798964112	0.00450876798964112\\
400	0.119524120247871	0.00415615060466423	0.00415615060466423\\
450	0.119830155205483	0.0037359653706664	0.0037359653706664\\
500	0.119643734019314	0.0034512685321494	0.0034512685321494\\
550	0.11981637970435	0.00353664134645258	0.00353664134645258\\
600	0.120656956792477	0.0034267834459552	0.0034267834459552\\
650	0.12042636986446	0.00334481813526727	0.00334481813526727\\
700	0.119914536432794	0.00294959540820838	0.00294959540820838\\
750	0.120802034930048	0.00337673778356554	0.00337673778356554\\
800	0.119894286633133	0.00285292280931343	0.00285292280931343\\
850	0.119079943606681	0.00304916111338271	0.00304916111338271\\
900	0.11978345508958	0.00288311731707298	0.00288311731707298\\
950	0.119262945768902	0.00263095005259702	0.00263095005259702\\
1000	0.120207938267943	0.00247812170247027	0.00247812170247027\\
};
\end{axis}
\end{tikzpicture}%
%\caption{wrap around MSE.}
\end{subfigure}\hfill
\begin{subfigure}[b]{0.49\textwidth}
% This file was created by matlab2tikz.
%
%The latest updates can be retrieved from
%  http://www.mathworks.com/matlabcentral/fileexchange/22022-matlab2tikz-matlab2tikz
%where you can also make suggestions and rate matlab2tikz.
%
\begin{tikzpicture}

\begin{axis}[%
width=0.7\textwidth,
height=0.7\textwidth,
at={(1.011in,0.669in)},
scale only axis,
xmin=95,
xmax=1005,
xlabel style={font=\color{white!15!black}},
xlabel={$n$},
ymin=-0.2,
ymax=1.2,
ylabel style={font=\color{white!15!black}},
ylabel={MSE},
axis background/.style={fill=white},
legend style={legend cell align=left, align=left, draw=white!15!black}
]

\addplot [color=blue, mark size=1.0pt, mark=*, mark options={solid, fill=blue, blue}]
  table[row sep=crcr]{%
100	0.335414663627178\\
150	0.178555585069878\\
200	0.0762199837131883\\
250	0.0658917274462335\\
300	0.0517278875055046\\
350	0.0461251318649553\\
400	0.0473771931188364\\
450	0.0438036473514967\\
500	0.0413572200465899\\
550	0.0403455200925965\\
600	0.0397109231371738\\
650	0.0377006652629591\\
700	0.0368673702379985\\
750	0.0367600325587473\\
800	0.0347604655605448\\
850	0.0343140832958685\\
900	0.0335021689378553\\
950	0.0324777380905308\\
1000	0.0321732624558777\\
};
\addlegendentry{kNN}

\addplot [color=red, mark size=1.0pt, mark=*, mark options={solid, fill=red, red}]
  table[row sep=crcr]{%
100	0.495005873795914\\
150	0.236228336459594\\
200	0.0712891177465969\\
250	0.0546177868747506\\
300	0.0504337749333957\\
350	0.0476632080402951\\
400	0.0442228330316765\\
450	0.0428622095290986\\
500	0.0403348715168119\\
550	0.0390627554012755\\
600	0.0392795525469969\\
650	0.0382895396047606\\
700	0.0365288988353365\\
750	0.036523816539696\\
800	0.0359285410093866\\
850	0.0339408866354022\\
900	0.0343189498733724\\
950	0.033451380918684\\
1000	0.0335462164619445\\
};
\addlegendentry{TRS}

\addplot [color=green, mark size=1.0pt, mark=*, mark options={solid, fill=green, green}]
  table[row sep=crcr]{%
100	0.297315394036055\\
150	0.131466863096044\\
200	0.0845366076445992\\
250	0.0548558035341657\\
300	0.0518333458623859\\
350	0.0492740227574149\\
400	0.0459647467882573\\
450	0.0446471755819917\\
500	0.0427916400355689\\
550	0.0418935277511215\\
600	0.0414102772001693\\
650	0.0402900008235282\\
700	0.0392497584976353\\
750	0.0390562789002127\\
800	0.0373535208332508\\
850	0.0375553969185762\\
900	0.037031527090952\\
950	0.035860037009796\\
1000	0.035766938728163\\
};
\addlegendentry{UCQP}

\addplot [color=black, mark size=1.3pt, mark=triangle*, mark options={solid, rotate=90, fill=black, black}]
  table[row sep=crcr]{%
100	0.288248119801725\\
150	0.378923277157153\\
200	0.382104339986587\\
250	0.289160836829882\\
300	0.37128045640465\\
350	0.515416954689873\\
400	0.516677311961897\\
450	0.530348208003375\\
500	0.501255633658429\\
550	0.557983386597323\\
600	0.622014828656269\\
650	0.663285790984183\\
700	0.612457455719728\\
750	0.605852995502545\\
800	0.708099704496549\\
850	0.718072920411894\\
900	0.786384300264658\\
950	0.754401840131202\\
1000	0.597226649709252\\
};
\addlegendentry{Noisy}

%%%%%%% Error bars

%%% error bars: kNN

\addplot [color=blue]
 plot [error bars/.cd, y dir = both, y explicit]
 table[row sep=crcr, y error plus index=2, y error minus index=3]{%
100	0.335414663627178	0.26430802159396	0.26430802159396\\
150	0.178555585069878	0.199430156117186	0.199430156117186\\
200	0.0762199837131883	0.0860488380049692	0.0860488380049692\\
250	0.0658917274462335	0.0879592091576373	0.0879592091576373\\
300	0.0517278875055046	0.0059477519687284	0.0059477519687284\\
350	0.0461251318649553	0.00515919560900999	0.00515919560900999\\
400	0.0473771931188364	0.0067309057063189	0.0067309057063189\\
450	0.0438036473514967	0.00526138796738274	0.00526138796738274\\
500	0.0413572200465899	0.0045721128279558	0.0045721128279558\\
550	0.0403455200925965	0.00354882911857053	0.00354882911857053\\
600	0.0397109231371738	0.0046088614798604	0.0046088614798604\\
650	0.0377006652629591	0.00331250805895449	0.00331250805895449\\
700	0.0368673702379985	0.00309459058297167	0.00309459058297167\\
750	0.0367600325587473	0.0030124949898623	0.0030124949898623\\
800	0.0347604655605448	0.00335161035646106	0.00335161035646106\\
850	0.0343140832958685	0.00310891899693169	0.00310891899693169\\
900	0.0335021689378553	0.00246087474202897	0.00246087474202897\\
950	0.0324777380905308	0.00302900025818255	0.00302900025818255\\
1000	0.0321732624558777	0.00320811365316976	0.00320811365316976\\
};

%%% error bars: TRS
\addplot [color=red]
 plot [error bars/.cd, y dir = both, y explicit]
 table[row sep=crcr, y error plus index=2, y error minus index=3]{%
100	0.495005873795914	0.300565029849951	0.300565029849951\\
150	0.236228336459594	0.239975271497659	0.239975271497659\\
200	0.0712891177465969	0.0665122926069388	0.0665122926069388\\
250	0.0546177868747506	0.00755763230987725	0.00755763230987725\\
300	0.0504337749333957	0.00554194854289963	0.00554194854289963\\
350	0.0476632080402951	0.00559920499580079	0.00559920499580079\\
400	0.0442228330316765	0.00544621506193924	0.00544621506193924\\
450	0.0428622095290986	0.00503794397223175	0.00503794397223175\\
500	0.0403348715168119	0.00329723765415021	0.00329723765415021\\
550	0.0390627554012755	0.00298818207221838	0.00298818207221838\\
600	0.0392795525469969	0.00406772627623179	0.00406772627623179\\
650	0.0382895396047606	0.00352267141020176	0.00352267141020176\\
700	0.0365288988353365	0.00395112858283516	0.00395112858283516\\
750	0.036523816539696	0.00310872195013833	0.00310872195013833\\
800	0.0359285410093866	0.00344128775422996	0.00344128775422996\\
850	0.0339408866354022	0.00288075916374246	0.00288075916374246\\
900	0.0343189498733724	0.00330157922195938	0.00330157922195938\\
950	0.033451380918684	0.00306909703426014	0.00306909703426014\\
1000	0.0335462164619445	0.00285522156930927	0.00285522156930927\\
};

%%% error bars: UCQP

\addplot [color=green]
 plot [error bars/.cd, y dir = both, y explicit]
 table[row sep=crcr, y error plus index=2, y error minus index=3]{%
100	0.297315394036055	0.238692608361066	0.238692608361066\\
150	0.131466863096044	0.165190542090695	0.165190542090695\\
200	0.0845366076445992	0.108837829963133	0.108837829963133\\
250	0.0548558035341657	0.00601259224092583	0.00601259224092583\\
300	0.0518333458623859	0.00652512529916838	0.00652512529916838\\
350	0.0492740227574149	0.00543080437144109	0.00543080437144109\\
400	0.0459647467882573	0.00466823076859921	0.00466823076859921\\
450	0.0446471755819917	0.00445575143181852	0.00445575143181852\\
500	0.0427916400355689	0.00371259790064496	0.00371259790064496\\
550	0.0418935277511215	0.00403376084285229	0.00403376084285229\\
600	0.0414102772001693	0.0041316323731762	0.0041316323731762\\
650	0.0402900008235282	0.00383347528893143	0.00383347528893143\\
700	0.0392497584976353	0.00423020102389422	0.00423020102389422\\
750	0.0390562789002127	0.00422999430804358	0.00422999430804358\\
800	0.0373535208332508	0.00281547222074163	0.00281547222074163\\
850	0.0375553969185762	0.00345087477404495	0.00345087477404495\\
900	0.037031527090952	0.00348653423809483	0.00348653423809483\\
950	0.035860037009796	0.00275561517158382	0.00275561517158382\\
1000	0.035766938728163	0.00291415409310818	0.00291415409310818\\
};

%%% error bars: Noisy

\addplot [color=black]
 plot [error bars/.cd, y dir = both, y explicit]
 table[row sep=crcr, y error plus index=2, y error minus index=3]{%
100	0.288248119801725	0.210337523010702	0.210337523010702\\
150	0.378923277157153	0.266301974870641	0.266301974870641\\
200	0.382104339986587	0.275358942345669	0.275358942345669\\
250	0.289160836829882	0.262964068915543	0.262964068915543\\
300	0.37128045640465	0.28964223066878	0.28964223066878\\
350	0.515416954689873	0.374240757183105	0.374240757183105\\
400	0.516677311961897	0.348375483005141	0.348375483005141\\
450	0.530348208003375	0.389271831791571	0.389271831791571\\
500	0.501255633658429	0.303241220501532	0.303241220501532\\
550	0.557983386597323	0.418631062975315	0.418631062975315\\
600	0.622014828656269	0.38752148440873	0.38752148440873\\
650	0.663285790984183	0.314460834857172	0.314460834857172\\
700	0.612457455719728	0.305369785260269	0.305369785260269\\
750	0.605852995502545	0.422099567179537	0.422099567179537\\
800	0.708099704496549	0.418459788223531	0.418459788223531\\
850	0.718072920411894	0.492265732209777	0.492265732209777\\
900	0.786384300264658	0.409791535656971	0.409791535656971\\
950	0.754401840131202	0.456036954756993	0.456036954756993\\
1000	0.597226649709252	0.306541780111847	0.306541780111847\\
};

\end{axis}
\end{tikzpicture}%
%\caption{MSE.}
\end{subfigure}
\caption{Comparisons between error rates obtained by (kNN), (UCQP) and (TRS) for Example 1, see Figure~\ref{fig:DenoisingExample_ex2}. Error bars are standard deviations over 50 Monte-Carlo runs. Parameters: $C = 0.09$, $\kappa = 0.04$.
\label{fig:DenoisingErrorRates_ex2}}
\end{figure}
Clearly, the modulo $1$ samples of the second example function, given in Figure~\ref{fig:DenoisingExample_ex1}, have a more complex pattern, with respect to the samples of the first example function in Figure~\ref{fig:DenoisingExample_ex2}.
For the unwrapping performance plots, we also compare with the unwrapping performed on the raw data (so without any denoising), which is given at bottom of Figure~\ref{fig:DenoisingExample_ex2} and Figure~\ref{fig:DenoisingExample_ex1}. Namely, the advantage of the denoising procedure before the unwrapping stage is clear by comparing the two last rows of the latter figures. Indeed, unwrapping the noisy $\bmod$ $1$ samples can yield spurious jumps in the recovered function values.
The denoising performance as a function of the number of samples $n$ can be visualized in Figure~\ref{fig:DenoisingErrorRates_ex2} and Figure~\ref{fig:DenoisingErrorRates_ex1}, respectively. These error plots display averages over 50 Monte-Carlo trials for Gaussian noise for (i) recovery of mod 1 samples with respect to the mean square wrap-around error, and (ii) unwrapped samples (after alignment) with respect to the Mean Square Error (MSE). The alignment procedure follows the same methodology as in \cite{CMT18_long}, i.e., it relies on the determination of the mode of a histogram constructed from the distances between the unwrapped and clean samples.
%
%%%%%%%%%% Harder example: Comparison and Denoising Plot
\begin{figure}[h!]
\centering
\includegraphics[scale=0.9,trim={1.8cm 1.2cm 1.8cm 0.8cm}, clip]{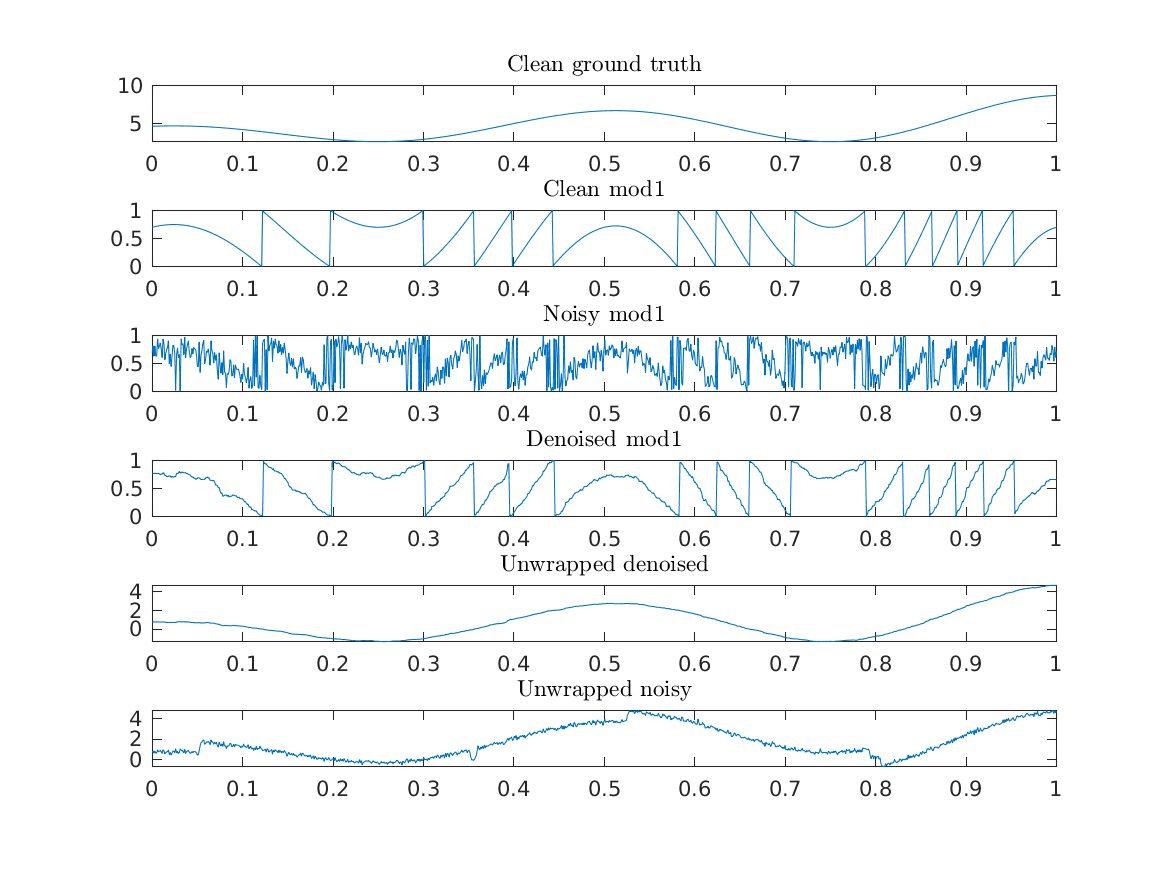}
\caption{kNN denoising and unwrapping for example 2. Parameters: $n = 10^3$, $C = 0.07$. \label{fig:DenoisingExample_ex1}}
\end{figure}
%
%%%%%%%%%% Harder example: Error rates
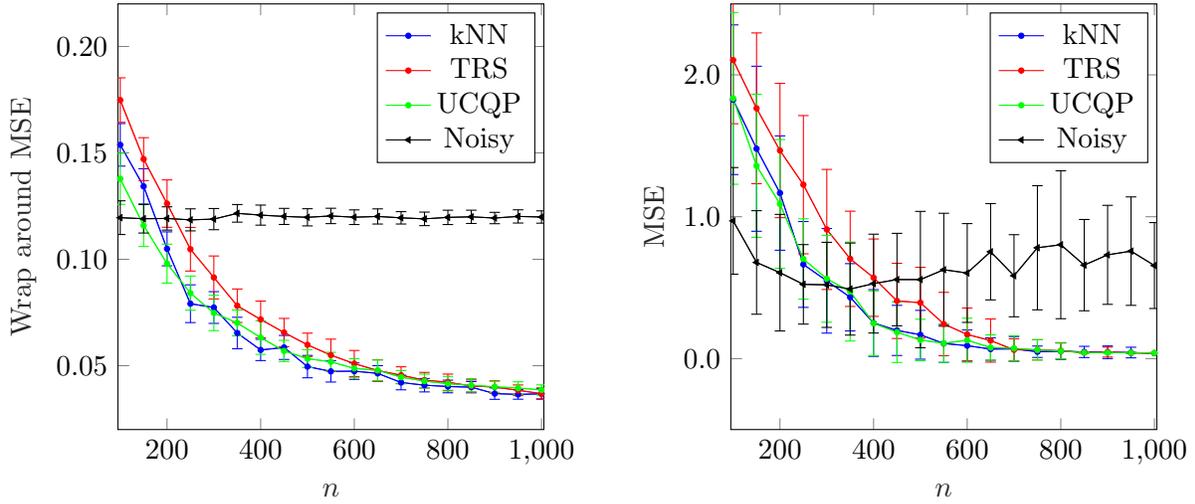
\begin{figure}[h!]
\centering
\begin{subfigure}[b]{0.49\textwidth}
% This file was created by matlab2tikz.
%
%The latest updates can be retrieved from
%  http://www.mathworks.com/matlabcentral/fileexchange/22022-matlab2tikz-matlab2tikz
%where you can also make suggestions and rate matlab2tikz.
%
\begin{tikzpicture}

\begin{axis}[%
width = 0.7\textwidth,
height = 0.7\textwidth,
at={(1.011in,0.669in)},
scale only axis,
xmin=95,
xmax=1005,
xlabel style={font=\color{white!15!black}},
xlabel={$n$},
ymin=0.02,
ymax=0.22,
ylabel style={font=\color{white!15!black}},
ylabel={Wrap around MSE},
axis background/.style={fill=white},
yticklabel style={
            /pgf/number format/fixed,
            /pgf/number format/precision=2,
            /pgf/number format/fixed zerofill
        },
]
\addplot [color=blue, mark size=1.0pt, mark=*, mark options={solid, fill=blue, blue}]
  table[row sep=crcr]{%
100	0.153798154577139\\
150	0.134273114641413\\
200	0.104925820665816\\
250	0.0791305777503625\\
300	0.0773407420773675\\
350	0.0653601188301108\\
400	0.0574002058678111\\
450	0.0585742416004332\\
500	0.0495761491199599\\
550	0.0473377976515018\\
600	0.0474426264232085\\
650	0.0464003961759218\\
700	0.0421128526748917\\
750	0.0408401249504456\\
800	0.0402354850722373\\
850	0.0399020868067104\\
900	0.0368923518552702\\
950	0.0365168311320332\\
1000	0.0367954547714065\\
};
\addlegendentry{kNN}

\addplot [color=red, mark size=1.0pt, mark=*, mark options={solid, fill=red, red}]
  table[row sep=crcr]{%
100	0.174816753586518\\
150	0.147080042408982\\
200	0.12617553881031\\
250	0.104723915117187\\
300	0.0914676590931708\\
350	0.0781839457231029\\
400	0.0717395922975993\\
450	0.0655992339143166\\
500	0.0598207062050119\\
550	0.0549749993125783\\
600	0.0509886401488155\\
650	0.0477433118593411\\
700	0.0455471777625261\\
750	0.0432109654035068\\
800	0.0422342253067438\\
850	0.0404747036373679\\
900	0.0399655036246173\\
950	0.0384364861922008\\
1000	0.0368791172450678\\
};
\addlegendentry{TRS}

\addplot [color=green, mark size=1.0pt, mark=*, mark options={solid, fill=green, green}]
  table[row sep=crcr]{%
100	0.13782062203062\\
150	0.115879274532208\\
200	0.097978602001695\\
250	0.0841014153471732\\
300	0.0747804893884347\\
350	0.0701057492920689\\
400	0.0632093204959428\\
450	0.0569406992612277\\
500	0.0533184959171926\\
550	0.0518325294537198\\
600	0.0487514275767543\\
650	0.0478428787202811\\
700	0.0446148497991723\\
750	0.0427689438649284\\
800	0.0415387341205111\\
850	0.0408155512103672\\
900	0.0400754481883836\\
950	0.0394897285657306\\
1000	0.0386134032690082\\
};
\addlegendentry{UCQP}

\addplot [color=black, mark size=1.3pt, mark=triangle*, mark options={solid, rotate=90, fill=black, black}]
  table[row sep=crcr]{%
100	0.119560980008362\\
150	0.119059709654392\\
200	0.11919017337829\\
250	0.11853116487909\\
300	0.118887911701252\\
350	0.121595367671944\\
400	0.120762718487265\\
450	0.120150481591267\\
500	0.119730880732556\\
550	0.120375893106136\\
600	0.119783624625407\\
650	0.120133463479633\\
700	0.119485054759521\\
750	0.119003355189846\\
800	0.119703054701423\\
850	0.119959301681906\\
900	0.119393653732902\\
950	0.120150266288875\\
1000	0.119878127942255\\
};
\addlegendentry{Noisy}

%%%%%%%%%% Errors kNN 
\addplot [color=blue]
 plot [error bars/.cd, y dir = both, y explicit]
 table[row sep=crcr, y error plus index=2, y error minus index=3]{%
100	0.153798154577139	0.0100058174179975	0.0100058174179975\\
150	0.134273114641413	0.00829496811644118	0.00829496811644118\\
200	0.104925820665816	0.0079657353478482	0.0079657353478482\\
250	0.0791305777503625	0.00887072887238675	0.00887072887238675\\
300	0.0773407420773675	0.00745858571784817	0.00745858571784817\\
350	0.0653601188301108	0.00739957447441304	0.00739957447441304\\
400	0.0574002058678111	0.0049894472524091	0.0049894472524091\\
450	0.0585742416004332	0.00564079634101932	0.00564079634101932\\
500	0.0495761491199599	0.00525566163454884	0.00525566163454884\\
550	0.0473377976515018	0.00500975552888894	0.00500975552888894\\
600	0.0474426264232085	0.00386880113447761	0.00386880113447761\\
650	0.0464003961759218	0.00362121300732072	0.00362121300732072\\
700	0.0421128526748917	0.00343577490180742	0.00343577490180742\\
750	0.0408401249504456	0.00314816003093411	0.00314816003093411\\
800	0.0402354850722373	0.00296521181511721	0.00296521181511721\\
850	0.0399020868067104	0.00265618333039854	0.00265618333039854\\
900	0.0368923518552702	0.00267314471843824	0.00267314471843824\\
950	0.0365168311320332	0.00224905349098874	0.00224905349098874\\
1000	0.0367954547714065	0.00237803787811094	0.00237803787811094\\
};

%%%%%%%%%% Errors TRS 
\addplot [color=red]
 plot [error bars/.cd, y dir = both, y explicit]
 table[row sep=crcr, y error plus index=2, y error minus index=3]{%
100	0.174816753586518	0.0104569029223759	0.0104569029223759\\
150	0.147080042408982	0.0100798100801842	0.0100798100801842\\
200	0.12617553881031	0.0111316385705209	0.0111316385705209\\
250	0.104723915117187	0.0102366966774122	0.0102366966774122\\
300	0.0914676590931708	0.0100964071412464	0.0100964071412464\\
350	0.0781839457231029	0.00787000498918576	0.00787000498918576\\
400	0.0717395922975993	0.00863067352299028	0.00863067352299028\\
450	0.0655992339143166	0.0066237012927385	0.0066237012927385\\
500	0.0598207062050119	0.00548187123464488	0.00548187123464488\\
550	0.0549749993125783	0.00750551029658304	0.00750551029658304\\
600	0.0509886401488155	0.00617568595334232	0.00617568595334232\\
650	0.0477433118593411	0.00503199396292275	0.00503199396292275\\
700	0.0455471777625261	0.00394373478964505	0.00394373478964505\\
750	0.0432109654035068	0.00366012153628352	0.00366012153628352\\
800	0.0422342253067438	0.00384396344469375	0.00384396344469375\\
850	0.0404747036373679	0.00285216103226333	0.00285216103226333\\
900	0.0399655036246173	0.00293246539617604	0.00293246539617604\\
950	0.0384364861922008	0.00251465034776513	0.00251465034776513\\
1000	0.0368791172450678	0.00271854323586807	0.00271854323586807\\
};

%%%%%%%%%% Errors UCQP
\addplot [color=green]
 plot [error bars/.cd, y dir = both, y explicit]
 table[row sep=crcr, y error plus index=2, y error minus index=3]{%
100	0.13782062203062	0.0120041500868268	0.0120041500868268\\
150	0.115879274532208	0.00987944730536392	0.00987944730536392\\
200	0.097978602001695	0.00912948592079228	0.00912948592079228\\
250	0.0841014153471732	0.00802542642018863	0.00802542642018863\\
300	0.0747804893884347	0.00828696654958054	0.00828696654958054\\
350	0.0701057492920689	0.0060685782676137	0.0060685782676137\\
400	0.0632093204959428	0.00787705532882959	0.00787705532882959\\
450	0.0569406992612277	0.0049728485605665	0.0049728485605665\\
500	0.0533184959171926	0.00422795013009393	0.00422795013009393\\
550	0.0518325294537198	0.00443927518796356	0.00443927518796356\\
600	0.0487514275767543	0.00451788576493299	0.00451788576493299\\
650	0.0478428787202811	0.00481526073760379	0.00481526073760379\\
700	0.0446148497991723	0.00311931778422693	0.00311931778422693\\
750	0.0427689438649284	0.00330378545542799	0.00330378545542799\\
800	0.0415387341205111	0.00343246516422998	0.00343246516422998\\
850	0.0408155512103672	0.00310561997180481	0.00310561997180481\\
900	0.0400754481883836	0.0025758260735265	0.0025758260735265\\
950	0.0394897285657306	0.00295871212753987	0.00295871212753987\\
1000	0.0386134032690082	0.00238279701152394	0.00238279701152394\\
};
%%%%%%%%%% Errors Noisy

\addplot [color=black]
 plot [error bars/.cd, y dir = both, y explicit]
 table[row sep=crcr, y error plus index=2, y error minus index=3]{%
100	0.119560980008362	0.00791662163756174	0.00791662163756174\\
150	0.119059709654392	0.00671369999261523	0.00671369999261523\\
200	0.11919017337829	0.00547768162457721	0.00547768162457721\\
250	0.11853116487909	0.00517624849219253	0.00517624849219253\\
300	0.118887911701252	0.00497449863651598	0.00497449863651598\\
350	0.121595367671944	0.0041450386121429	0.0041450386121429\\
400	0.120762718487265	0.00471808696119325	0.00471808696119325\\
450	0.120150481591267	0.00378162431694507	0.00378162431694507\\
500	0.119730880732556	0.00403720109784654	0.00403720109784654\\
550	0.120375893106136	0.00361070689162907	0.00361070689162907\\
600	0.119783624625407	0.00343994118301227	0.00343994118301227\\
650	0.120133463479633	0.00356407953361567	0.00356407953361567\\
700	0.119485054759521	0.00290639284899588	0.00290639284899588\\
750	0.119003355189846	0.0031741628642111	0.0031741628642111\\
800	0.119703054701423	0.00336799419042755	0.00336799419042755\\
850	0.119959301681906	0.0031164367451544	0.0031164367451544\\
900	0.119393653732902	0.00266866217114665	0.00266866217114665\\
950	0.120150266288875	0.00310760737186706	0.00310760737186706\\
1000	0.119878127942255	0.00290146328033767	0.00290146328033767\\
};

\end{axis}
\end{tikzpicture}%
%\caption{wrap around MSE.}
\end{subfigure}\hfill
\begin{subfigure}[b]{0.49\textwidth}
% This file was created by matlab2tikz.
%
%The latest updates can be retrieved from
%  http://www.mathworks.com/matlabcentral/fileexchange/22022-matlab2tikz-matlab2tikz
%where you can also make suggestions and rate matlab2tikz.
%
\begin{tikzpicture}

\begin{axis}[%
width = 0.7\textwidth,
height = 0.7\textwidth,
at={(1.011in,0.669in)},
scale only axis,
xmin=95,
xmax=1005,
xlabel style={font=\color{white!15!black}},
xlabel={$n$},
ymin=-0.5,
ymax=2.5,
ylabel style={font=\color{white!15!black}},
ylabel={MSE},
axis background/.style={fill=white},
yticklabel style={
            /pgf/number format/fixed,
            /pgf/number format/precision=1,
            /pgf/number format/fixed zerofill
        },
]

\addplot [color=blue, mark size=1.0pt, mark=*, mark options={solid, fill=blue, blue}]
  table[row sep=crcr]{%
100	1.82489650910949\\
150	1.47917678926904\\
200	1.16700474298178\\
250	0.66369133883039\\
300	0.548379072283833\\
350	0.432357261668304\\
400	0.251051015164831\\
450	0.19881323714685\\
500	0.168001763465612\\
550	0.106715468794802\\
600	0.091262047847051\\
650	0.0673969270323127\\
700	0.0683840835227182\\
750	0.04834367438932\\
800	0.0526839463516865\\
850	0.0470661710108119\\
900	0.0447036271296892\\
950	0.0428269502173888\\
1000	0.0380884225725691\\
};
\addlegendentry{kNN}

\addplot [color=red, mark size=1.0pt, mark=*, mark options={solid, fill=red, red}]
  table[row sep=crcr]{%
100	2.10389069621728\\
150	1.76423521190525\\
200	1.46671644688293\\
250	1.22600520889835\\
300	0.910523936726631\\
350	0.703461003871666\\
400	0.570451886690433\\
450	0.406620238431925\\
500	0.394199507307344\\
550	0.243944538738523\\
600	0.170274586390196\\
650	0.127352269031124\\
700	0.0622077274275254\\
750	0.0620337401033492\\
800	0.0550700239251039\\
850	0.0420421088635763\\
900	0.0460205229077124\\
950	0.0400266572202735\\
1000	0.0385591898824889\\
};
\addlegendentry{TRS}

\addplot [color=green, mark size=1.0pt, mark=*, mark options={solid, fill=green, green}]
  table[row sep=crcr]{%
100	1.83439205736385\\
150	1.35914038122376\\
200	1.09019422447076\\
250	0.702128525538687\\
300	0.562786937682207\\
350	0.472588486843576\\
400	0.248681627982775\\
450	0.186147211524665\\
500	0.131470648105739\\
550	0.108258844875331\\
600	0.129723168261569\\
650	0.0780879640243185\\
700	0.0718865220291225\\
750	0.061770749626476\\
800	0.0537069856111946\\
850	0.0418320642924385\\
900	0.0417553783715356\\
950	0.0409994996976265\\
1000	0.0395964169265215\\
};
\addlegendentry{UCQP}

\addplot [color=black, mark size=1.3pt, mark=triangle*, mark options={solid, rotate=90, fill=black, black}]
  table[row sep=crcr]{%
100	0.971026761264218\\
150	0.677785029830423\\
200	0.605753669044636\\
250	0.523507986060279\\
300	0.519941102537277\\
350	0.490228356391999\\
400	0.528746032113317\\
450	0.557212798902196\\
500	0.557177293588506\\
550	0.626582876536989\\
600	0.603500584966159\\
650	0.753466521276782\\
700	0.584165008738905\\
750	0.7814153229414\\
800	0.802628960000568\\
850	0.659012954080241\\
900	0.7314665835895\\
950	0.757587334895182\\
1000	0.656011898901198\\
};
\addlegendentry{Noisy}

%%%%%%% error bars %%%%%%%%%%%%%%%%%%

 %% Error kNN
\addplot [color=blue]
 plot [error bars/.cd, y dir = both, y explicit]
 table[row sep=crcr, y error plus index=2, y error minus index=3]{%
100	1.82489650910949	0.527351956169716	0.527351956169716\\
150	1.47917678926904	0.581579974147872	0.581579974147872\\
200	1.16700474298178	0.402939959370551	0.402939959370551\\
250	0.66369133883039	0.302565609432688	0.302565609432688\\
300	0.548379072283833	0.368682959937534	0.368682959937534\\
350	0.432357261668304	0.237091042504067	0.237091042504067\\
400	0.251051015164831	0.2353787978942	0.2353787978942\\
450	0.19881323714685	0.17674840945892	0.17674840945892\\
500	0.168001763465612	0.172083328883309	0.172083328883309\\
550	0.106715468794802	0.133551544579808	0.133551544579808\\
600	0.091262047847051	0.1103602710992	0.1103602710992\\
650	0.0673969270323127	0.0790318830668068	0.0790318830668068\\
700	0.0683840835227182	0.0872671685161345	0.0872671685161345\\
750	0.04834367438932	0.0401632626537744	0.0401632626537744\\
800	0.0526839463516865	0.0585406976500256	0.0585406976500256\\
850	0.0470661710108119	0.0395611627582402	0.0395611627582402\\
900	0.0447036271296892	0.0445395660216431	0.0445395660216431\\
950	0.0428269502173888	0.0374503718456177	0.0374503718456177\\
1000	0.0380884225725691	0.00295089107142471	0.00295089107142471\\
};

%% Error bars TRS
\addplot [color=red]
 plot [error bars/.cd, y dir = both, y explicit]
 table[row sep=crcr, y error plus index=2, y error minus index=3]{%
100	2.10389069621728	0.448200568340739	0.448200568340739\\
150	1.76423521190525	0.53133849632672	0.53133849632672\\
200	1.46671644688293	0.473207660135176	0.473207660135176\\
250	1.22600520889835	0.487150678949561	0.487150678949561\\
300	0.910523936726631	0.423103441443183	0.423103441443183\\
350	0.703461003871666	0.335733573945207	0.335733573945207\\
400	0.570451886690433	0.272689616433689	0.272689616433689\\
450	0.406620238431925	0.264984086886044	0.264984086886044\\
500	0.394199507307344	0.249161647818091	0.249161647818091\\
550	0.243944538738523	0.22309394952311	0.22309394952311\\
600	0.170274586390196	0.186207011204017	0.186207011204017\\
650	0.127352269031124	0.152228313732075	0.152228313732075\\
700	0.0622077274275254	0.0775719026223323	0.0775719026223323\\
750	0.0620337401033492	0.0715660635267171	0.0715660635267171\\
800	0.0550700239251039	0.0571241951079896	0.0571241951079896\\
850	0.0420421088635763	0.00325759747432862	0.00325759747432862\\
900	0.0460205229077124	0.0350772237672648	0.0350772237672648\\
950	0.0400266572202735	0.00387293554949646	0.00387293554949646\\
1000	0.0385591898824889	0.00372954921429038	0.00372954921429038\\
};

%%%%%%% Error bars UCQP
\addplot [color=green]
 plot [error bars/.cd, y dir = both, y explicit]
 table[row sep=crcr, y error plus index=2, y error minus index=3]{%
100	1.83439205736385	0.605548823702429	0.605548823702429\\
150	1.35914038122376	0.504536961087574	0.504536961087574\\
200	1.09019422447076	0.453858906867593	0.453858906867593\\
250	0.702128525538687	0.283309886786492	0.283309886786492\\
300	0.562786937682207	0.306412816080416	0.306412816080416\\
350	0.472588486843576	0.348645436926507	0.348645436926507\\
400	0.248681627982775	0.228711959982663	0.228711959982663\\
450	0.186147211524665	0.213528034058625	0.213528034058625\\
500	0.131470648105739	0.146021385783975	0.146021385783975\\
550	0.108258844875331	0.135527334910791	0.135527334910791\\
600	0.129723168261569	0.155238738858745	0.155238738858745\\
650	0.0780879640243185	0.0891738547427768	0.0891738547427768\\
700	0.0718865220291225	0.0893556427014407	0.0893556427014407\\
750	0.061770749626476	0.071503728862229	0.071503728862229\\
800	0.0537069856111946	0.055815917561368	0.055815917561368\\
850	0.0418320642924385	0.00375915171348126	0.00375915171348126\\
900	0.0417553783715356	0.00380093378357986	0.00380093378357986\\
950	0.0409994996976265	0.00341077082356768	0.00341077082356768\\
1000	0.0395964169265215	0.00319271716840842	0.00319271716840842\\
};

 %% Error Noisy
\addplot [color=black]
 plot [error bars/.cd, y dir = both, y explicit]
 table[row sep=crcr, y error plus index=2, y error minus index=3]{%
100	0.971026761264218	0.375782786292267	0.375782786292267\\
150	0.677785029830423	0.364868876154961	0.364868876154961\\
200	0.605753669044636	0.41089163276404	0.41089163276404\\
250	0.523507986060279	0.280267604175843	0.280267604175843\\
300	0.519941102537277	0.299916200504643	0.299916200504643\\
350	0.490228356391999	0.324416391037901	0.324416391037901\\
400	0.528746032113317	0.34838818783777	0.34838818783777\\
450	0.557212798902196	0.325854587979594	0.325854587979594\\
500	0.557177293588506	0.480508121987767	0.480508121987767\\
550	0.626582876536989	0.398000346408338	0.398000346408338\\
600	0.603500584966159	0.348209180651663	0.348209180651663\\
650	0.753466521276782	0.339496968207407	0.339496968207407\\
700	0.584165008738905	0.289334298897642	0.289334298897642\\
750	0.7814153229414	0.437555244378394	0.437555244378394\\
800	0.802628960000568	0.52127228104616	0.52127228104616\\
850	0.659012954080241	0.321731753007632	0.321731753007632\\
900	0.7314665835895	0.348983444066967	0.348983444066967\\
950	0.757587334895182	0.382224883408305	0.382224883408305\\
1000	0.656011898901198	0.302825488013546	0.302825488013546\\
};

\end{axis}
\end{tikzpicture}%
%\caption{MSE.}
\end{subfigure}
\caption{Comparisons between error rates obtained by (kNN), (UCQP) and (TRS) on the example 2, see Figure~\ref{fig:DenoisingExample_ex1}. Error bars are standard deviations over 50 Monte-Carlo runs. Parameters: $C = 0.07$, $\kappa = 0.04$.\label{fig:DenoisingErrorRates_ex1}}
\end{figure}
The Gaussian noise level in our experiments is taken to be $\sigma=0.12$ while the parameters of the methods are chosen by relying on the statistical results obtained in this work and in the related work \cite{tyagi20} addressing a similar question for (UCQP) and (TRS). 
%\vspace{2cm}

Specifically, the parameters are chosen as follows.
\begin{itemize}
    \item
    For kNN, in view of Corollary~\ref{cor:knn_insample_rates}, the number of neighbours is chosen such that $k = \ceil{k^\star}$ with $k^\star = C n^{\frac{2}{3}}\left( \log n\right)^{\frac{1}{3}}$, where $C>0$ is given hereafter.
    %$k= n^{\frac{2}{3}}\left( \log n\right)^{\frac{1}{3}}\left(\frac{\frac{4\pi^2\sigma^2+2}{3}+\pi\sigma}{\pi \Lip} \right)^{\frac{2}{3}}$.
    \item
    For (UCQP) and (TRS),  the analysis of the corresponding problems by Tyagi \cite{tyagi20} (see Corollary $4$ and Corollary $8$ therein) indicates the choice $\lambda \asymp \left(\sigma^2 n^{10/3}/M^2 \right)^{1/4}$. Hence, for the example of Figure~\ref{fig:DenoisingExample_ex2} and Figure~\ref{fig:DenoisingExample_ex1}, we take $\lambda= \kappa n^{10/12}$ where $\kappa$ is given hereafter. Both methods rely on an appropriate smoothness graph $G = ([n], E)$ which in our experiments is taken to be the path graph where 
$E = \set{\set{i,i+1}: i = 1,\dots,n-1}$.
    %For (SDP), choose $\lambda \leq 1/16$? \HT{I updated the choice of $\lambda$ for (UCQP) and (TRS) after making some corrections in my draft...}
\end{itemize}

In the relatively simple example of Figure~\ref{fig:DenoisingErrorRates_ex2}, one observes that, for the chosen parameters, all methods have a similar performance. For the more complex example corresponding to Figure~\ref{fig:DenoisingErrorRates_ex1}, we observe that (kNN) and (UCQP) yield a slightly smaller error. Notice that the differences between all the methods are not always significant. Also, fine-tuning the parameters for a given $n$ might further improve the results, however all three methods seem to achieve a similar error when $n$ becomes large.

\subsection{2D example}
In order to provide a proof-of-concept in two dimensions, 
we  illustrate the performance of Algorithm~\ref{algo:Main} for denoising and unwrapping mod 1 samples on a 2D grid.
To do so, in Figure~\ref{fig:Vesuvius}, we simulate the reconstruction of the elevation map of Mount Vesuvius from noisy mod 1 samples by following a similar methodology as in~\cite{CMT18_long}.

Firstly, the latitude and longitude grid is rescaled to be a uniform grid in $[0,1]^2$.
Next, the elevation data is scaled down by a factor $500$. This ``change of units'' is necessary to make sure that the clean data is smooth enough so that the unwrapping of the noiseless mod 1 samples match the original noiseless data. \rev{Subsequently, elevation data is corrupted by an additive zero mean Gaussian noise before the modulo $1$ is taken. D}enoising is performed with a kNN estimator. In Figure~\ref{fig:Vesuvius}, the output of  Algorithm~\ref{algo:Main} is also compared to a simple unwrapping of the noisy data using Algorithm~\ref{algo:seq_unwrap_mult}.

The upshot is that, for a large enough noise level, the denoising step in Algorithm~\ref{algo:Main} is indeed a necessary step before applying the unwrapping algorithm.  Namely, spurious jumps are visible in the  plots of Figure~\ref{fig:jump_contour} and Figure~\ref{fig:jump_2D}. Naturally, the denoising procedure also smoothes out the peaks on top of the mount.
%%%%%%%%%% 2D example
\begin{figure}[h!]
\centering
\begin{subfigure}[b]{0.3\textwidth}
\includegraphics[scale=0.29,trim={0cm 0cm 0cm 0.cm}]{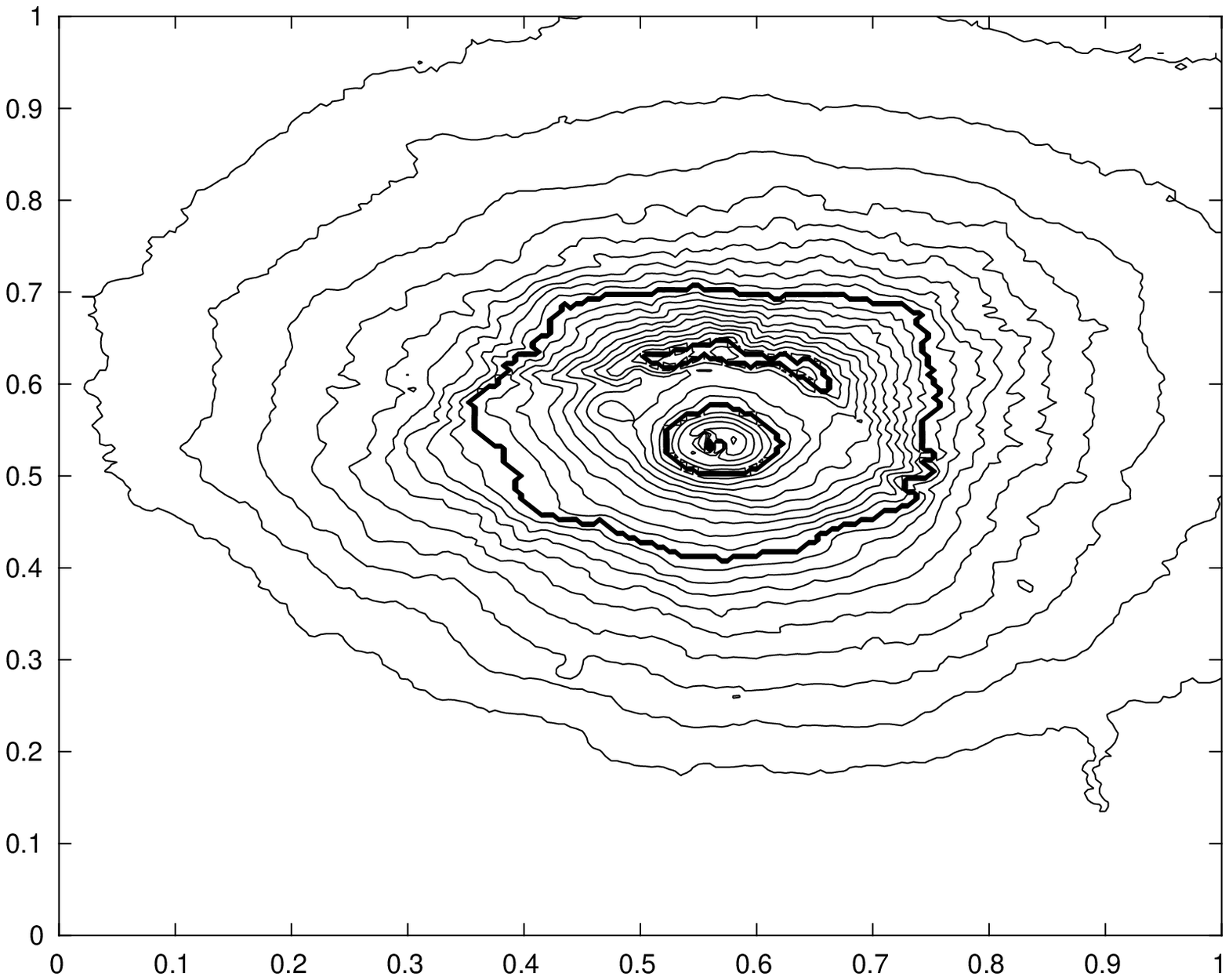}
\caption{Noiseless mod 1 samples.}
\end{subfigure}
\hfill
\begin{subfigure}[b]{0.3\textwidth}
\includegraphics[scale=0.29,trim={0cm 0cm 0cm 0.cm}]{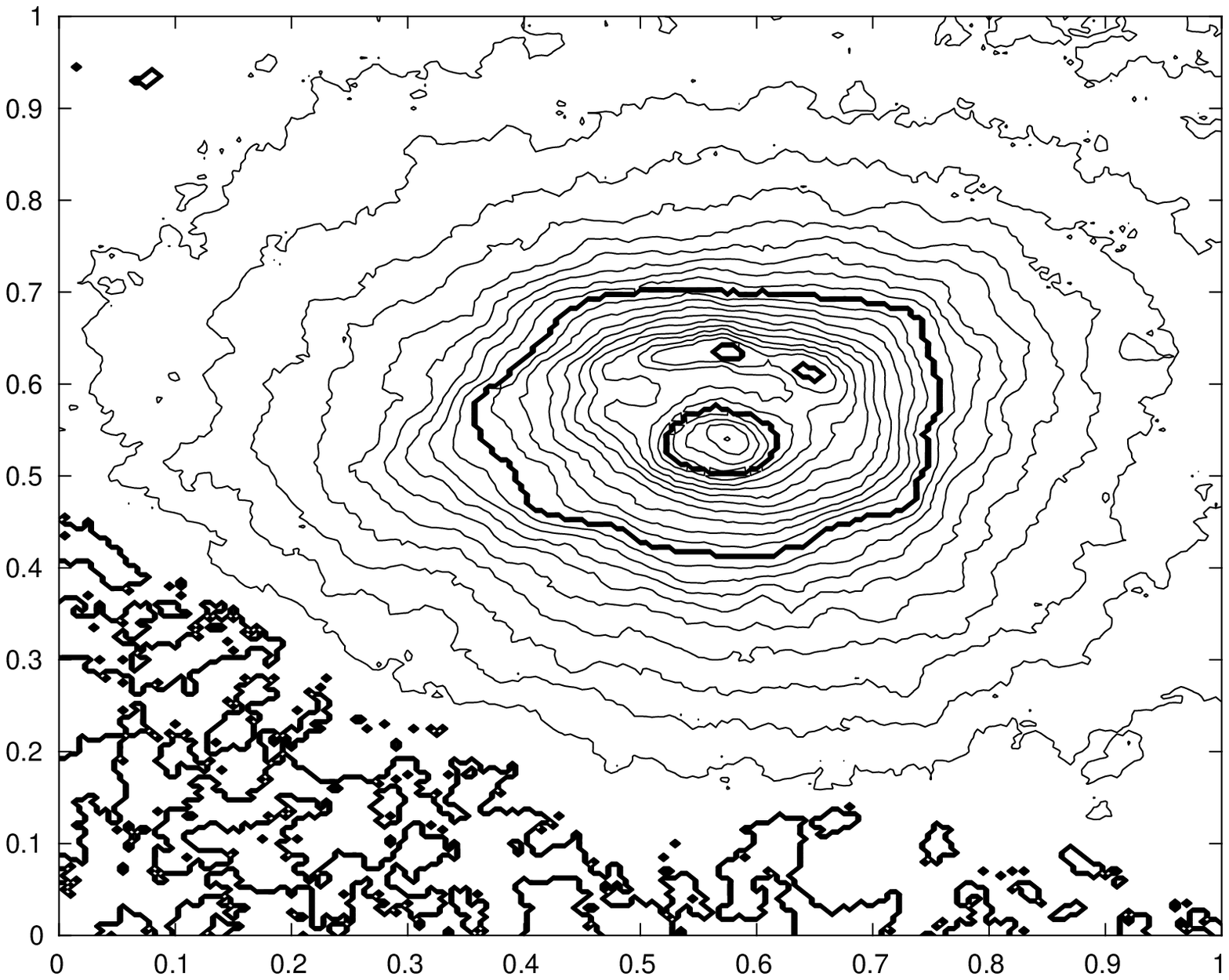}
\caption{Denoised mod 1 samples.}
\end{subfigure}
\hfill
\begin{subfigure}[b]{0.3\textwidth}
\includegraphics[scale=0.29,trim={0cm 0cm 0cm 0.cm}]{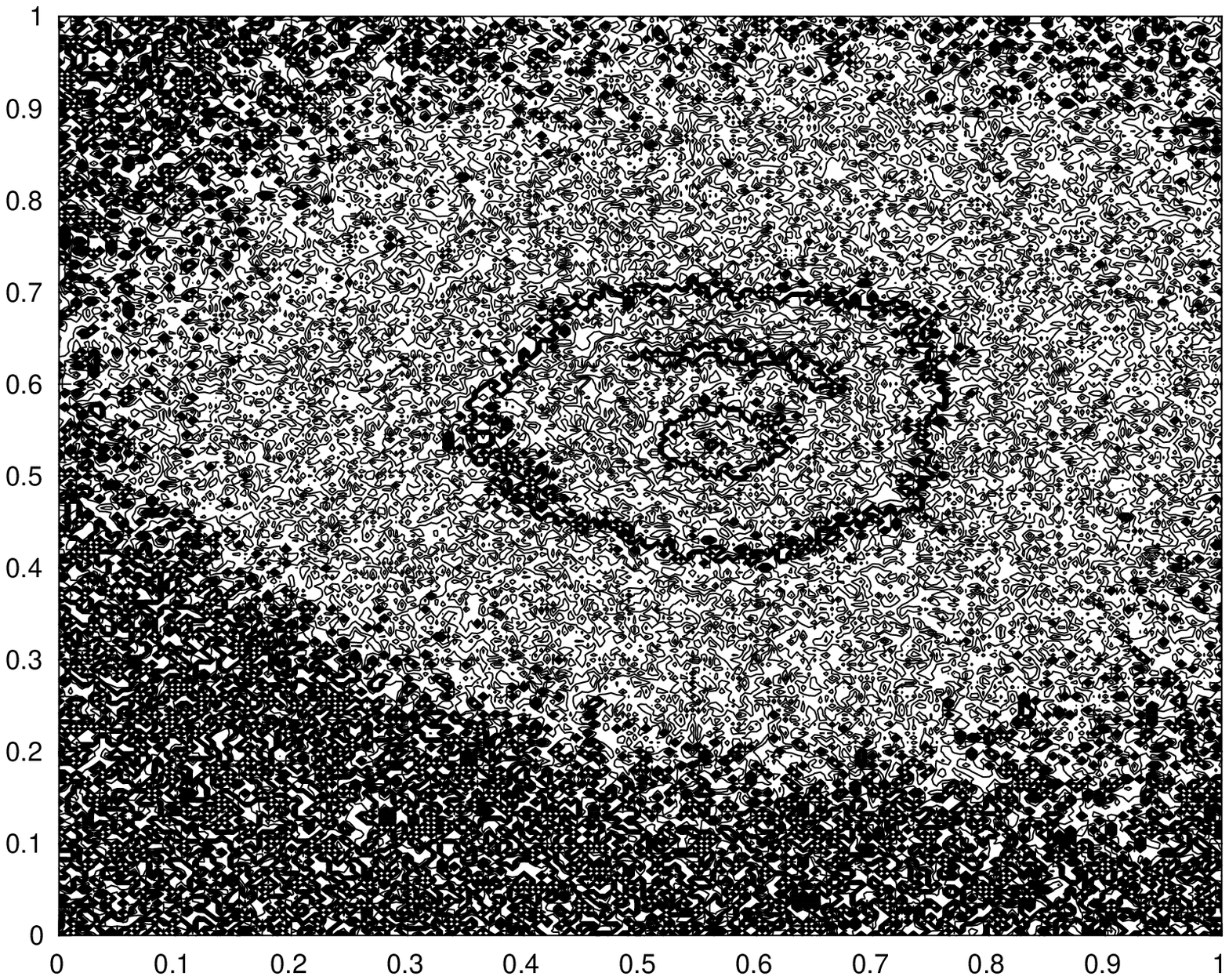}
\caption{Noisy mod 1 samples.}
\end{subfigure}
\begin{subfigure}[b]{0.3\textwidth}
\includegraphics[scale=0.29,trim={0cm 0cm 0cm 0.cm}]{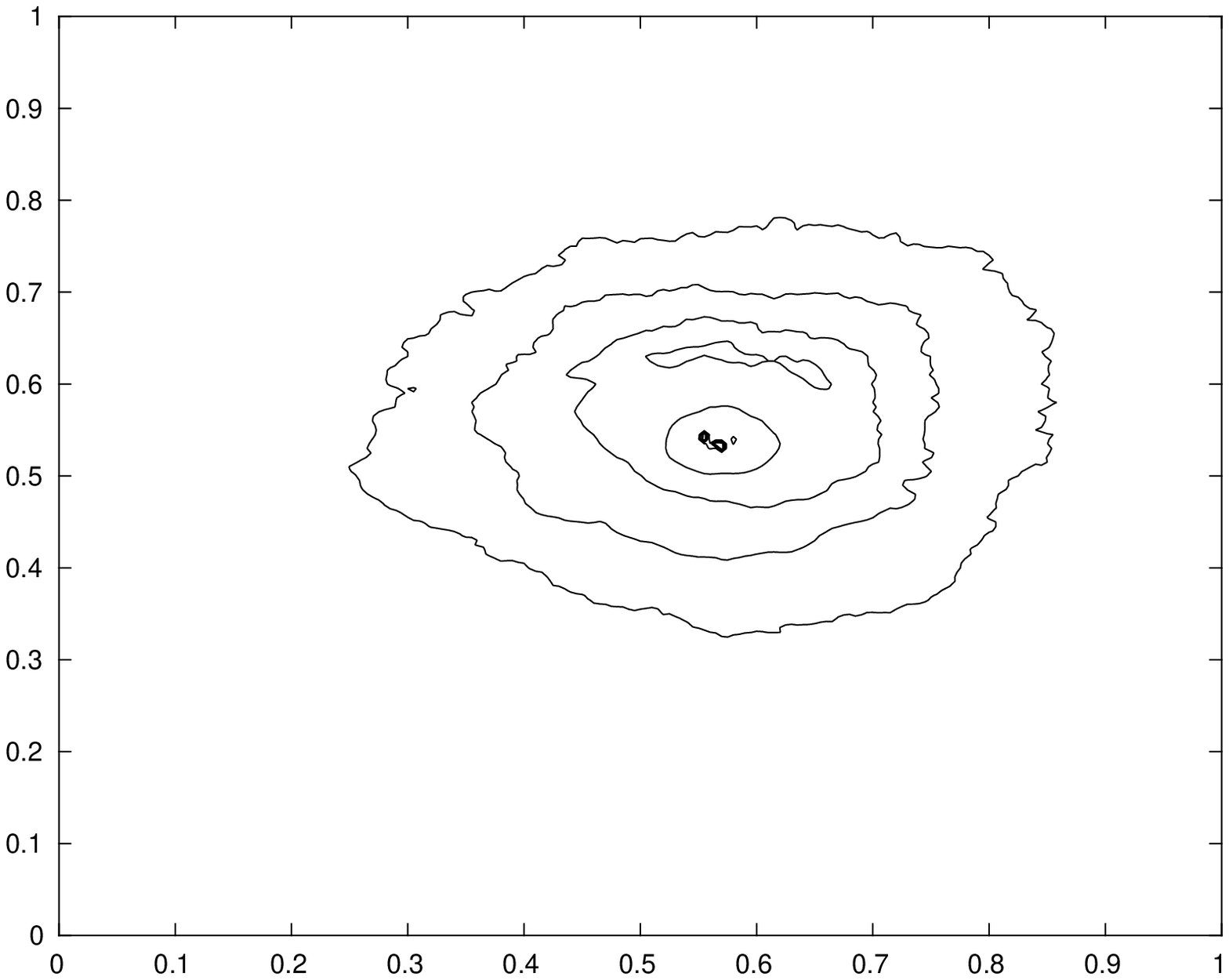}
\caption{Noiseless samples.}
\end{subfigure}
\hfill
\begin{subfigure}[b]{0.3\textwidth}
\includegraphics[scale=0.29,trim={0cm 0cm 0cm 0.cm}]{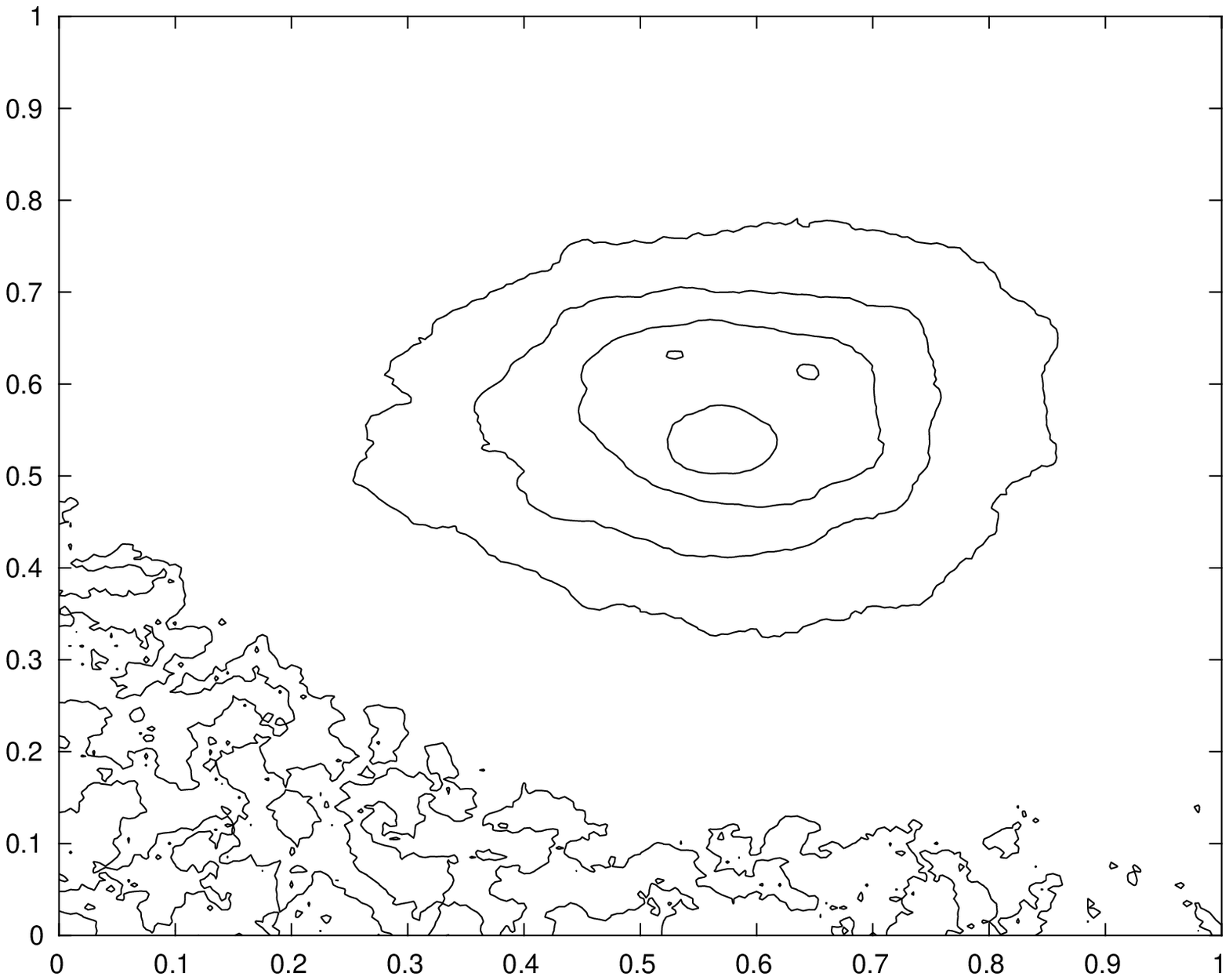}
\caption{Denoised unwrapped samples.}
\end{subfigure}
\hfill
\begin{subfigure}[b]{0.3\textwidth}
\includegraphics[scale=0.29,trim={0cm 0.1cm 0cm 0.1cm}]{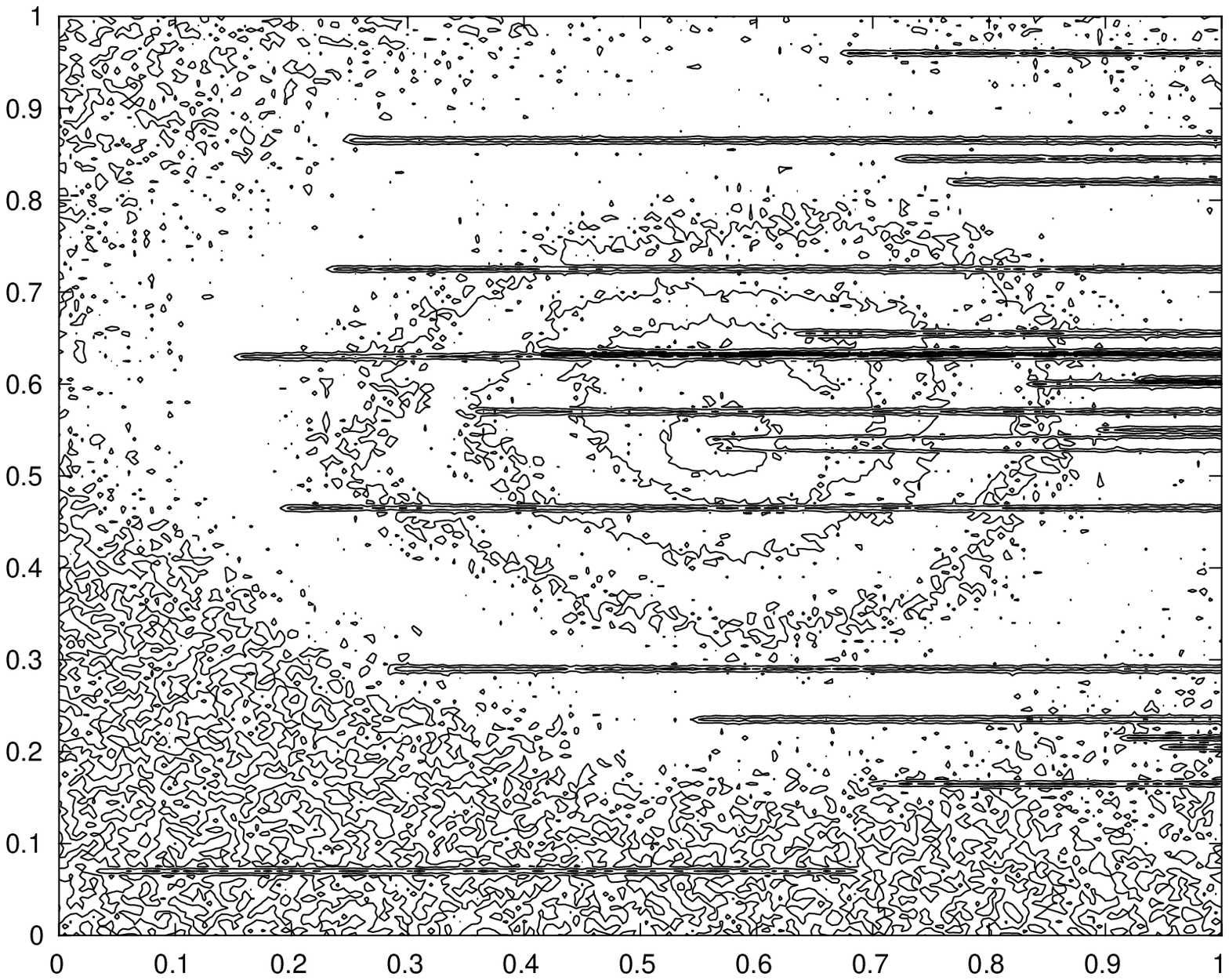}
\caption{Noisy unwrapped samples.\label{fig:jump_contour}}
\end{subfigure}
\begin{subfigure}[b]{0.29\textwidth}
\includegraphics[scale=0.25,trim={0.4cm 0cm 0.4cm 1cm}, clip]{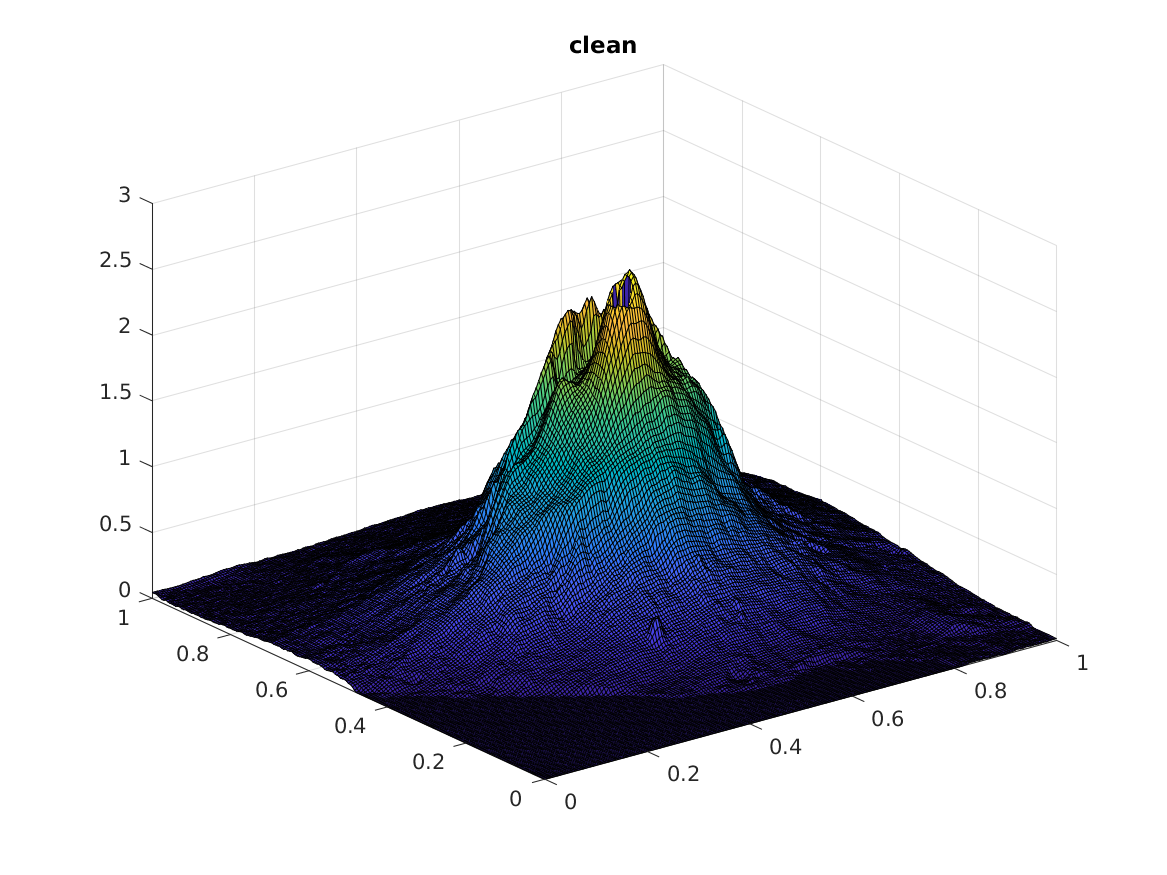}
\caption{Noiseless samples.}
\end{subfigure}\hfill
\begin{subfigure}[b]{0.29\textwidth}
\includegraphics[scale=0.25,trim={0.4cm 0cm 0.4cm 1cm}, clip]{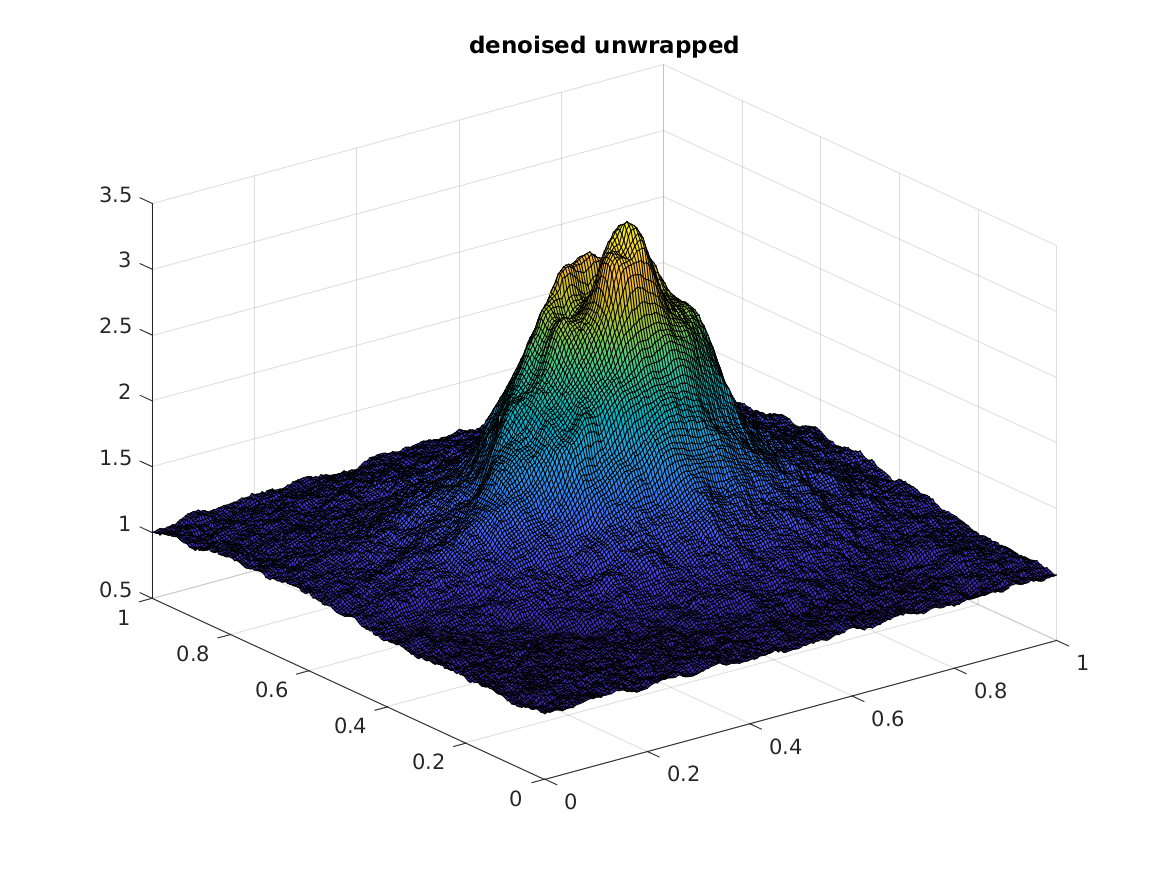}
\caption{Denoised unwrapped.}
\end{subfigure}
\hfill
\begin{subfigure}[b]{0.29\textwidth}
\includegraphics[scale=0.25,trim={0.4cm 0cm 0.4cm 1cm}, clip]{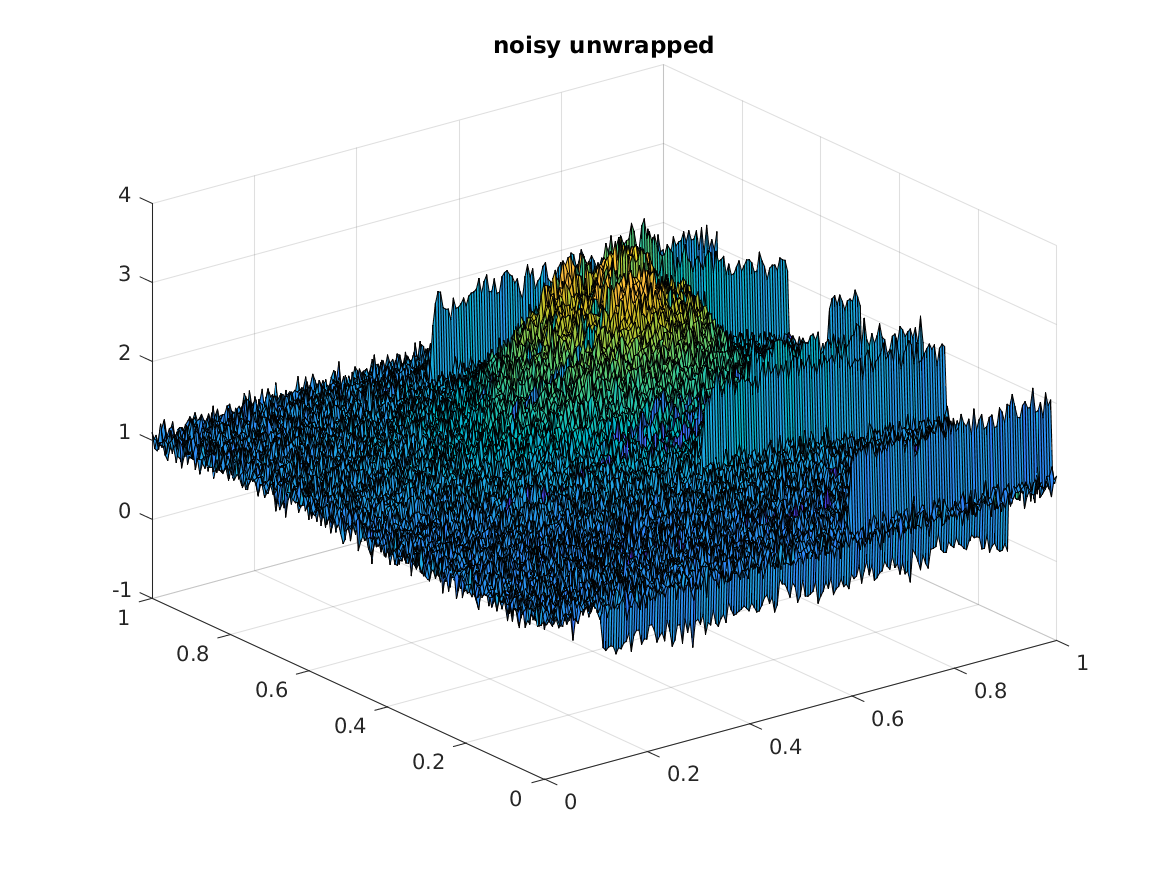}
\caption{Noisy unwrapped.\label{fig:jump_2D}}
\end{subfigure}
\caption{Denoising and unwrapping mod 1 samples obtained from the elevation map of Mount Vesuvius. The first two rows contain contour plots which allow to visualize the smoothness of the samples. Parameters:  $k = 40$ and $\sigma = 0.1$.
The elevation map of Mount Vesuvius (\texttt{N40E014.hgt.zip}) was downloaded thanks to the \texttt{readhgt.m} script written by  Fran\c{c}ois Beauducel from  \protect\url{https://dds.cr.usgs.gov/srtm/version2_1} .\label{fig:Vesuvius}}
\end{figure}
\FloatBarrier

% Discussion
\section{Discussion} \label{sec:discussion}
We start with a detailed overview of related work from the literature and conclude by outlining directions for future work.
\subsection{Related work} \label{subsec:existing_work}
%\HT{Add related work, use material from the other mod 1 paper. Also mention the $\ell_2$ denoising results for TRS, UCQP.}
As discussed in Section~\ref{sec:intro}, the phase unwrapping problem has been studied extensively in the signal processing community with a long history of work. Let us define a ``wrap'' function $w_{\gamma}:\matR \to [-\gamma,\gamma)$ 
\begin{equation*}
 w_{\gamma}(t) := 2\gamma\left(\left[\frac{t}{2\gamma} + \frac{1}{2} \right] - \frac{1}{2} \right)   
\end{equation*}
that outputs centered modulo $2\gamma$ values, with $[a]$ denoting the fractional part of $a \in \matR$. 
%Note that $t \mod (2\gamma) = w_{\gamma} (t)$ so 
\rev{In Appendix~\ref{appsec:corresp_modulo}, we show that 
there is a one-to-one correspondence between the operator $\frac{w_{\gamma}}{2\gamma}$ and the operator $t \mapsto \left(\frac{t}{2\gamma} \right) \mod 1$}. In phase unwrapping $\gamma = \pi$ so that we are given noisy modulo samples $y_i = w_{\pi}(f(x_i) + \eta_i)$ for $1 \leq i \leq n$, with $f:\matR^d \to \matR$ the unknown signal of interest. Denoting $\est{f}_i = f(x_i) + \eta_i$, the classical Itoh's condition \cite{Itoh_82}  for $d=1$ states that if 
\begin{equation*}
    \abs{\est{f}_i - \est{f}_{i-1}} \leq \pi; \quad i = 2,\dots,n,
\end{equation*}
then this implies $\est{f}_i - \est{f}_{i-1} = w_{\pi}(y_i - y_{i-1})$ for all $i$. This suggests that if Itoh's condition holds, then one can recover the samples $\est{f}_i$, up to a global shift of an integer multiple of $2\pi$, in a sequential manner. As mentioned in Remark~\ref{rem:itoh_multiv}, the generalization of this for the case $d = 2$ is known, however we are unaware of a general version of Itoh's condition for the multivariate setting. 

Apart from the natural approach where one denoises the wrapped samples with the hope that the denoised estimates satisfy Itoh's condition, numerous other robust methods have been proposed in the phase unwrapping literature. While the list is too long to review in detail here, we remark that these approaches can be roughly classified as (i) least squares approach (e.g. \cite{Pritt94,Marroquin95}), (ii) branch cut methods \cite{Prati90,Chavez02} and (iii) network flow methods \cite{ching92,takeda96}. The reader is referred to \cite{CMT18_long} as well as the excellent survey by Ying \cite{Ying06} for a more comprehensive discussion about the literature.  One drawback of the phase unwrapping literature is that the methods are typically based on heuristics, and do not, in general, come with theoretical performance guarantees. 

In the past couple of years several new approaches have been proposed for this problem with an emphasis on theoretical guarantees \rev{focusing also on the recovery of $f$}. As detailed below, these approaches typically rely on making certain smoothness assumptions on the underlying $f$.
\begin{enumerate}
\item Bhandari \textit{et al.} \cite{bhandari17} considered a setup where $f$ is a univariate bandlimited function (spectrum lying in $[-\pi,\pi]$), with equispaced \emph{noiseless} modulo samples available via the map $w_{\gamma}$. Their main result was to show that if the sampling width satisfies $T \leq \frac{1}{2\pi e}$, then the samples of $f$, and hence $f$ itself, can be recovered exactly. \rev{The main idea is to use the fact that the modulo operation commutes with the higher order finite difference operator in a certain sense. This is leveraged to shrink the amplitude of the bandlimited signals by taking finite differences of sufficiently high order so that the modulo operation has no effect on the signal. The analysis requires $f$ to be $C^m$ smooth with $m$ sufficiently large so that finite differences of sufficiently large order can be used.} The same authors extended their results to other settings where different assumptions were made on $f$. Specifically, they assume in \cite{sparse_unlimsin18} that $f$ can be represented as the convolution of a sum of $k$ Diracs, while in \cite{sparse_unlim18}, they consider $f$ to be a sum of $k$ sinusoids. In both these papers, they show that if the number of samples is large enough (i.e., $n \gtrsim k$), and $T \leq \frac{1}{2\pi e}$, then $f$ can be recovered exactly.  \rev{In~\cite[Section IV]{bhandari17}, a bounded noise model is also considered. However, in contrast with this paper, an additive noise is added to the modulo samples, while our noise assumption~\eqref{eq:noisy_mod} assumes that the modulo is taken on the noisy signal.}

\item Rudresh \textit{et al.} \cite{rudresh_wavelet} consider $f$ to be a univariate Lipschitz function (with Lipschitz constant $M$) 
with equispaced sampling through the map $w_{\gamma}$. The method proposed therein involves the application of a wavelet filter to the modulo samples, which is then followed by a LASSO type procedure to ultimately recover $f$. Their main result states if $f$ is a polynomial of degree $p$, then a sampling width less than (up to a constant) $\frac{1}{M p}$ suffices for exact recovery of $f$. While $M$ should of course depend on $p$, this was not stated explicitly in \cite{rudresh_wavelet}. While no theoretical results are provided in the presence of noise, they showed their approach to be more robust than that of Bhandari \textit{et al.} \cite{bhandari17} through numerical simulations.

\item The work of Cucuringu and Tyagi \cite{CMT18_long} that we introduced in Section~\ref{subsec:tightness_SDP_intro} essentially focuses on solving \eqref{prog:qcqp} and its relaxations for denoising mod $1$ samples, for the model \eqref{eq:noisy_mod}. Apart from the SDP relaxation discussed eariler, they also considered a ``sphere-relaxation'' of the constraint set leading to a trust region subproblem (TRS)
\begin{equation} 
\min_{\norm{g}_2^2 = n} \norm{g - z}_2^2 + \lambda g^* L g \iff \min_{\norm{g}_2^2 = n} \lambda g^* L g - 2\real(g^* z).  \label{prog:trs} \tag{$\text{TRS}$}  
\end{equation}
The main theoretical results in \cite{CMT18_long} revolve around bounding the error term $\norm{\gest - h}_2$ where $\gest$ is the solution of \eqref{prog:trs} and $h \in \bbT_n$ is the ground truth as defined in Section~\ref{sec:prob_setup}. For instance, when $d = 1$ and $\eta_i \sim \calN(0,\sigma^2)$ i.i.d, they show that if $\lambda \Delta \lesssim 1$ and $\sigma \lesssim 1$, then provided the $x_i$'s form a uniform grid in $[0,1]$, we have\footnote{The result in \cite{CMT18_long} bounds $\norm{\gest - h}_2$ but we can use the inequality in Fact~\ref{Prop:Projection}. Moreover, \cite[Theorem 14]{CMT18_long} has a more complicated statement than what is stated in \eqref{eq:cmt_err_bd}, however it can be verified  that it is of the same order as in \eqref{eq:cmt_err_bd}.} w.h.p
\begin{align} \label{eq:cmt_err_bd} 
    \norm{\frac{\gest}{\abs{\gest}} - h}_2^2 \lesssim \sigma n + \frac{\lambda M^2 \Delta^3}{n}.
\end{align}
where $M$ is the Lipschitz constant of $f$. However, this bound is in general weak due to the fact that w.h.p, $\norm{z-h}_2^2 \asymp \sigma^2 n$ when $\sigma \lesssim 1$. Therefore the bound in \eqref{eq:cmt_err_bd} does not show that $\norm{\frac{\gest}{\abs{\gest}} - h}_2^2 \ll \norm{z-h}_2^2$. 

\item In a parallel work with the present paper, Tyagi \cite{tyagi20} provided an improved $\ell_2$ error analysis for the \eqref{prog:trs} estimator, as well as an unconstrained quadratic program (UCQP) corresponding to the unconstrained relaxation of \eqref{prog:qcqp}  
\begin{equation} 
\min_{g \in \mathbb{C}^n} \norm{g - z}_2^2 + \lambda g^* L g.  \label{prog:ucqp} \tag{$\text{UCQP}$}  
\end{equation}
For both \eqref{prog:trs} and \eqref{prog:ucqp}, $\ell_2$ error bounds are derived for the more general denoising setting where $h \in \bbT_n$ is smooth with respect to an undirected, connected graph $G = ([n], E)$ in the sense that the quadratic variation $h^{*} L h$ is ``small''. The results are also applied to the model \eqref{eq:noisy_mod} when $d=1$ and $\eta_i \sim \calN(0,\sigma^2)$ i.i.d, with the $x_i$'s forming a uniform grid, and $G$ a path graph. For the choice $\lambda \asymp \left(\frac{\sigma^2 n^{10/3}}{M^2} \right)^{1/4}$ it is shown for any fixed $\varepsilon \in (0,1)$ that if $o(1) \leq \sigma \lesssim 1$ for $n$ large enough, then the solutions $\gest$ of \eqref{prog:ucqp} and \eqref{prog:trs} satisfy (w.h.p) 
$$\norm{\frac{\gest}{\abs{\gest}} - h}_2^2 \leq \varepsilon \norm{z - h}_2^2.$$ 
\end{enumerate}
The above results are in the nonparametric setting where $f$ is typically highly non-linear. However the setting where $f$ is  linear has also been considered recently. For \rev{instance}, Shah and Hegde \cite{shah2018signal} assume $f$ to be a sparse linear function, and provide conditions for exact recovery of $f$ in the noiseless setting (in the regime $n \ll d$). This is accomplished via an alternating minimization based algorithm. Musa \textit{et al.} \cite{MusaJG18} also consider $f$ to be a sparse linear function, but assume that it is generated from a Bernoulli-Gaussian distribution. The recovery of $f$ is achieved via a generalized approximate message passing algorithm, but no theoretical analysis is provided.

%---------------------
% Future directions
%---------------------
\subsection{Future directions} \label{subsec:fut_work}
An important direction for future work is to improve our analysis for the tightness of the SDP estimator. As discussed in Remark~\ref{rem:qcqp_linf_bd}, the main bottleneck of our analysis is in the $\ell_{\infty}$ error bound for the solution $\gest$ of \eqref{prog:qcqp} (in Theorem~\ref{thm:main_linf_qcqp}) which is admittedly not satisfactory. Deriving bounds satisfying the property $\norm{\gest - h}_{\infty} \ll \norm{z - h}_{\infty}$ is an important question in its own right, and the analysis for the same should utilize information about $\gest$ available through its first and second order optimality conditions. Moreover, we expect to see an ``optimal choice'' of the regularizer $\lambda$ -- similar to the aforementioned $\ell_2$ error analysis for \eqref{prog:trs}, \eqref{prog:ucqp} derived in \cite{tyagi20} -- which minimizes the bound on $\norm{\gest - h}_{\infty}$. Such an analysis is typically facilitated in a random noise model where $z_i = h_i \exp(\rmi 2\pi \eta_i)$, with $\eta_i \sim \calN(0,\sigma^2)$ i.i.d Gaussian for each $i$. 
\subsection*{Acknowledgments}
EU: The research leading to these results has received funding from the European Research Council under the European Union's Horizon 2020 research and innovation program / ERC Advanced Grant E-DUALITY (787960). This paper reflects only the authors' views and the Union is not liable for any use that may be made of the contained information.
Research Council KUL: Optimization frameworks for deep kernel machines C14/18/068. Flemish Government: FWO: projects: GOA4917N (Deep Restricted Kernel Machines: Methods and Foundations), PhD/Postdoc grant. This research received funding from the Flemish Government (AI Research Program). Ford KU Leuven Research Alliance Project KUL0076 (Stability analysis and performance improvement of deep reinforcement learning algorithms).

% Bibliography
%\newpage
\bibliographystyle{plain}
\bibliography{references}

% Appendix
\appendix
\section{Technical results}
The following technical results are instrumental for proving our main results.
\begin{fact}[\cite{LiuPhaseSynchronization}]\label{Prop:Projection}
Let $q\geq 1$, $z\in \bbT_n$ and $w\in \mathbb{C}_n$. Then, it holds
$\| \frac{w}{|w|} -z \|_q \leq 2 \| w - z\|_q.$
\end{fact}
Fact \ref{Prop:Projection} indeed means that the distance between $z\in \bbT_n$ and $w\in \mathbb{C}_n$ after projection on $\bbT_n$ can be upper bounded by the distance before projection, up to a scalar factor. It is proved in \cite{LiuPhaseSynchronization}.

The following result relates the distance between two points $u, v \in \mathbb{T}_1$ with their arguments.
\begin{fact}\label{Fact:Wrap-around}
Let $u = \exp(2\pi \rmi f)$ and $v =  \exp(2\pi \rmi \widehat{f})$
For $0 \leq \epsilon \leq 2$, let $|v -u|\leq \epsilon$. Then, we have
\[
d_w(\widehat{f} \bmod 1,f \bmod 1) \leq \frac{\epsilon}{4} .
\]
\end{fact}
\begin{proof}
The proof follows the same lines as in \cite{CMT18_long}.
We know that
\begin{align*}
|u-v| &= |\exp(2\pi \rmi f)- \exp(2\pi \rmi \widehat{f})|\\
&= |1-\exp(2\pi \rmi (\widehat{f}\bmod 1-f\bmod 1))|\\
&= 2|\sin(\pi(\widehat{f}\bmod 1-f \bmod 1))|\\
&= 2\sin(\pi|\widehat{f}\bmod 1-f\bmod 1|)\\
&= 2\sin(\pi(1-|\widehat{f}\bmod 1-f\bmod 1|))\\
&= 2\sin(\pi d_w(\widehat{f}\bmod 1, f\bmod 1))
\end{align*}
where the second and third last equalities follow from $\widehat{f}\bmod 1-f\bmod 1\in [-1,1]$. Clearly, \[
\min\{|\widehat{f}\bmod 1-f\bmod 1| , 1-|\widehat{f}\bmod 1-f\bmod 1| \} \leq \frac{1}{\pi}\arcsin(\epsilon/2).
\]
Notice that $\arcsin (x)$ is a convex function on $[0,1]$  with $\arcsin(0) = 0$ and $\arcsin(1) = \pi/2$. The upshot is that $\arcsin(x)\leq \pi x/2$ for all $x\in [0,1]$. Hence, we find
\[
d_w(\widehat{f}\bmod 1, f\bmod 1) \leq \frac{1}{4}\epsilon.
\]
\end{proof}
Finally, we recall the well known Bernstein's concentration inequality for sums of independent random variables.
\begin{theorem}[Bernstein inequality for bounded random variables\rev{; see Corollary 7.31 in \cite{RauhutBook}}]\label{thm:Bernstein}
Let $X_1,\dots ,X_M$ be independent random variables with zero mean such that $|X_\ell|\leq K$ almost surely for $\ell\in [M]$ and some constant $K>0$. Furthermore, assume $\mathbb{E}|X_\ell|^2 \leq \sigma_\ell^2$ for constants $\sigma_\ell >0$, $\ell\in [M]$. Then, for all $t>0$,
\[
\Pr \left( \left|\sum_{\ell=1}^M X_\ell \right|\geq t \right)\leq 2\exp\left(-\frac{t^2/2}{\sigma^2 +Kt/3} \right),
\]
where $\sigma^2 = \sum_{\ell = 1}^M \sigma_\ell^2$.

\end{theorem}

\section{Correspondence between mod $1$ samples and folding with a centered modulo} \label{appsec:corresp_modulo}
\rev{As explained in the introduction, a self-reset ADC introduces discontinuities in the signals so that its range is an interval $[-\gamma, \gamma)$ where $\gamma>0$ is the so-called as the ADC threshold.
Following~\cite{bhandari17}, the folding  is performed by the hardware and is modeled by the following centered modulo operator 
$$
w_\gamma: \mathbb{R}\to [-\gamma, \gamma) \text{ such that } t\mapsto 2\gamma \left(\left[ \frac{t}{2\gamma} +\frac{1}{2}\right]-\frac{1}{2}\right),$$
where $[t] = t - \lfloor t \rfloor$ denotes the fractional part. Note that the range of the centered modulo is a half-open interval. We now show the correspondence between this centered modulo operator and the noise model~\eqref{eq:noisy_mod} which is studied in this paper.
By a case-by-case analysis, it is straightforward to check that 
$
w_\gamma(t) = 2\gamma \calH\left(\frac{t}{2\gamma} \bmod 1 \right) 
$
where $\calH$ is the following discontinuity
\begin{equation*}
    \calH: \left[0, 1\right)\to \left[-1/2,  1/2\right) \text{ such that }  x \mapsto \begin{cases}
    x &\text{ if } x < 1/2\\
    x-1 &\text{ if } x \geq  1/2
                \end{cases}.
\end{equation*}
This function admits an inverse which is given by
\begin{equation*}
    \calH^{-1}:\left[-1/2,  1/2\right) \to \left[0, 1\right) \text{ such that }  x \mapsto \begin{cases}
    x + 1 &\text{ if } x < 0\\
    x &\text{ if } x \geq  0
                \end{cases}.
\end{equation*}
Therefore, mod 1 samples are in one-to-one correspondence with the centered modulo, since
$$
\frac{w_\gamma(t)}{2\gamma} = \calH\left(\frac{t}{2\gamma} \bmod 1 \right).
$$
Let us analyse the role of $\gamma$.
For simplicity, we discuss the case $d = 1$ and consider noisy samples $f(x_i) + \eta_i$  where $f$ is $M$-Lipschitz and $\eta_i \sim \mathcal{N}(0,\sigma^2)$ i.i.d for $i=1, \dots, n$. Then, the rescaled samples can be written as
$$
\frac{f(x_i) + \eta_i}{2\gamma} = f^{(\gamma)} (x_i) + \eta_i^{(\gamma)},
$$
 where $f^{(\gamma)}$ is  $\frac{M}{2\gamma}$-Lipschitz and   $\eta_i^{(\gamma)}\sim \mathcal{N}(0,\left(\frac{\sigma}{2\gamma}\right)^2)$ \rev{i.i.d} with $i=1, \dots, n$. In other words, the effect of $\gamma$ is to rescale the Lipschitz constant and the noise variance. Hence, the noise model~\eqref{eq:noisy_mod} and the settings of our analysis applies to signals obtained in the context of a self-reset ADC.}

%\rev{@Hemant: In view of the Unlimited sampling paper, I wonder if the order 'sampling $\to$ modulo' or ' modulo $\to$ sampling' matters or not.}

\section{Optimality conditions of QCQP}
Denote the objective of the QCQP by 
\[
F(g) = \lambda g^* L g - 2\real(g^* z).
\]
We recall the covariant derivative is
\[
\grad F(g) = \proj_g \nabla F(g) = 2\left\{ \rev{\lambda}Lg -z -\real \diag\left( (\rev{\lambda}Lg-z)g^* \right)\right\}  
\]
where the projection on the tangent space at $g\in \bbT_n$ is defined in \eqref{eq:proj}. The first order necessary optimality condition in then indeed $\grad F(\widehat{g})=0$. The second order necessary condition involves the Hessian as follows
\[
\langle \gdot , \Hess F(g) [\gdot]\rangle \geq 0 \text{ for all } \gdot\in T_g\bbT_n,
\]
where $\Hess F(g)[\gdot] = \proj_g D \grad F(g)[\gdot ]$ with $D$ denoting the directional derivative.
\begin{proposition} \label{prop:SecondOrder}
We have
$
\Hess F(g)[\gdot ] = 2\left\{ \rev{\lambda} L \gdot -\real[\diag\left((\rev{\lambda} Lg-z)g^*\right)] \gdot \right\}
$
for all $\gdot\in T_g\bbT_n$.
\end{proposition}
\begin{proof}
We have simply
\[
\grad F(g+t\gdot) = 2\Big\{ \rev{\lambda} L(g+t\gdot) -z -\real \diag\big[ (\rev{\lambda} L(g+t\gdot)-z)(g+t\gdot)^* \big] \rev{(g+t\gdot)}\Big\}.
\]
Then, by differentiating with respect to $t$ and evaluating the result at $t=0$, we find
\begin{align*}
D\grad F(g)[\gdot ] &= \left[\frac{\rmd}{\rmd t}\grad F(g+t\gdot)\right]_{t=0}\\
& = 2\left\{ \rev{\lambda} L\gdot -\real[\diag(\rev{\lambda} L\gdot g^*)]g -\real[\diag((\rev{\lambda} Lg -z) \gdot^*)]g -\real[\diag((\rev{\lambda} Lg -z) g^*)]\gdot  \right\}.
\end{align*}
The final result follows by projecting on the tangent space to $g$ and by noticing that $\proj_g (D g) = 0$ and $\proj_g (D\gdot) = D\gdot$, where $D$ is a real diagonal matrix.

\end{proof}

\end{document}